\documentclass[10pt, a4paper, oneside, reqno]{amsart}

\makeatletter
\def\@settitle{\begin{center}%
		\baselineskip14\p@\relax
		\normalfont\LARGE\scshape\bfseries
		\@title
	\end{center}%
}
\makeatother

\makeatletter

\def\subsection{\@startsection{subsection}{2}%
	\z@{.5\linespacing\@plus.7\linespacing}{.5\linespacing}%
	{\normalfont\large\bfseries}}

\def\subsubsection{\@startsection{subsubsection}{3}%
	\z@{.5\linespacing\@plus.7\linespacing}{.5\linespacing}%
	{\normalfont\itshape}}

\usepackage[usenames, dvipsnames]{color}
\definecolor{darkblue}{rgb}{0.0, 0.0, 0.45}

\usepackage[colorlinks	= true,
			raiselinks	= true,
			linkcolor	= darkblue, 
			citecolor	= Mahogany,
			urlcolor	= ForestGreen,
			pdfauthor	= {Peyman Mohajerin Esfahani},
			pdftitle	= {},
			pdfkeywords	= {},
			pdfsubject	= {},
			plainpages	= false]{hyperref}

\usepackage{dsfont,amssymb,amsmath,subfigure, graphicx,enumitem} 
\usepackage{amsfonts,dsfont,mathtools, mathrsfs,amsthm} 
\usepackage[amssymb, thickqspace]{SIunits}
\usepackage{algorithm,algorithmicx,fancyhdr}

\allowdisplaybreaks
\date{\today}
\theoremstyle{theorem}
\newtheorem{Thm}{Theorem}[section]
\newtheorem{Prop}[Thm]{Proposition}

\newtheorem{Lem}[Thm]{Lemma}
\newtheorem{Cor}[Thm]{Corollary}
\newtheorem{As}[Thm]{Assumption}

\newtheorem{Def}[Thm]{Definition}

\newtheorem{Rem}[Thm]{Remark}

\theoremstyle{remark}
\newtheorem{Ex}[Thm]{Example}

\newcommand{\key}[1]{\textbf{Keywords.} #1}

\newcommand{\Min}{\min\limits_}
\newcommand{\Sup}{\sup\limits_}
\newcommand{\Inf}{\inf\limits_}

\newcommand{\PP}{\mathds{P}}
\newcommand{\EE}{\mathds{E}}
\newcommand{\R}{\mathbb{R}}

\newcommand{\N}{\mathbb{N}}
\newcommand{\borel}{\mathfrak{B}}

\newcommand{\Llra}{\Longleftrightarrow}
\newcommand{\ra}{\rightarrow}

\newcommand{\ind}{\mathds{1}}
\newcommand{\Let}{\coloneqq}

\newcommand{\diff}{\mathrm{d}}

\newcommand{\wt}{\widetilde}

\newcommand{\opt}{^\star}
\newcommand{\ball}[2]{\mathsf{B}_{#2}(#1)}		

\newcommand{\eps}{\varepsilon}
\DeclareMathOperator{\e}{e}

\newcommand{\X}{\mathbb{X}}
\newcommand{\C}{\mathbb{C}}
\newcommand{\B}{\mathbb{B}}
\newcommand{\Y}{\mathbb{Y}}
\newcommand{\U}{\mathbb{U}}
\newcommand{\conv}{\mathsf{conv}}

\newcommand{\Lip}{\mathscr{L}}
\newcommand{\dir}[1]{\delta_{#1}}

\newcommand{\inner}[2]{\big \langle #1, #2 \big \rangle}

\newcommand{\set}[1]{\mathbb{#1}}
\DeclareMathOperator{\st}{s.t.}

\newcommand{\Prim}{\mathcal{P}}
\newcommand{\Dual}{\mathcal{D}}
\newcommand{\op}{\mathcal{A}}
\newcommand{\opn}{\mathcal{A}_{n}}
\newcommand{\Jp}{J}
\newcommand{\Jpn}{J_n}
\newcommand{\Jpnd}{J_n(\delta)}
\newcommand{\JpnN}{J_{n,N}}
\newcommand{\Jd}{\wt{J}}
\newcommand{\Jdn}{\wt{J}_n}
\newcommand{\Jdnd}{\wt{J}_n(\delta)}
\newcommand{\cone}{\mathbb{K}}

\newcommand{\proj}{\Pi}
\newcommand{\comp}[1]{\overline {#1}}

\newcommand{\Jdnr}{\wt{J}_{n,\eta}}

\newcommand{\Jac}{J^{\mathrm{AC}}}
\newcommand{\Jdc}{J^{\mathrm{DC}}}

\newcommand{\Jacn}{J^{\mathrm{AC}}_n}
\newcommand{\JacnN}{J^{\mathrm{AC}}_{n,N}}

\newcommand{\Jdcn}{J^{\mathrm{DC}}_n}

\newcommand{\Jnlb}{J_{n,\eta}^{\rm LB}}
\newcommand{\Jnub}{J_{n,\eta}^{\rm UB}}
\newcommand{\Jneta}{\wt{J}_{n,\eta}}

\newcommand{\Jlb}{J_n^{\rm LB}}

\newcommand{\yeta}{y_{\eta}^{\star}}

\newcommand{\yn}{y^\star_n}

\newcommand{\ext}{\mathcal{E}}
\newcommand{\Cont}{\mathcal{C}}
\newcommand{\Prob}{\mathcal{P}}
\newcommand{\Meas}{\mathcal{M}}
\newcommand{\Sigalg}{\mathcal{G}}
\newcommand{\Func}{\mathcal{F}}
\newcommand{\lip}{\mathrm{L}}
\newcommand{\wass}{\mathrm{W}}

\newcommand{\uball}{\mathsf{B}_n}

\newcommand{\ynb}{{\theta_{\mathcal{D}}}}
\newcommand{\xnb}{\theta_{\mathcal{P}}}

\newcommand{\Rnorm}{\mathfrak{R}}
\newcommand{\Uball}{\U_n}

\newcommand{\cnew}{{\mathbf{c}}}

\newcommand{\wh}{\widehat}

\newcommand{\order}{\mathcal{O}}
\newcommand{\bdot}{\LargerCdot}
\newcommand{\ratio}{\varrho_n}
\newcommand{\T}{\mathds{T}}

\newcommand{\Leb}{\lambda}
\newcommand{\Ab}{\mathscr{A}}
\newcommand{\Yb}{\mathscr{Y}}

\newcommand{\cost}{\psi}
\newcommand{\NN}{\mathsf{N}}

\newcommand{\gmax}{g_{\rm max}}

\usepackage{makecell}
\usepackage{todonotes}

\usepackage{mathrsfs}

\newcommand{\geqc}[1]{\succeq_{#1}}

\usepackage{tikz}
\usetikzlibrary{fadings,shapes.arrows,shadows}   
\usetikzlibrary{chains}
\usetikzlibrary{fit}
\usepackage{pgflibraryarrows}		
\usepackage{pgflibrarysnakes}		
\usepackage{xcolor}
\usepackage{epsfig}
\usetikzlibrary{shapes.symbols,patterns} 
\usepackage{pgfplots}

\definecolor{darkgreen}{rgb}{0.0, 0.42, 0.24}

\newcommand{\K}{\mathbb{K}}

\newcommand{\drv}{\ensuremath{\mathrm{d}}}
\newcommand{\indic}[1]{\ensuremath{\boldsymbol{1}_{#1}}}

\newcommand{\Borel}[1]{\ensuremath{\mathcal{B}\!\left(#1\right)}}

\newcommand{\Rp}{\mathbb{R}_{+}}

\newcommand{\Probpi}[1]{\ensuremath{\mathbb{P}^{\pi}_{\nu}\!\left(\vphantom{\big|}#1\vphantom{\big|}\right)}}
\newcommand{\Expecpi}[1]{\ensuremath{\mathbb{E}^{\pi}_{\nu}\!\left[\vphantom{\big|}#1\vphantom{\big|}\right]}}

\newcommand{\Kb}{\mathcal{K}}

\newcommand{\supp}[2]{\ensuremath{\sigma_{#1}\!\left({#2}\right)}}

\newcommand*{\LargerCdot}{\raisebox{-0.45ex}{\scalebox{1.75}{$\cdot$}}}

\graphicspath{{Fig/}}

\title[]{From Infinite to Finite Programs: Explicit Error Bounds with Applications to Approximate Dynamic Programming}

\author[P. Mohajerin Esfahani]{Peyman Mohajerin Esfahani}
\author[Second]{Tobias Sutter}
\author[D. Kuhn]{Daniel Kuhn}
\author[Second]{John Lygeros}

\thanks{The authors are with the Delft Center for Systems and Control, TU Delft, The Netherlands ({\tt P.MohajerinEsfahani@tudelft.nl}), the Automatic Control Laboratory, ETH Zurich, Switzerland, ({\tt \{sutter,lygeros\}@control.ee.ethz.ch}), and the Risk Analytics and Optimization Chair, EPFL, Switzerland ({\tt daniel.kuhn@epfl.ch})}

\begin{document}
	\maketitle
	
	\begin{abstract}
		We consider linear programming (LP) problems in infinite dimensional spaces that are in general computationally intractable. Under suitable assumptions, we develop an approximation bridge from the infinite-dimensional LP to tractable finite convex programs in which the performance of the approximation is quantified explicitly. To this end, we adopt the recent developments in two areas of randomized optimization and first order methods, leading to a priori as well as a posteriori performance guarantees. We illustrate the generality and implications of our theoretical results in the special case of the long-run average cost and discounted cost optimal control problems for Markov decision processes on Borel spaces. The applicability of the theoretical results is demonstrated through a constrained linear quadratic optimal control problem and a fisheries management problem.
		
	\end{abstract}
	
	\key{infinite-dimensional linear programming, Markov decision processes, approximate dynamic programming, randomized and convex optimization}
	
	\section{Introduction} \label{sec:introduction}
	Linear programming (LP) problems in infinite dimensional spaces appear in, among other areas, engineering, economics, operations research and probability theory  \cite{ref:Luenberger_OptVecS, ref:Anderson-87, ref:Lasserre-11}. Infinite LPs offer remarkable modeling power, subsuming general finite dimensional optimization problems and the generalized moment problem as special cases. They are, however, often computationally formidable, motivating the study of approximations schemes.
	
	A particularly rich class of problems that can be modeled as infinite LPs involves Markov decision processes (MDP) and their optimal control. More often than not, it is impossible to obtain explicit solutions to MDP problems, making it necessary to resort to approximation techniques. Such approximations are the core of a methodology known as \emph{approximate dynamic programming} \cite{ref:Bertsekas-96, ref:Bertsekas-12}. Interestingly, a wide range of optimal control problems involving MDP can be equivalently expressed as \emph{static} optimization problems over a closed convex set of measures, more specifically, as infinite LPs \cite{ref:Hernandez-96, ref:Hernandez-99, ref:chapter:Lerma}. This LP reformulation is particularly appealing for dealing with unconventional settings involving additional constraints \cite{ref:Hernandez-03, ref:Borkar-08}, secondary costs \cite{ref:Dufour-14}, information-theoretic considerations \cite{ref:Raginsky-15}, and reachability problems \cite{ref:Kariot-13,ref:Moh:motion-16}. In addition, the infinite LP reformulation allows one to leverage the developments in the optimization literature, in particular convex approximation techniques, to develop approximation schemes for MDP problems. This will also be the perspective adopted in the present article.
	
	Approximation schemes to tackle infinite LPs have historically been developed for special classes of problems, e.g., the general capacity problem \cite{ref:Lai-92}, or the generalized moment problem \cite{ref:Lasserre-11}. The literature on control of MDP with infinite state or action spaces mostly concentrates on approximation schemes with asymptotic performance guarantees \cite{ref:Hernandez-99, ref:Hernandez-98}, see also the comprehensive book \cite{ref:Kush-03} for controlled stochastic differential equations and \cite{ref:PrandiniHu-2007,ref:MohChatLyg-15} for reachability problems in a similar setting. From a practical viewpoint, a challenge using these schemes is that the convergence analysis is not constructive and does not lead to explicit error bounds. 
	A wealth of approximation schemes have been proposed in the literature under the names of approximate dynamic programming \cite{ref:Bert-75}, neuro-dynamic programming \cite{ref:Bertsekas-96}, reinforcement learning \cite{ref:Konda-03,ref:TsitVan-97}, and value and/or policy iteration \cite{ref:Bertsekas-12}. Most, however, deal with discrete (finite or at most countable) state and action spaces, while approximation over uncountable spaces remains largely unexplored. 
	
	The MDP literature on explicit approximation errors in uncountable settings can, roughly speaking, be divided to two groups in terms of the performance criteria considered: discounted cost, and average cost. Of the two, the discounted cost setting has received more attention as the corresponding dynamic programming operator is a contraction, a useful property to obtain a convergence rate for the approximation error. Examples include the linear programming approach \cite{ref:FarVan-03,ref:FarVan-04}, and also a recent series of works \cite{ref:Dufour-13, ref:Dufour-14, ref:Dufour2-15} on approximating a probability measure that underlies the random transitions of the dynamics of the system using different discretization procedures. Long-run average cost problems introduce new challenges due to loosing the contraction property. The authors in \cite{ref:Dufour-15} develop approximation schemes leading to finite but non-convex optimization problems, while \cite{ref:Saldi-15} investigates the convergence rate of the finite-state approximation to the original (uncountable) MDP problem. 
	
	The approach presented in this article tackles a class of general infinite LPs that, as a special case, cover both long-run discounted and average cost performance criteria in the optimal control of MDP. The resulting approximation is based on finite convex programs that are different from the existing schemes. Closest in spirit to our proposed approximation is the linear programming approach based on constraint sampling in \cite{ref:FarVan-03,ref:FarVan-04,ref:SutMoh-14}. Unlike these works, however, we introduce an additional norm constraint that effectively acts as a \emph{regularizer}. We study in detail the conditions under which this regularizer can be exploited to bound the optimizers of the primal and dual programs, and hence provide an explicit approximation error for the proposed solution. 
			
	The proposed approximation scheme involves a restriction of the decision variables from an infinite dimensional space to a finite dimensional subspace, followed by the approximation of the infinite number of constraints by a finite subset; we develop two complementary methods for performing the latter step. The structure of the article is illustrates in Figure~\ref{fig:overview}, where the contributions are summarized as follows:

	\begin{figure}[t!]
		\centering
		\scalebox{.85}{
\def \sca{1}

\def \xb{3.9*\sca}
\def \yb{1.3*\sca}

\def \r{6.5*\sca}

\def \y{3.3*\sca} 
\def \yy{2*\sca} 

\def \la{0.2*\sca}

\def \di{0.2*\sca}

\def \backx{-0.1*\sca}
\def \backy{-1.4*\sca}
\def \back{0.2*\sca}

\def \buf{0.35*\sca}

\begin{tikzpicture}[scale=1,auto, node distance=1cm,>=latex']

\shade[rounded corners,bottom color = darkgreen!50, top color = darkgreen!25] (\r-\buf,\yb+\buf)--(\r+\r+\xb+\buf,\yb+\buf)--(\r+\r+\xb+\buf,-0.5*\y+3.6*\buf)--(\r-\buf,-0.5*\y+3.6*\buf)--cycle;

\shade[rounded corners,bottom color = violet!50, top color = violet!25] (\r-\buf,-0.6*\y+\buf)--(\r+\r+\xb+\buf,-0.6*\y+\buf)--(\r+\r+\xb+\buf,-1.0*\y-\buf)--(\r-\buf,-1.0*\y-\buf)--cycle;

\shade[rounded corners,bottom color = blue!50, top color = blue!25] (\r-\buf,-1.6*\y+\buf)--(\r+\r+\xb+\buf,-1.6*\y+\buf)--(\r+\r+\xb+\buf,-2*\y-\buf)--(\r-\buf,-2*\y-\buf)--cycle;


\shade[rounded corners,bottom color = gray!20, top color = gray!10] (0,0)--(\xb,0)--(\xb,\yb)--(0,\yb)--cycle;
\node at (0.5*\xb,0.7*\yb) {discrete time};
\node at (0.5*\xb,0.3*\yb) {MDP};
\shade[rounded corners,bottom color = gray!20, top color = gray!10] (\r,0)--(\r+\xb,0)--(\r+\xb,\yb)--(\r,\yb)--cycle;
\node at (\r+0.5*\xb,0.7*\yb) {infinite LP};
\node at (\r+0.5*\xb,0.3*\yb) {$\eqref{primal-inf}$: \ $\Jp$};
\shade[rounded corners,bottom color = gray!20, top color = gray!10] (\r,-\y)--(\r+\xb,-\y)--(\r+\xb,\yb-\y)--(\r,\yb-\y)--cycle;
\node at (\r+0.5*\xb,0.7*\yb-\y) {robust program};
\node at (\r+0.5*\xb,0.3*\yb-\y) {$\eqref{primal-n}$: \ $\Jpn$};
\shade[rounded corners,bottom color = gray!20, top color = gray!10] (\r,-\y-\y)--(\r+\xb,-\y-\y)--(\r+\xb,\yb-\y-\y)--(\r,\yb-\y-\y)--cycle;
\node at (\r+0.5*\xb,0.7*\yb-\y-\y) {scenario program};
\node at (\r+0.5*\xb,0.3*\yb-\y-\y) {$\eqref{primal-n,N}$: \ $\JpnN$};

\shade[rounded corners,bottom color = gray!20, top color = gray!10] (\r+\r,-\y)--(\r+\r+\xb,-\y)--(\r+\r+\xb,\yb-\y)--(\r+\r,\yb-\y)--cycle;
\node at (\r+\r+0.5*\xb,0.7*\yb-\y) {semi-infinite program};
\node at (\r+\r+0.5*\xb,0.3*\yb-\y) {$\eqref{dual-n}$: \ $\Jdn$};

\shade[rounded corners,bottom color = gray!20, top color = gray!10] (\r+\r,-\y-\y)--(\r+\r+\xb,-\y-\y)--(\r+\r+\xb,\yb-\y-\y)--(\r+\r,\yb-\y-\y)--cycle;
\node at (\r+\r+0.5*\xb,0.7*\yb-\y-\y) {regularized program};
\node at (\r+\r+0.5*\xb,0.3*\yb-\y-\y) {$\eqref{dual-n-eta}$: \ $\Jdnr$};

\shade[rounded corners,bottom color = gray!20, top color = gray!10] (\r+0.5*\r,-\y-\y-\y)--(\r+0.5*\r+\xb,-\y-\y-\y)--(\r+0.5*\r+\xb,\yb-\y-\y-\y)--(\r+0.5*\r,\yb-\y-\y-\y)--cycle;
\node at (\r+0.5*\r+0.5*\xb,0.7*\yb-\y-\y-\y) {prior $\&$ posterior error};
\node at (\r+0.5*\r+0.5*\xb,0.3*\yb-\y-\y-\y) {$\Jdnr - \JpnN$};

  \draw[>=latex,<->,line width=1.2mm,gray!25] (\xb,0.5*\yb) -- (\r,0.5*\yb);
    \draw[>=latex,<->,line width=1.2mm,gray!25] (\xb+\r,-\y+0.5*\yb) -- (\r+\r,-\y+0.5*\yb);
  \draw[>=latex,<->,dashed, line width=1.2mm,gray!25] (\r+0.5*\xb,0) -- (\r+0.5*\xb,-\y+\yb);
 \draw[>=latex,<->,dashed, line width=1.2mm,gray!25] (\r+0.5*\xb,-\y+0.25*\la) -- (\r+0.5*\xb,-\y+0.25*\la-\y+\yb);
 \draw[>=latex,<->,dashed, line width=1.2mm,gray!25] (\r+\r+0.5*\xb,-\y+0.25*\la) -- (\r+\r+0.5*\xb,-\y+0.25*\la-\y+\yb);
 
   \draw[>=latex,->,line width=1.2mm,gray!25] (\r+0.5*\xb,-\y-\y+0.2*\di) -- (\r+0.5*\r,\yb-\y-\y-\y);
   \draw[>=latex,->, line width=1.2mm,gray!25] (\r+\r+0.5*\xb,-\y-\y+0.2*\di) -- (\r+0.5*\r+\xb,\yb-\y-\y-\y);
 


  \shade[bottom color = darkgreen!50, top color = darkgreen!25] (0,-2.4*\y+1*\y)--(\di,-2.4*\y+1*\y)--(\di,-2.4*\y-\di+1*\y)--(0,-2.4*\y-\di+1*\y)--cycle;
  \node[] at (0.35*\r,-2.4*\y-0.5*\di+1*\y) {\  infinite program};

  \shade[bottom color = violet!50, top color = violet!25] (0,-2.7*\y+1*\y)--(\di,-2.7*\y+1*\y)--(\di,-2.7*\y-\di+1*\y)--(0,-2.7*\y-\di+1*\y)--cycle;
  \node[] at (0.35*\r,-2.7*\y-0.5*\di+1*\y) {\ \ \ \ \ \ \ \ \ semi-infinite programs};

  \shade[bottom color = blue!50, top color = blue!25] (0,-3.0*\y+1*\y)--(\di,-3.0*\y+1*\y)--(\di,-3.0*\y-\di+1*\y)--(0,-3.0*\y-\di+1*\y)--cycle;
  \node[] at (0.35*\r,-3.0*\y-0.5*\di+1*\y) {finite programs};

 \node[] at (0.5*\r+0.5*\xb+\r,-0.95*\y+0.5*\yb+1.3*\buf) {Proposition~$\ref{prop:SD}$};
 \node[] at (0.5*\r+0.5*\xb+\r,-\y+0.5*\yb -1.3*\buf) {}; 
 \node[] at (0.5*\r+0.5*\xb+\r,-1.05*\y+0.5*\yb -1.3*\buf) {strong duality};   
 \node[] at (0.5*\xb+0.5*\r,0.5*\yb+1.3*\buf) {equivalent};
 \node[] at (0.19*\r+0.5*\xb+\r,-0.5*\y+0.5*\yb-0.0*\buf) {Theorem~$\ref{thm:inf-semi}$};
 \node[] at (0.19*\r+0.5*\xb+\r,-1.5*\y+0.5*\yb-0.0*\buf) {Theorem~$\ref{thm:semi-fin:rand}$};  
 \node[] at (1.19*\r+0.5*\xb+\r,-1.5*\y+0.5*\yb-0.0*\buf) {Theorem~$\ref{thm:semi-fin:smooth}$};     
  \node[] at (0.5*\r+0.5*\xb+\r,-2.6*\y+0.5*\yb -1.3*\buf) {Theorem~\ref{thm:inf-fin}};  
\end{tikzpicture}}
		\caption{Graphical representation of the article structure and its contributions}
		\label{fig:overview}
	\end{figure}
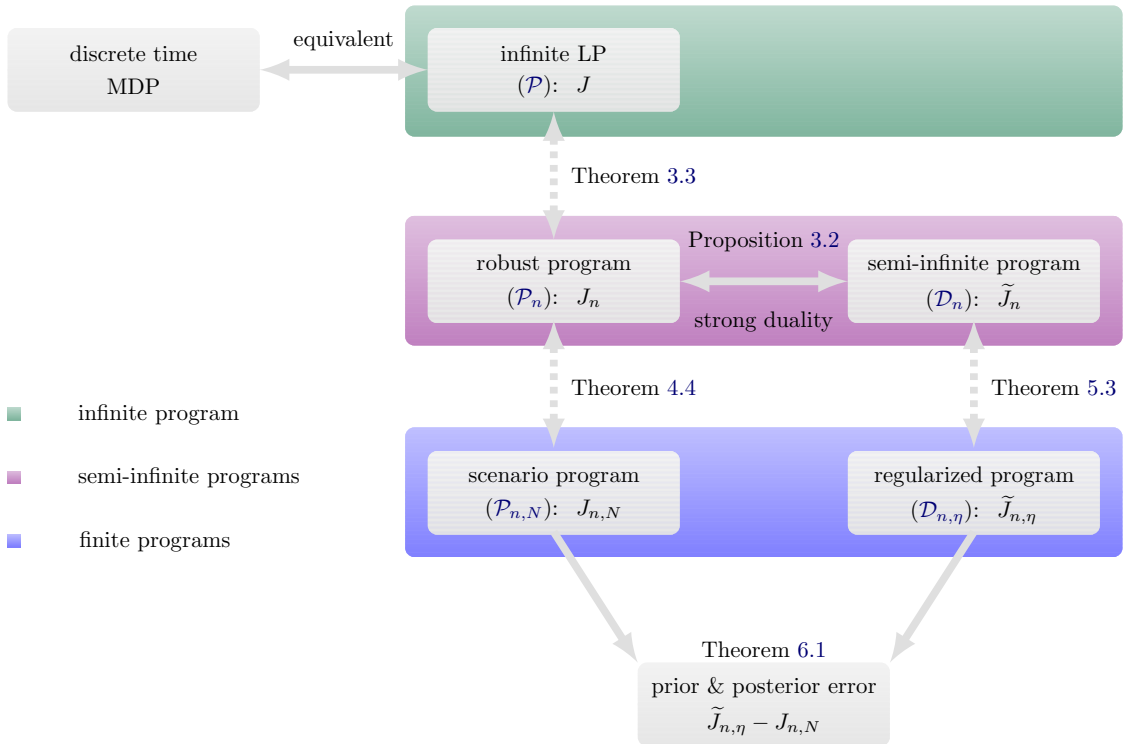
	
	\begin{enumerate}[label=$\bullet$, itemsep = 1mm, topsep = -1mm]
		\item \label{cont:1}
		We introduce a subclass of infinite LPs whose {\emph{regularized}} semi-infinite restriction enjoys analytical bounds for both primal and dual optimizers (Proposition~\ref{prop:SD}). The implications for MDP with average cost (Lemma~\ref{lem:MDP:bd-dual}) and with discounted cost (Lemma~\ref{lem:MDP:DC}) are also investigated. 
		
		\item \label{cont:2}
		We derive an explicit error bound between the original infinite LP and the regularized semi-infinite counterpart, providing insights on the impact of the underlying norm structure as well as on how the choice of basis functions contributes to the approximation error (Theorem~\ref{thm:inf-semi}, Corollary~\ref{cor:Hilbert}). In the MDP setting, we recover an existing result as a special case (Corollary~\ref{cor:MDP:inf-semi}).
		
		\item \label{cont:3}
		We adopt the recent developments from the randomized optimization literature to propose a finite convex program whose solution enjoys a priori probabilistic performance bounds (Theorem~\ref{thm:semi-fin:rand}). We extend the existing results to offer also an a posteriori bound under a generic underlying norm structure. The required conditions and theoretical assertions are validated in the MDP setting (Corollary~\ref{cor:adp:semi-finite}).
		
		\item \label{cont:4}
		In parallel to the randomized approach, we also utilize the recent developments in the structural convex optimization literature to propose an iterative algorithm for approximating the semi-infinite program. For this purpose, we extend the setting to incorporate unbounded prox-terms with a certain growth rate (Theorem~\ref{thm:semi-fin:smooth}). We illustrate how this extension allows us to deploy the entropy prox-term in the MDP setting (Lemma~\ref{lem:entropy:bound}, Corollary~\ref{cor:adp:smooth}).
	\end{enumerate}

	Section~\ref{sec:problem:statement} introduces the main motivation for the work, namely the control of discrete-time MDP and their LP characterization. Using standard results in the literature we embed these MDP in the more general framework of infinite LPs. Section~\ref{sec:inf-semi} studies the link from infinite LPs to semi-infinite programs. Section~\ref{sec:semi-fin:rand} presents the approximation of semi-infinite programs based on randomization, while Section~\ref{sec:semi-fin:smoothing} approaches the same objective using first-order convex optimization methods. Section~\ref{sec:inf-fin} summarizes the results in the preceding sections, establishing the approximation error from the original infinite LP to the finite convex counterparts. Section~\ref{sec:sim} illustrates the theoretical results through a truncated LQG example and a fisheries management problem. 
	

	\paragraph{\bf Notation} 
	
	The set $\R_+$ denotes the set of non-negative reals and $\|\cdot\|_{\ell_p}$ for $p \in [1,\infty]$ the standard $p$-norm in $\R^n$. Given a function $u:S\ra \R$, we denote the infinity norm of the function by $\|u\|_\infty \Let \sup_{s \in S}|u(s)|$, and the Lipschitz norm by $\|u\|_{\lip} \Let \sup_{s, s' \in S} \big\{ |u(s)|, {|u(s) - u(s')| \over \|s - s'\|_{\ell_\infty}}  \big\}$. The space of Lipschitz functions on a set $S$ is denoted by $\Lip(S)$; define the function $\ind(s) \equiv 1$ for all $s \in S$. We denote the Borel $\sigma$-algebra on the (topological) space $S$ by $\borel(S)$. Measurability is always understood in the sense of Borel. Products of topological spaces are assumed to be endowed with the product topology and the corresponding product $\sigma$-algebra. The space of finite signed measures (resp.\ probability measures) on $S$ is denoted by $\Meas(S)$ (resp.\ $\Prob(S)$). The Wasserstein norm on the space of signed measures $\Meas(S)$ is defined by $\|\mu\|_\wass\Let \sup_{ \|u\|_{\lip} \leq 1} \int_{S}u(s)\mu(\diff s)$ and can be shown to be the dual of the Lipschitz norm. The set of extreme points of a set $A$ is denoted by $\ext\{A\}$. Given a bilinear form $\inner{\cdot}{\cdot}$, the support function of $A$ is defined by $\supp{A}{y}=\sup_{x\in A}\inner{y}{x}$. The standard bilinear form in $\R^n$ (i.e., the inner product) is denoted by $y\bdot x$.

	\section{Motivation: Control of MDP and LP Characterization} \label{sec:problem:statement}
	
	\subsection{MDP setting} 
	\label{subsec:AC-setting}
	We briefly recall some standard definitions and refer interested readers to \cite{ref:Hernandez-96, ref:Hernandez-03, ref:Arapostathis-93} for further details. Consider a \emph{Markov control model}
	$\big( S,A,\{A(s) : s\in S\},Q,\cost \big),$
	where $S$ (resp.\ $A$) is a metric space called the \emph{state space} (resp.\ \emph{action space}) and for each $s \in S$ the measurable set $A(s) \subseteq A$ denotes the set of \textit{feasible actions} when the system is in state $s\in S$. The \emph{transition law} is a stochastic kernel $Q$ on $S$ given the feasible state-action pairs in $K \Let\{(s,a):s\in S, a\in A(s)\}$. A stochastic kernel acts on real valued measurable functions $u$ from the left as
	\begin{align*}
	Qu(s,a):= \int_{S}u(s') Q(\drv s'|s,a), \quad \forall (s,a)\in K,
	\end{align*}
	and on probability measures $\mu$ on $K$ from the right as
	\begin{align*}
	\mu Q(B) := \int_{K}Q(B|s,a)\mu\big(\drv(s,a)\big), \quad \forall B\in \borel(S).
	\end{align*} 
	Finally $\cost:K \to\R_+$ denotes a measurable function called the \emph{one-stage cost function}. The \emph{admissible history spaces} are defined recursively as $H_{0}\Let S$ and $H_{t}\Let H_{t-1}\times K$ for $t\in\N$ and the canonical sample space is defined as $\Omega\Let(S\times A)^{\infty}$. All random variables will be defined on the measurable space $(\Omega,\Sigalg)$ where $\Sigalg$ denotes the corresponding product $\sigma$-algebra. A generic element $\omega\in\Omega$ is of the form $\omega=(s_{0},a_{0},s_{1},a_{1},\hdots)$, where $s_{i}\in S$ are the states and $a_{i}\in A$ the action variables. An \textit{admissible policy} is a sequence $\pi=(\pi_{t})_{t\in\N_{0}}$ of stochastic kernels $\pi_{t}$ on $A$ given $h_t\in H_{t}$, satisfying the constraints $\pi_{t}(A(s_{t})|h_{t})=1$. The set of admissible policies will be denoted by $\Pi$. 
	Given a probability measure $\nu\in \Prob(S)$ and policy $\pi\in\Pi$, by the Ionescu Tulcea theorem \cite[p.~140-141]{ref:Bertsekas-78} there exists a unique probability measure $\mathbb{P}^{\pi}_{\nu}$ on $\left(\Omega,\Sigalg \right)$ such that for all measurable sets $B \subset S$, $C \subset A$, $h_{t}\in H_{t}$, and $t\in\N_{0}$
	\begin{align*}
	\Probpi{s_{0}\in B}			&= \nu(B)  \\
	\Probpi{a_{t}\in C|h_{t}}		&= \pi_{t}(C|h_{t})  \\
	\Probpi{s_{t+1}\in B|h_{t},a_{t}}	&= Q(B|s_{t},a_{t}). 
	\end{align*}
	The expectation operator with respect to $\mathbb{P}^{\pi}_{\nu}$ is denoted by $\mathbb{E}^{\pi}_{\nu}$. The stochastic process $\big( \Omega,\Sigalg,\mathbb{P}^{\pi}_{\nu},(s_{t})_{t\in\N_{0}} \big)$ is called a \emph{discrete-time MDP}.
	For most of the article we consider optimal control problems where the aim is to minimise a long term \emph{average cost} (AC) over the set of admissible policies and initial state measures. We definite the optimal value of the optimal control problem by
	\begin{align} \label{AC}
	\Jac \Let  \inf_{(\pi,\nu) \in \Pi \times \Prob(S)}\limsup_{T\to\infty} \frac{1}{T} \Expecpi{\sum_{t=0}^{T-1}\cost(s_{t},a_{t})}.
	\end{align}	
	We emphasize, however, that the results also apply to other performance objective, including the long-run \emph{discounted cost} problem as shown in Appendix~\ref{app:discouted:cost}.
	
	\subsection{Infinite LP characterization} \label{subsec:LP}
	The problem in \eqref{AC} admits an alternative LP characterization under some mild assumptions. 
	
	\begin{As}[Control model] \label{a:CM} We stipulate that
		\begin{enumerate}[label=(\roman*), itemsep = 1mm, topsep = -1mm]
			\item \label{a:CM:K} the set of feasible state-action pairs is the unit hypercube $K = [0,1]^{\dim(S\times A)}$;
			\item \label{a:CM:Q} the transition law $Q$ is Lipschitz continuous, i.e., there exists $L_Q>0$ such that for all $k, k'\in K$ and all continuous functions $u$  
			$$|Qu(k) - Qu(k')| \le L_Q \|u\|_\infty \|k - k'\|_{\ell_\infty};$$
			\item \label{a:CM:cost} the cost function $\cost$ is non-negative and Lipschitz continuous on $K$ with respect to the $\ell_\infty$-norm.
		\end{enumerate}
	\end{As}
	
	Assumption \ref{a:CM}\ref{a:CM:K} may seem restrictive, however, essentially it simply requires that the state-action set $K$ is compact. We refer the reader to Example~\ref{ex:fisheries} where a non-rectangular $K$ is transferred to a hypercube, and to \cite[Chapter~12.3]{ref:Hernandez-99} for further information about the LP characterization in more general settings.
	
	\begin{Thm}[LP characterization {\cite[Proposition~2.4]{ref:Dufour-15}}] 
		\label{thm:equivalent:LP}
		Under Assumption~\ref{a:CM}, 
		\begin{align}
		\label{AC-LP} 
		-\Jac = & \left\{ \begin{array}{ll}
		\inf\limits_{\rho, u} & -\rho   \\
		\st &\rho + u(s) - Qu(s,a) \leq \cost(s,a), \quad \forall (s,a)\in K \\
		& \rho\in\R, \quad u\in\Lip(S).
		\end{array} \right. 
		\end{align}
	\end{Thm}
	
	The LP \eqref{AC-LP} can be expressed in the standard conic form $\inf_{x \in \X} \big\{\inner{x}{c} : \op x - b \in \cone \big\}$ by introducing
	\begin{align}
	\label{AC-setting}
	\left\{
	\begin{array}{l}
	\X = \R\times\Lip(S) \\
	x = (\rho, u) \in \X\\
	c = (c_{1},c_{2})=(-1,0) \in \R\times\Meas(S) \\
	b(s,a) = -\cost(s,a) \\ 
	\inner{x}{c} = c_{1}\rho +\int_S u(s) c_{2}(\diff s) \\
	\op x(s,a) = -\rho - u(s) + Qu(s,a) \\
	\cone = \Lip_+(K),  
	\end{array}\right.
	\end{align}
	where $\Meas(S)$ is the set of finite signed measures supported on $S$, and $\Lip_+(K)$ is the cone of Lipschitz functions taking non-negative values. It should be noted that the choice of the positive cone $\K = \Lip_+(K)$ is justified since, thanks to Assumption~\ref{a:CM}\ref{a:CM:Q}, the linear operator $\op$ maps the elements of $\X$ into $\Lip(K)$. 
	
	\begin{Rem}[Constrained MDP]
		The LP characterization of MDP naturally allows us to incorporate constraints in the form of 
		\begin{align*}
		\limsup_{T\to\infty} \frac{1}{T}\, \Expecpi{\sum_{t=0}^{T-1}d_i(s_{t},a_{t})} \le \ell_i, \qquad \forall i \in \{1,\cdots,I\},
		\end{align*}
		where the functions $d_i: K \ra \R$ and constants $\ell_i$ reflect our desired specifications. To this end, it suffices to introduce auxiliary decision variables $\beta_i \in \R_+$, and in \eqref{AC-LP} replace $\rho$ in the objective with $\rho - \sum_{i=1}^{I}\beta_{i} \ell_{i}$ and in the constraint with $ \rho - \sum_{i=1}^{I}\beta_{i}d_{i}$, see \cite[Theorem 5.2]{ref:Hernandez-03}. 
	\end{Rem}
	
	Our aim is to derive an approximation scheme for a class of such infinite dimensional LPs, including problems of the form \eqref{AC-LP}, that comes with an explicit bound on the approximation error.
	
	\section{Infinite to Semi-infinite Programs} 
	\label{sec:inf-semi}

	\subsection{Dual pairs of normed vector spaces} 
	\label{subsec:dual-pair}
	The triple $\big(\X,\C, \|\cdot\|\big)$ is called a \emph{dual pair} of normed vector spaces if
	\begin{itemize}
		\item $\X$ and $\C$ are vector spaces; 
		
		\item $\inner{\cdot}{\cdot}$ is a bilinear form on $\X\times \C$ that ``separates points", i.e., 
		\begin{itemize}
			\item for each nonzero $x \in \X$ there is some $c \in \C$ such that $\inner{x}{c} \neq 0$,
			\item for each nonzero $c \in \C$ there is some $x \in \X$ such that $\inner{x}{c} \neq 0$;
		\end{itemize}
		
		\item $\X$ is equipped with the norm $\|{\cdot}\|$, which together with the bilinear form induces a \emph{dual} norm in $\C$ defined through $\|{c}\|_* \Let \sup_{\|{x}\|\le 1}\inner{x}{c}$. 
	\end{itemize}
	The norm in the vector spaces is used as a means to quantify the performance of the approximation schemes. In particular, we emphasize that the vector spaces are not necessarily complete with respect to these norms.
	
	Let $\big(\B,\Y, \|\cdot\| \big)$ be another dual pair of normed vector spaces. As there is no danger of confusion, we use the same notation for the potentially different norm and bilinear form for each pair. Let $\op:\X \ra \B$ be a linear operator, and $\cone$ be a convex cone in $\B$. Given the fixed elements $c \in \C$ and $b \in \B$, we define a linear program, hereafter called the \emph{primal} program \ref{primal-inf}, as
	\begin{align} 
	\label{primal-inf} 
	\tag{$\Prim$}
	\Jp \Let \left\{ \begin{array}{ll}
	\Inf{x \in \X} & \inner{x}{c}   \\
	\st & \op x \geqc{\cone} b 
	\end{array} \right.
	\end{align}
	where the conic inequality $\op x \geqc{\cone} b$ is understood in the sense of $\op x - b \in \cone$. Throughout this study we assume that the program \ref{primal-inf} has an optimizer (i.e., the infimum is indeed a minimum), the cone $\cone$ is closed and the operator $\op$ is continuous where the corresponding topology is the weakest in which the topological duals of $\X$ and $\B$ are $\C$ and $\Y$, respectively. Let $\op^*:\Y \ra \C$ be the adjoint operator of $\op$ defined by
		\begin{align*}
		\inner{\op x}{y} = \inner{x}{\op^* y}, \qquad \forall x \in \X, \quad \forall y \in \Y. 
		\end{align*}
		Recall that if $\op$ is weakly continuous, then the adjoint operator $\op^*$ is well defined as its image is a subset of $\C$ \cite[Proposition 12.2.5]{ref:Hernandez-99}. The \emph{dual} program of \ref{primal-inf} is denoted by \ref{dual-inf} and is given by 
		\begin{align} 
		\label{dual-inf} 
		\tag{$\Dual$}
		\Jd \Let \left\{ \begin{array}{ll}
		\Sup{y\in \Y} & \inner{b}{y}   \\
		\st & \op^* y = c \\
		& y \in \cone^*, 
		\end{array} \right.
		\end{align}
		where $\cone^*$ is the dual cone of $\cone$ defined as $\cone^* \Let \big\{ y \in \Y : \inner{b}{y} \ge 0, \  \forall b \in \cone \big\}.$ It is not hard to see that \emph{weak duality} holds, as
		\begin{align*}
		\Jp = \Inf{x \in \X} \Sup{y \in \cone^*} \inner{x}{c} - \inner{\op x - b}{y} \ge \Sup{y \in \cone^*} \Inf{x \in \X} \inner{x}{c} - \inner{\op x - b}{y} = \Jd.
		\end{align*}
		An interesting question is when the above assertion holds as an equality. This is known as \emph{zero duality gap}, also referred to as \emph{strong duality} particularly when both \ref{primal-inf} and \ref{dual-inf} admit an optimizer \cite[p.\ 52]{ref:Anderson-87}. Our study is not directly concerned with conditions under which strong duality between \ref{primal-inf} and \ref{dual-inf} holds; see \cite[Section 3.6]{ref:Anderson-87} for a comprehensive discussion of such conditions. The programs \ref{primal-inf} and \ref{dual-inf} are assumed to be \emph{infinite}, in the sense that the dimensions of the decision spaces ($\X$ in \ref{primal-inf}, and $\Y$ in \ref{dual-inf}) as well as the number of constraints are both infinite. 
		
		\subsection{Semi-infinite approximation}
		\label{subsec:semi}
		
		Consider a family of linearly independent elements $\{x_n\}_{n \in \N} \subset \X$, and let $\X_n$ be the finite dimensional subspace generated by the first $n$ elements $\{x_i\}_{i\le n}$. Without loss of generality, we assume that $x_i$ are normalized, i.e., $\|x_i\| = 1$. Restricting the decision space $\X$ of \ref{primal-inf} to $\X_n$, along with an additional norm constraint, yields the program
		\begin{align} 
		\label{primal-n:old}
		\Jpn \Let \left\{ \begin{array}{ll}
		\Inf{\alpha \in \R^n} & \sum_{i = 1}^{n} \alpha_i \inner{x_i}{c}   \\
		\st & \sum_{i = 1}^{n} \alpha_i \op x_i \geqc{\cone} b \vspace{1mm}\\
		& \|\alpha\|_{\Rnorm} \le \xnb \vspace{1mm}
		\end{array} \right.
		\end{align}
		where $\|\cdot\|_{\Rnorm}$ is a given norm on $\R^n$ and $\xnb$ determines the size of the feasible set. In the spirit of dual-paired normed vector spaces, one can approximate $(\X,\C,\|\cdot\|)$ by the finite dimensional counterpart $(\R^n,\R^n,\|\cdot\|_\Rnorm)$ where the bilinear form is the standard inner product. In this view, the linear operator $\op:\X \ra \B$ may also be approximated by the linear operator $\opn:\R^n \ra \B$ with the respective adjoint $\opn^*:\Y \ra \R^n$ defined as
		\begin{align}
		\label{Ln}
		\opn \alpha \Let \sum_{i=1}^{n} \alpha_i \op x_i, 
		\qquad 
		\opn^*y \Let \big[\inner{\op x_1}{y}, \cdots, \inner{\op x_n}{y} \big].
		\end{align}
		It is straightforward to verify the definitions \eqref{Ln} by noting that $\inner{\opn\alpha}{y} = \alpha \bdot \opn^*y$ for all $\alpha\in\R^n$ and $y\in\Y$. 
		Defining the vector $\cnew \Let [\inner{x_1}{c}, \cdots, \inner{x_n}{c}]$, we can rewrite the program \eqref{primal-n:old} as
		\begin{align} 
		\label{primal-n} 
		\tag{$\Prim_n$}	
		\Jpn \Let \left\{ \begin{array}{ll}
		\Inf{\alpha \in \R^n} & \alpha \bdot \cnew   \\
		\st & \opn \alpha \geqc{\cone} b \vspace{1mm}\\
		& \|\alpha\|_{\Rnorm} \le \xnb. 
		\end{array} \right.
		\end{align}
		We call \ref{primal-n} a \emph{semi-infinite} program, as the decision variable is a finite dimensional vector $\alpha \in \R^n$, but the number of constraints is still in general infinite due to the conic inequality. The additional constraint on the norm of $\alpha$ in \ref{primal-n} acts as a \emph{regularizer} and is a key difference between the proposed approximation schemes and existing schemes in the literature. Methods for choosing the parameter $\xnb$ will be discussed later. 
		
		Dualizing the conic inequality constraint in \ref{primal-n} and using the dual norm definition leads to a dual counterpart
		\begin{align} 
		\label{dual-n} 
		\tag{$\Dual_n$}
		\Jdn \Let \left\{ \begin{array}{ll}
		\Sup{y \in \Y} & \inner{b}{y} - \xnb \|\opn^* y - \cnew\|_{\Rnorm^*}  \\
		\st& y \in \cone^*, 
		\end{array} \right.
		\end{align}
		where$\|\cdot\|_{\Rnorm^*}$ denotes the dual norm of $\|\cdot\|_\Rnorm$. Note that setting $\xnb = \infty$ effectively implies that the second term of the objective in \ref{dual-n} introduces $n$ hard constraints $\opn^* y = \cnew$ (cf. \eqref{Ln}). We study further the connection between \ref{primal-n} and \ref{dual-n} under the following regularity assumption: 
		
		\begin{As}[Semi-infinite regularity]
			\label{a:reg}
			We stipulate that 
			\begin{enumerate}[label=(\roman*), itemsep = 1mm, topsep = -1mm]
				\item \label{a:reg:feas} the program \ref{primal-n} is feasible; 
				\item \label{a:reg:inf-sup} there exists a positive constant $\gamma$ such that $\|\opn^* y\|_{\Rnorm^*} \geq \gamma \|y\|_*$ for every $y \in \cone^*$, and $\xnb$ is large enough so that $\gamma \xnb > \|b\|$.
				
			\end{enumerate}
		\end{As}
		
		Assumption \ref{a:reg}\ref{a:reg:inf-sup} is closely related to the condition
		\begin{align*}
		\inf_{y \in \cone^*}\sup_{x \in \X_n} {\inner{\op x}{y} \over \|x\| \|y\|_*} \geq \gamma, 
		\end{align*}
		that in the literature of numerical algorithms in infinite dimensional spaces, in particular the Galerkin discretization methods for partial differential equations, is often referred to as the ``\emph{inf-sup}" condition, see \cite{ref:Ern-04} for a comprehensive survey. To see this, note that for every $x \in \X_n$ the definitions in \eqref{Ln} imply that
		$$\inner{\op x}{y} = \inner{\opn \alpha}{y} = \alpha \bdot \opn^* y, \qquad x = \sum_{i=1}^{n} \alpha_i x_i.$$
		These conditions are in fact equivalent if the norm $\|\cdot\|_\Rnorm$ is induced by the original norm on $\X$, i.e., $\|\alpha\|_\Rnorm \Let \| \sum_{i=1}^{n} \alpha_i x_i\|$. We note that $\opn^*$ maps an infinite dimensional space to a finite dimensional one, and as such  Assumption~\ref{a:reg}\ref{a:reg:inf-sup} effectively necessitates that the null-space of $\opn^*$ intersects the positive cone $\cone^*$ only at $0$. In the following we show that this regularity condition leads to a zero duality gap between \ref{primal-n} and \ref{dual-n}, as well as an upper bound for the dual optimizers. The latter turns out to be a critical quantity for the performance bounds of this study. 
		
		\begin{Prop}[Duality gap \& bounded dual optimizers]
			\label{prop:SD}
			Under Assumption \ref{a:reg}\ref{a:reg:feas}, the duality gap between the programs \ref{primal-n} and \ref{dual-n} is zero, i.e., $\Jpn = \Jdn$. 
			If in addition Assumption \ref{a:reg}\ref{a:reg:inf-sup} holds, then for any optimizer $\yn$ of the program \ref{dual-n} and any lower bound $\Jlb \le \Jpn$ we have
			\begin{align}
			\label{yb}
			\|\yn\|_* \leq
			\ynb \Let {\xnb\|\cnew\|_{\Rnorm^*} - \Jlb \over \gamma \xnb - \|b\|} \le  {2\xnb\|\cnew\|_{\Rnorm^*} \over \gamma \xnb - \|b\|}.
			\end{align} 
		\end{Prop}
		
		\begin{proof}
			Since the elements $\{x_i\}_{i\le n}$ are linearly independent, the feasible set of the decision variable $\alpha$ in program \ref{primal-n} is a bounded closed subset of a finite dimensional space, and hence compact. Thus, thanks to the feasibility Assumption~\ref{a:reg}\ref{a:reg:feas} and compactness of the feasible set, the zero duality gap follows because
			\begin{align*}
			\Jpn = \inf_{\|\alpha\|_\Rnorm \le \xnb} \Big\{\alpha \bdot \cnew + \sup_{y \in \cone^*} \inner{b - \opn \alpha}{y} \Big\} &=  \sup_{y \in \cone^*}\inf_{\|\alpha\|_\Rnorm \le \xnb}\Big\{ \inner{b}{y} - \alpha \bdot (\opn^* y - \cnew)\Big\} = \Jdn,
			\end{align*}
			where the first equality holds by the definition of the dual cone $\cone^*$, and the second equality follows from Sion's minimax theorem \cite[Theorem 4.2]{ref:Sion-58}. Thanks to the zero duality gap above, we have 
			\begin{align*}
			\Jlb \le \Jpn = \Jdn & = \inner{b}{\yn} - \xnb\|\opn^* \yn - \cnew\|_{\Rnorm^*} \le \inner{b}{\yn}  - \xnb\|\opn^* \yn\|_{\Rnorm^*} + \xnb \|\cnew\|_{\Rnorm^*}.
			\end{align*}
			By Assumption \ref{a:reg}\ref{a:reg:inf-sup}, we then have
			\begin{align*}
			\Jpn & \le \|b\|\|{\yn}\|_* - \gamma \xnb\|\yn\|_* + \xnb\|\cnew\|_{\Rnorm^*} = \xnb\|\cnew\|_{\Rnorm^*} - \big(\gamma \xnb - \|b\|\big)\|\yn\|_*,
			\end{align*}
			which together with the simple lower bound $\Jlb \Let -\xnb\|\cnew\|_{\Rnorm^*} \le \Jpn$ concludes the proof. 
		\end{proof}
		
		Proposition~\ref{prop:SD} effectively implies that in the program \ref{dual-n} one can add a norm constraint $\|y\|_* \le \ynb$ without changing the optimal value. The parameter $\ynb$ depends on $\Jlb$, a lower bound for the optimal value of $\Jpn$. A simple choice for such a lower bound is $-\xnb\|\cnew\|_{\Rnorm^*}$, but in particular problem instances one may be able to obtain a less conservative bound. We validate the assertions of Proposition~\ref{prop:SD} for long-run average cost problems in the next section and for long-run discounted cost problems in Appendix~\ref{app:discouted:cost}.
		
		Program~\ref{primal-n} is a restricted version of the original program \ref{primal-inf} (also called an \emph{inner approximation} \cite[Definition 12.2.13]{ref:Hernandez-99}), and thus $\Jp \le \Jpn$. 
		However, under Assumption \ref{a:reg}, we show that the gap $\Jpn-\Jp$ can be quantified explicitly.
		To this end, we consider the projection mapping $\proj_{\set A}(x) \Let \arg\min_{x' \in \set A} \| x' - x\|$, the operator norm $\|{\op}\| \Let \sup_{\|x\| \le 1} \|{\op x}\|$, and define the set 
		\begin{align}
		\label{uball}
		\uball \Let \Big\{\sum_{i=1}^{n} \alpha_i x_i \in \X_n ~:~ \|\alpha\|_\Rnorm \le \xnb\Big\}.
		\end{align}
		
		\begin{Thm}[Semi-infinite approximation]
			\label{thm:inf-semi}
			Let $x\opt$ and $\yn$ be optimizers for the programs \ref{primal-inf} and \ref{dual-n}, respectively, and let $r_n \Let x\opt - \proj_{\uball}(x\opt)$ be the projection residual of the optimizer $x\opt$ onto  the set $\uball$ as defined in \eqref{uball}. Under Assumption~\ref{a:reg}\ref{a:reg:feas}, we have $0\leq \Jpn - \Jp \le \inner{r_n}{\op^*\yn - c}$ where $\Jpn$ and $J$ are the optimal value of the programs \ref{primal-n} and \ref{primal-inf}. In addition, if Assumption~\ref{a:reg}\ref{a:reg:inf-sup} holds, then 
			\begin{align}
			\label{inf-semi error}
			0 \le \Jpn - \Jp \le \big(\|c\|_*+\ynb\|\op\| \big)\|r_n\|,
			\end{align}
			where $\ynb$ is the dual optimizer bound introduced in \eqref{yb}.
		\end{Thm}
		
		\begin{proof}
			The lower bound $0\le \Jpn - \Jp$ is trivial, and we only need to prove the upper bound. Note that since the optimizer $x\opt \in \X$ is a feasible solution of \ref{primal-inf}, then $\op x\opt - b \in \cone$. By the definition of the dual cone $\cone^*$, this implies that $\inner{\op x\opt - b}{y} \ge 0$ for all $y \in \cone^*$. Since the dual optimizer $\yn$ belongs to the dual cone $\cone^*$, then 
			\begin{align*}
			\Jpn - \Jp &\le \Jpn - \Jp + \inner{\op x\opt - b}{\yn} = \Jpn - \inner{x\opt}{c} + \inner{\op x\opt}{\yn} - \inner{b}{\yn} \\
			& = \Jpn + \inner{x\opt}{\op^*\yn - c} - \inner{b}{\yn} \\
			& = \Jpn +  \inner{r_{n}}{\op^*\yn - c} + \inner{\proj_{\uball}(x\opt)}{\op^*\yn - c} - \inner{b}{\yn},\\
			& = \Jpn +  \inner{r_{n}}{\op^*\yn - c} + \wt{\alpha}\bdot\big({\opn^*\yn - \cnew}\big) - \inner{b}{\yn},
			\end{align*}
			for some $\wt\alpha \in \R^n$ with norm $\|\wt\alpha\|_\Rnorm \le \xnb$; for the last line, see the definition of the operator $\opn$ in \eqref{Ln} as well as the vector $\cnew$ in the program \ref{primal-n}. Using the definition of the dual norm and the operators \eqref{Ln}, one can deduce from above that
			\begin{align*}
			\Jpn - \Jp & \le J_n + \inner{r_{n}}{\op^* \yn - c} + \xnb \|\opn^* \yn - \cnew\|_{\Rnorm^*} - \inner{b}{\yn} = J_n + \inner{r_{n}}{\op^* \yn - c} - \Jdn,
			\end{align*}
			which in conjunction with the zero duality gap ($\Jpn = \Jdn$) establishes the first assertion of the proposition. The second assertion is simply the consequence of the first part and the norm definitions, i.e.,
			\begin{align*}
			\inner{r_n}{\op^* \yn- c} = \inner{r_n}{-c} + \inner{\op r_n}{\yn} \le  \|r_n\|\|{c}\|_* + \|\op r_n\|\| \yn\|_{*} \le \|r_n\| \Big( \|c\|_{*} + \|\op\| \|\yn\|_* \Big).
			\end{align*}
			Invoking the bound on the dual optimizer $\yn$ from Proposition \ref{prop:SD} completes the proof. 
		\end{proof}		
		
		\begin{Rem}[Impact of norms on semi-infinite approximation]\label{rem:norm-semi} 
			We note the following concerning the impact of the choice of norms on the approximation error:
			\begin{enumerate} [label=(\roman*), itemsep = 1mm, topsep = -1mm]
				\item \label{rem:norm-semi:theta}
				The only norm that influences the semi-infinite program \ref{primal-n} is $\|\cdot\|_\Rnorm$ on $\R^n$. When it comes to the approximation error \eqref{inf-semi error}, the norm $\|\cdot\|_\Rnorm$ may have an impact on the residual $r_n$ only if the set $\uball$ in \eqref{uball} does not contain $\proj_{\X_n}(x\opt)$, the projection $x\opt$ on the subspace $\X_n$, where $x\opt$ is an optimizer of the infinite  program~\ref{primal-inf}. 	
				
				\item \label{rem:norm-semi:dual-pair}
				The norms of the dual pairs of vector spaces only appear in Theorem~\ref{thm:inf-semi} to quantify the approximation error. Note that in \eqref{inf-semi error} the stronger the norm on $\X$, the higher $\|r_n\|$, and the lower $\|c\|_*$ and $\|\op\|$. On the other hand, the stronger the norm on $\B$, the higher $\|b\|$ and $\|\op\|$ and the lower $\gamma$ (cf. Assumption~\ref{a:reg}\ref{a:reg:inf-sup}). 
			\end{enumerate}
		\end{Rem}				
		
		The error bound \eqref{inf-semi error} can be further improved when $\X$ is a Hilbert space. In this case, let $\comp{\X}_n$ denote the orthogonal complement of $\X_n$. We define the \emph{restricted} norms by
		\begin{align}
		\label{norm-res}
		\|c\|_{*n} \Let \sup_{x \in \comp{\X}_n} {\inner{x}{c} \over {\|{x}\|}}, \qquad \|\op\|_{n} \Let \sup_{x \in \comp{\X}_{n}} \frac{\|{\op x}\|}{\|{x}\|}. 
		\end{align}
		It is straightforward to see that by definition $\|c\|_{*n} \le \|c\|_{*}$ and $\|\op\|_{n} \le \|\op\|$.
		
		\begin{Cor}[Hilbert structure]\label{cor:Hilbert}
			Suppose that $\X$ is a Hilbert space and $\|\cdot\|$ is the norm induced by the corresponding inner product. Let $\{x_i\}_{i\in\N}$ be an orthonormal dense family and $\|\cdot\|_\Rnorm = \|\cdot\|_{\ell_2}$. Let $x\opt$ be an optimal solution for \ref{primal-inf} and chose $\xnb \ge \|x\opt\|$. Under the assymptions of Theorem~\ref{thm:inf-semi}, we have 
			\begin{align*}
			0 \le \Jpn - \Jp \le \big(\|c\|_{n}+\ynb\|\op\|_n \big)\big\|\proj_{\comp{\X}_{n}}(x\opt) \big\|.
			\end{align*}
		\end{Cor}
		
		\begin{proof}
			We first note that the $\ell_2$-norm on $\R^n$ is indeed the norm induced by $\|\cdot\|$, since due to the orthonormality of $\{x_i\}_{i\in\N}$ we have
			$$\|\alpha\|_\Rnorm \Let \Big\|\sum_{i=1}^{n} \alpha_i x_i \Big\| = \sqrt{\sum_{i=1}^{n} \alpha_i^2\|x_i\|^2} = \|\alpha\|_{\ell_2}.$$ 
			If $\xnb \ge \|x\opt\|$, then $\proj_{\uball}(x\opt) = \proj_{\X_n}(x\opt)$, i.e., the projection of the optimizer $x\opt$ on the ball $\uball$ is in fact the projection onto the subspace $\X_{n}$. Therefore, thanks to the orthonormality, the projection residual $r_n = x\opt - \proj_{\X_{n}}(x\opt)$ belongs to the orthogonal complement $\comp{\X}_{n}$. Thus, following the same reasoning as in the proof of Theorem~\ref{thm:inf-semi}, one arrives at a bound similar to \eqref{inf-semi error} but using the restricted norms \eqref{norm-res}; recall that the norm in a Hilbert space is self-dual. 
		\end{proof}
		
		\subsection{Semi-infinite results in the MDP setting}
		\label{subsec:semi:MDP}
		We now return to the MDP setting in Section~\ref{sec:problem:statement}, and in particular the AC problem \eqref{AC-LP}, to investigate the application of the proposed approximation scheme. Recall that the AC problem \eqref{AC} can be recast in an LP framework in the form of \ref{primal-inf}, see \eqref{AC-setting}. To complete this transition to the dual pairs, we introduce the spaces 
		\begin{align} 
		\label{AC:pairs}
		\left\{\begin{array}{ll}
		\X = \R\times \Lip(S), & \C=\R\times\Meas(S), \\
		\B=\Lip(K), & \Y = \Meas(K), \\
		\cone=\Lip_+(K), & \cone^{*}=\Meas_+(K).
		\end{array}\right.
		\end{align}
		The bilinear form between each pair $(\X, \C)$ and $(\B,\Y)$ is defined in an obvious way (cf.\ \eqref{AC-setting}). The linear operator $\op:\X \ra \B$ is defined as $\op(\rho,u)(s,a) \Let -\rho - u(s) + Qu(s,a)$, and it can be shown to be weakly continuous \cite[p.~220]{ref:Hernandez-99}. On the pair $(\X, \C)$ we consider the norms
		\begin{subequations}
			\label{norm}
			\begin{align}
			\label{norm-1}
			\left\{
			\begin{array}{l}
			\|x\| = \|(\rho, u)\| = \max\big\{|\rho|, \|u\|_\lip \} = \max\big\{|\rho|, \|u\|_\infty, \sup_{s,s' \in S}{u(s) - u(s') \over \|s - s'\|_{\ell_\infty}}   \big\}, \vspace{2mm}\\
			\|c\|_* \Let \sup_{\|x\| \le 1} \inner{x}{c} =|c_1| + \sup_{\|u\|_\lip \le 1} \int_S u(s) c_2(\diff s) = |c_1| + \|c_2\|_{\wass}. 
			\end{array} \right. 
			\end{align}
			Recall that $\|\cdot\|_\lip$ is the Lipschitz norm on $\Lip(S)$ whose dual norm $\|\cdot\|_\wass$ in $\Meas(S)$ is known as the Wasserstein norm \cite[p.~105]{ref:Villani}. The adjoint operator $\op^{*}:\Y\to\C$ is given by $\op^{*}y(\cdot) \Let \big(-\inner{\ind}{y},-y(\cdot\times A)+y Q(\cdot)\big)$, where $\ind$ is the constant function in $\Lip(S)$ with value 1. In the second pair $(\B,\Y)$, we consider the norms
			\begin{align}
			\label{norm-2}
			\left\{
			\begin{array}{l}
			\|b\| = \|b\|_\lip \Let \max\big\{\|b\|_\infty, \sup_{k,k' \in K}{b(k) - b(k') \over \|k - k'\|_{\ell_\infty}}   \big\}, \vspace{2mm} \\
			\|y\|_* \Let \sup_{\|b\|_\lip \le 1} \inner{b}{y} = \|y\|_{\wass}. 
			\end{array} \right. 
			\end{align}
		\end{subequations}
		A commonly used norm on the set of measures is the total variation whose dual (variational) characterization is associated with $\|\cdot\|_\infty$ in the space of continuous functions \cite[p.~2]{ref:Hernandez-99}. We note that in the positive cone $\cone^* = \Meas_+(K)$ the total variation and Wasserstein norms indeed coincide. 
		
		Following the construction in \ref{primal-n}, we consider a collection of $n$-linearly independent, normalized functions $\{u_i\}_{i \le n}$, $\|u_i\|_\lip = 1$, and define the semi-infinite approximation of the AC problem \eqref{AC-LP} by
		\begin{align}
		\label{AC-LP-n} 
		-\Jacn = & \left\{ \begin{array}{ll}
		\inf\limits_{(\rho, \alpha)\in\R\times\R^n} & -\rho   \\
		\st &\rho + \sum\limits_{i=1}^{n} \alpha_i\big(u_i(s) - Qu_i(s,a)\big) \leq \cost(s,a), \quad \forall (s,a)\in K \\
		&  \| \alpha \|_\Rnorm \le \xnb 
		\end{array} \right. 
		\end{align}
		Comparing with the program \ref{primal-n}, we note that the finite dimensional subspace $\X_n \subset \R \times \Lip(S)$ is the subspace spanned by the basis elements $x_0 = (1,0)$ and $x_i = (0,u_i)$ for all $i \in \{1,\cdots,n\}$, i.e., the subspace $\X_n$ is in fact $n+1$ dimensional.  Moreover, the norm constraint in \eqref{AC-LP-n} is only imposed on the second coordinate of the decision variables $(\rho,\alpha)$ (i.e., $\|\alpha\|_\Rnorm \le \xnb$). The following lemmas address the operator norm and the respective regularity requirements of Assumption~\ref{a:reg} for the program \eqref{AC-LP-n}.

\begin{Lem}[MDP operator norm] \label{lem:operator}
	In the AC problem \eqref{AC-LP} under Assumption~\ref{a:CM}\ref{a:CM:Q} with the specific norms defined in \eqref{norm}, the linear operator norm satisfies $\|I - Q\| \Let \sup_{\|u\|_\lip \le 1} \|u - Qu\|_\lip \le 1 + \max\{L_Q,1\}$.
\end{Lem}

\begin{proof}
	Using the triangle inequality it is straightforward to see that
	\begin{align*}
	\| I-Q \| 	&= 		\sup\limits_{u\in\Lip(S)} \frac{\| u - Qu \|_\lip}{\|u\|_\lip } \le 1 + \sup\limits_{u\in \Lip(S)} {\| Qu \|_\lip \over \| u \|_\lip} \le 1 + \sup\limits_{u\in \Lip(S)} {\| Qu \|_\lip \over \| u \|_\infty}\\ 
	& \le 1 + \max\Big\{L_Q, \sup\limits_{u\in \Lip(S)} {\| Qu \|_\infty \over \| u \|_\infty} \Big\} \le 1 + \max\{L_Q,1\},
	\end{align*}
	where the second line is an immediate consequence of Assumption~\ref{a:CM}\ref{a:CM:Q} and the fact that the operator $Q$ is a stochastic kernel. Hence, $|Qu(s,a)|=| \int_{S} u(y) Q(\drv y|s,a)| \leq \| u \|_\infty (\int_{S} Q(\drv y|s,a)) = \| u \|_\infty$.
\end{proof}

\begin{Lem}[MDP semi-infinite regularity]\label{lem:MDP:bd-dual}
	Consider the AC program \eqref{AC-LP} under Assumption \ref{a:CM}. Then, Assumption~\ref{a:reg} holds for the semi-infinite counterpart in \eqref{AC-LP-n} for any positive $\xnb$ and all sufficiently large $\gamma$. In particular, the dual optimizer bound in Proposition~\ref{prop:SD} simplifies to $\|\yn\|_\wass \le \ynb = 1$.
\end{Lem}
		
\begin{proof}
Since $K$ is compact, for any nonnegative $\xnb$, the program \eqref{AC-LP-n} is feasible and the optimal value is bounded; recall that $\|(Q-I)u_i\|_\lip \le 1+\max\{L_Q,1\}$ from Lemma~\ref{lem:operator} and $\|\cost\|_\infty < \infty$ thanks to Assumption~\ref{a:CM}\ref{a:CM:cost}. Hence, the optimal value of \eqref{AC-LP-n} is bounded and, without loss of generality, one can add a redundant constraint $|\rho| \le \omega^{-1}\xnb$, where $\omega$ is a sufficiently small positive constant. In this view, the last constraint $\|\alpha\|_\Rnorm \le \xnb$ may be replaced with 
\begin{align}
\label{norm-AC}
	\|(\rho,\alpha)\|_{\omega} \Let \max\{\omega|\rho|,\|\alpha\|_\Rnorm\} \le \xnb, 
\end{align}
where $\|\cdot\|_\omega$ can be cast as the norm on the pair $(\rho,\alpha) \in \R\times\R^{n+1}$. Using the $\omega$-norm as defined in \eqref{norm-AC}, we can now directly translate the program~\eqref{AC-LP-n} into the semi-infinite framework of \ref{primal-n}. As mentioned above, the feasibility requirement in  Assumption~\ref{a:reg}\ref{a:reg:feas} immediately holds. In addition, observe that for every $y\in\cone^*$ we have 
\begin{align*}
	\|\opn^* y\|_{\omega^*} &= \sup_{\|(\rho,\alpha)\|_{\omega} \le 1} (\rho,\alpha) \bdot \big[-\inner{\ind}{y}, \inner{Qu_1-u_1}{y}, \cdots,\inner{Qu_n-u_n}{y}\big]\\	
	&= \sup_{\omega |\rho| \le 1}  - \rho\inner{\ind}{y} + \sup_{\|\alpha\|_\Rnorm\le 1} \alpha \bdot \big[\inner{Qu_1-u_1}{y}, \cdots,\inner{Qu_n-u_n}{y}\big] \\
	& \ge \omega^{-1}\|y\|_\wass,
\end{align*}
where the third line above follows from the equality $\inner{\ind}{y} = \|y\|_\wass$ for every $y$ in the positive cone $\cone^*$, and the fact that the second term in the second line is nonnegative. Since $\omega$ can be arbitrarily close to 0, the inf-sup requirement Assumption~\ref{a:reg}\ref{a:reg:inf-sup} holds for all sufficiently large $\gamma = \omega^{-1}$. 
The second assertion of the lemma follows from the bound \eqref{yb} in Proposition \ref{prop:SD}. To show this, recall that in the MDP setting $c = (-1,0) \in \R \times \Meas(S)$ (cf. \eqref{AC-setting}) with the respective vector $\cnew = [-1,0,\cdots,0] \in \R\times \R^{n}$ (cf. \ref{primal-n}). Thus, $\|\cnew\|_{\omega^*} = \sup_{\|(\rho,\alpha)\|_{\omega}\le 1} {(\rho,\alpha)} \bdot [-1,0,\cdots,0] = \omega^{-1}$, that helps simplifying the bound \eqref{yb} to
\begin{align*}
\|\yn\|_\wass \le \ynb \Let {\xnb\|\cnew\|_{\Rnorm^*} - \Jlb \over \gamma \xnb - \|b\|} = {\xnb \omega^{-1} + \|\cost\|_\infty  \over \omega^{-1} \xnb -\|\cost\|_\lip},  
\end{align*}
which delivers the desired assertion when $\omega$ tends to 0. 
\end{proof}
		
\begin{Rem}[AC dual optimizers bound]
	\label{rem:ynb:AC}
	As opposed to the general LP in Proposition~\ref{prop:SD}, Lemma~\ref{lem:MDP:bd-dual} implies that the dual optimizers for the AC problem is not influenced by the primal norm bound $\xnb$ and is uniformly bounded by $1$. In fact, this result can be strengthened to $\|\yn\|_\wass = 1$ due to the special minimax structure of the AC program~\eqref{AC-LP-n}. This refinement is not needed at this stage and we postpone the discussion to Section~\ref{subsec:smooth:MDP}. The feature discussed in this remark does, however, not hold for the class of long-run discounted cost problems, see Lemma~\ref{lem:MDP:DC} in Appendix~\ref{app:discouted:cost}. 
\end{Rem}
	
Now we are in a position to translate Theorem~\ref{thm:inf-semi} to the MDP setting for the AC problem \eqref{AC-LP}.

\begin{Cor}[MDP semi-infinite approximation]
	\label{cor:MDP:inf-semi}
	Let $\Jac$ and $u\opt$ be the optimal value and an optimizer for the AC program~\eqref{AC-LP}, respectively. Consider the semi-infinite program \eqref{AC-LP-n} where $\xnb > \|\cost\|_\lip$, and let $\Uball \Let \{\sum_{i=1}^{n} \alpha_i u_i ~:~ \|\alpha\|_\Rnorm \le \xnb \}$. Then, the optimal value of \eqref{AC-LP-n} satisfies the inequality
		\begin{align*}
			0 \le \Jac - \Jacn \le \big(1 + \max\{L_Q, 1\}\big) \big\|u\opt - \proj_{\Uball}(u\opt)\big\|_\lip.
		\end{align*}
\end{Cor}
		
		\begin{proof}
			We first note that the existence of the optimizer $u\opt$ is guaranteed under Assumption \ref{a:CM} \cite[Theorem~12.4.2]{ref:Hernandez-99}.
			The proof is a direct application of Theorem~\ref{thm:inf-semi} under the preliminary results in Lemma \ref{lem:MDP:bd-dual} and \ref{lem:operator}. Observe that the projection error is $r_n \Let (\rho\opt,u\opt) - \proj_{\Uball}(\rho\opt,u\opt) = \big(0, u\opt - \proj_{\Uball}(u\opt)\big)$, resulting in $\inner{r_n}{c} = 0$. Thanks to this observation Lemma~\ref{lem:operator}, the assertion of Theorem~\ref{thm:inf-semi} translates to 
			\begin{align*}
			0 \le \Jac - \Jacn & = \Jpn - \Jp \le \inner{r_n}{\op^*\yn - c} = \inner{\op r_n}{\yn} \le \|I-Q\|\, \|r_n\|_\lip \, \|\yn\|_\wass\\
			& \le (1 + \max\{L_Q, 1\}) \|u\opt - \proj_{\Uball}(u\opt)\|_\lip.  \qedhere
			\end{align*}
		\end{proof}
		
		Observe that if from the beginning we consider the norm $\|\cdot\|_\infty$ on the spaces $\X$ and $\B$, it is not difficult to see that the operator norm in Lemma~\ref{lem:operator} simplifies to $2$ (recall that $Q$ is a stochastic kernel). Thus, the semi-infinite bound reduces to $\Jac - \Jacn \le 2 \|u\opt - \proj_{\Uball}(u\opt)\|_\infty$. One may arrive at this particular observation through a more straightforward approach: Using the shorthand notation $(Q-I)u \Let Qu - u$, we have 
		\begin{align*}
		\Jac - \Jacn & \le \min_{k\in K} \Big( \big(Q-I\big)u\opt(k) + \cost(k) \Big) - \min_{k\in K} \Big( \big(Q-I\big)\proj_{\Uball}(u\opt)(k) + \cost(k) \Big)\\
		& \le \max_{k\in K} \big(Q-I\big)\big(u\opt - \proj_{\Uball}(u\opt)\big)(k) \le \big\|\big(Q-I\big) \big(u\opt - \proj_{\Uball}(u\opt)\big)\big\|_\infty \\
		& \le 2\big\|u\opt - \proj_{\Uball}(u\opt)\big\|_\infty.
		\end{align*}
		Theorem~\ref{thm:inf-semi} is a generalization to the above observation in two respects: 
		\begin{itemize}
			\item It holds for a general LP that, unlike the AC problem \eqref{AC-LP}, may not necessarily enjoy a min-max structure.
			\item The result reflects how the bound on the decision space (i.e., $\xnb$ in \ref{primal-n}) influences the dual optimizers as well as the approximation performance in generic normed spaces. 
		\end{itemize}
	The latter feature is of particular interest as the boundedness of the decision space is often an a~priori requirement for optimization algorithms, see for instance \cite{ref:nesterov-book-04} and the results in Section~\ref{sec:semi-fin:smoothing}. The approximation error from the original infinite LP to the semi-infinite version is quantified in terms of the projection residual of the value function. Clearly, this is where the choice of the finite dimensional ball $\Uball$ plays a crucial role. We close this section with a remark on this point. 
		
		\begin{Rem}[Projection residual]\label{rem:projection}
			The residual error $\big\|u\opt - \proj_{\Uball}(u\opt)\big\|_\lip$ can be approximated by leveraging results from the literature on universal function approximation. Prior information about the value function $u\opt$ may offer explicit quantitative bounds. For instance, for MDP under Assumption~\ref{a:CM} we know that $u\opt$ is Lipschitz continuous. For appropriate choice of basis functions, we can therefore ensure a convergence rate of ${n}^{-1/\dim(S)}$ where $\dim(S)$ is the dimension of the state-action set $S$, see for instance \cite{ref:Farouki-12} for polynomials and \cite{ref:Olver-09} for the Fourier basis functions. 
			
		\end{Rem}
		
		\section{Semi-infinite to Finite Programs: Randomized Approach} 
		\label{sec:semi-fin:rand}
		We study conditions under which one can provide a finite approximation to the semi-infinite programs of the form \ref{primal-n}, that are in general known to be computationally intractable --- NP-hard \cite[p.~16]{ref:BenTal-09}. We approach this goal by deploying tools from two areas, leading to different theoretical guarantees for the proposed solutions. This section focuses on a randomized approach and the next section is dedicated to an iterative gradient-based decent method. The solution of each of these methods comes with a~priori as well as a posteriori performance certificates. 
		
		\subsection{Randomized approach} 
		\label{subsec:rand}
		
		We start with a lemma suggesting a simple bound on the norm of the operator $\opn$ in \eqref{Ln}. We will use the bound to quantify the approximation error of our proposed solutions. 
		
		\begin{Lem}[Semi-infinite operator norm]\label{lem:opn}
			Consider the operator $\opn:\R^n \ra \B$ as defined in \eqref{Ln}. Then,
			\begin{align}\label{opt}
			\|\opn\| \Let \sup_{\alpha \in \R^n} {\|\opn \alpha\| \over \|\alpha\|_\Rnorm} \le \|\op\|\ratio, \qquad  \ratio \Let \sup_{\|\alpha\|_\Rnorm \le 1} \|\alpha\|_{\ell_1},
			\end{align}
			where the constant $\ratio$ is the equivalence ratio between the norms $\|\cdot\|_\Rnorm$ and $\|\cdot\|_{\ell_1}$.\footnote{The constant $\ratio$ is indexed by $n$ as it potentially depends on the dimension of $\alpha \in \R^n$. } 
		\end{Lem}
		
		\begin{proof}
			The proof follows directly from the definition of the operator norm, that is, 
			$$\|\opn \alpha\| = \Big\|\sum_{i=1}^{n}\alpha_i \op x_i\Big\| \le \|\op\|\Big\|\sum_{i=1}^{n}\alpha_i x_i \Big\|,$$
			together with the inequality $\big\|\sum_{i=1}^{n}\alpha_i x_i \big \| \le \|\alpha\|_{\ell_1} \max_{i \le n} \|x_i\| = \|\alpha\|_{\ell_1}$, which concludes the proof. 
		\end{proof}
		
		Since $\cone$ is a closed convex cone, then $\cone^{**} = \cone$ \cite[p.\ 40]{ref:Anderson-87}, and as such the conic constraint in program \ref{primal-n} can be reformulated as 
		\begin{align}
		\label{conic-const}
		\opn \alpha \geqc{\cone} b \qquad \Llra \qquad \inner{\opn \alpha - b}{y} \ge 0, \quad \forall y \in \Kb \Let \ext\{y \in \cone^* : \|y\|_* = 1 \}, 
		\end{align}
		where $\ext\{B\}$ denotes the extreme points of the set $B$, i.e., the set of points that cannot be represented as a strict convex combination of some other elements of the set. Notice that the norm constraint as well as the restriction to the extreme points in the definition of $\Kb$ in \eqref{conic-const} does not sacrifice any generality, as conic constraints are homogeneous. These restrictions are introduced to improve the approximation errors. In what follows, however, one can safely replace the set $\Kb$ with any subset of the cone $\cone^*$ whose closure contains $\Kb$. This adjustment may be taken into consideration for computational advantages. Let $\PP$ be a Borel probability measure supported on $\Kb$, and $\{y_j\}_{j\le N}$ be independent, identically distributed (i.i.d.) samples generated from $\PP$. Consider the \emph{scenario} counterpart of the program \ref{primal-n} defined as
		\begin{align} 
		\label{primal-n,N} 
		\tag{$\Prim_{n,N}$}
		\JpnN \Let 
		\left\{ \begin{array}{ll}
		\Min{\alpha\in\R^n} & \alpha \bdot \cnew   \\
		\st & \alpha \bdot \opn^* y_j \ge \inner{b}{y_j},  \quad  j \in \{1,\cdots,N\}\\
		& \|\alpha\|_{\Rnorm} \le \xnb,
		\end{array} \right.
		\end{align}
		where the adjoint operator $\opn^*:\B \ra \R^n$ is introduced in \eqref{Ln}. The optimization problem \ref{primal-n,N} is a standard finite convex program, and thus computationally tractable whenever the norm constraint $\|\alpha\|_\Rnorm \le \xnb$ is tractable. Program \ref{primal-n,N} is a relaxation of \ref{primal-n}, i.e., $ \Jpn \ge \JpnN$; note that $\JpnN$ is a random variable, therefore the relaxation error $\Jpn - \JpnN$ can only be interpreted in a probabilistic sense. 
		
		\begin{Def}[Tail bound]
			\label{def:tail}
			Given a probability measure $\PP$ supported on $\Kb$, we define the function $p:\R^n \times \R_+ \ra [0,1]$ as
			\begin{align*}
			p(\alpha,\zeta) \Let \PP\Big[y ~ : ~  \supp{\Kb}{-\opn \alpha + b}<\inner{-\opn \alpha + b}{y} + \zeta \Big],
			\end{align*}
			where $\sigma_{\Kb}(\cdot) \Let \sup_{y\in\Kb}\inner{\cdot}{y}$ is the support function of $\Kb$. We call $h:\R^n\times [0,1] \ra \R_+$ a \emph{tail bound} (TB) of the program \ref{primal-n,N}, if for all $\eps \in [0,1]$ and $\alpha$ we have	
			\begin{align*}
			h(\alpha,\eps) \ge \sup \big \{ \zeta ~:~ p(\alpha,\zeta) \le \eps \big \}.
			\end{align*}
		\end{Def}		
		
		The TB function in Definition~\ref{def:tail} can be interpreted as a {\em shifted} quantile function of the mapping $y \mapsto \inner{-\opn \alpha + b}{y}$ on $\Kb$--- the ``shift" is referred to the maximum value of the mapping which is $\supp{\Kb}{-\opn \alpha + b}$. 		
		TB functions depend on the probability measure $\PP$ generating the scenarios $\{y_j\}_{j\le N}$ in the program \ref{primal-n,N}, as well as the properties of the optimization problem. Definition \ref{def:tail} is rather abstract and not readily applicable. The following example suggests a more explicit, but not necessarily optimal, candidate for a TB.
		
		\begin{Ex}[TB candidate] \label{ex:TB}
			Let $g:\R_{+}\to [0,1]$ be a non-decreasing function such that for any $\kappa \in \Kb$ we have $g(\gamma) \le \PP\big[\ball{\kappa}{\gamma} \big]$, where $\ball{\kappa}{\gamma}$ is the open ball centered at $\kappa$ with radius $\gamma$; note that function $g$ depends on the choice of the norm on $\Y$. Then, a candidate for a TB function of the program \ref{primal-n,N} is 
			$$h(\alpha,\eps) \Let \| \opn \alpha - b \| g^{-1}(\eps) \le \big(\ratio \|\op\| \|\alpha\|_\Rnorm + \|b\|\big)g^{-1}(\eps),$$ 
			where the inverse function is understood as $g^{-1}(\eps) \Let \sup\{\gamma \in~\R_+~: g(\gamma) \le \eps\}$, and $\ratio$ is the constant ratio defined in \eqref{opt}. 
			
			To see this note that according to Definition~\ref{def:tail} we have 
			\begin{align*}
			p(\alpha,\zeta) 	&= \PP\Big[y ~ : ~  \sup_{\kappa\in\Kb}\inner{-\opn \alpha + b}{\kappa -y}<\zeta \Big] \\
			&= \inf_{\kappa\in\Kb}\PP\Big[y ~ : ~  \inner{-\opn \alpha + b}{\kappa -y}<\zeta \Big] \\
			&\geq \inf_{\kappa\in\Kb} \PP\Big[ y ~ : ~ \| \opn \alpha - b \| \| y-\kappa \|_{*} <\zeta \Big] \\
			&= \inf_{\kappa\in\Kb} \PP\Big[ \ball{\kappa}{\gamma(\zeta)} \Big] \geq g(\gamma(\zeta)), \quad \gamma(\zeta) \Let {\zeta \| \opn \alpha- b \|^{-1}}.
			\end{align*}
			Thus, if $p(\alpha,\zeta)\leq \eps$, then $g(\gamma(\zeta))\leq \eps$ and by construction of the inverse function $g^{-1}$ we have $\zeta \|\opn \alpha - b\|^{-1} \leq g^{-1}(\eps).$ In view of Definition \ref{def:tail}, this observation readily suggests that the function $h(\alpha,\varepsilon) \Let \| \opn \alpha - b \| g^{-1}(\eps)$ is indeed a TB candidate, and the suggested upper bound follows readily from Lemma~\ref{lem:opn}.
		\end{Ex}

		\begin{Thm}[Randomized approximation error]
			\label{thm:semi-fin:rand}
			Consider the programs \ref{primal-n} and \ref{primal-n,N} with the associated optimum values $\Jpn$ and $\JpnN$, respectively. Let Assumption \ref{a:reg} hold, $\alpha\opt_N$ be the optimizer of the program \ref{primal-n,N}, and the function $h$ be a TB as in Definition \ref{def:tail}. Given $\eps, \beta$ in $(0,1)$, we define 
			\begin{align} 
			\label{N}
			\NN(n, \eps, \beta) \Let \min \Big\{ N\in \N ~:~ \sum_{i=0}^{n-1}  {N \choose i} \eps^{i}(1-\eps)^{N-i}\leq\beta \Big\}.
			\end{align}
			For all positive parameters $\eps, \beta$ and $N \ge \NN(n,\eps,\beta)$ we have
			\begin{subequations}
				\label{eq:thm:semi-fin:rand}	
				\begin{align}
				\label{eq:thm:semi-fin:rand:1} 
				\PP^N \bigg[0\leq \Jpn - \JpnN \le \ynb h\big(\alpha\opt_N,\eps \big) \bigg] \ge 1-\beta,
				\end{align}
				where the constant $\ynb$ is defined as in \eqref{yb}. In particular, suppose the function $h$ is the TB candidate from Example \ref{ex:TB} with corresponding $g$ function, and 
				\begin{align}
				\label{NN}
				N \ge \NN\big(n,g(z_n\eps),\beta\big), \qquad z_n \Let \Big(\ynb\big(\xnb \ratio\|\op\| + \|b\|\big)\Big)^{-1}
				\end{align} 
				where $\ratio$ is the ratio constant defined in Lemma~\ref{lem:opn}. We then have 
				\begin{align}
				\label{eq:thm:semi-fin:rand:2} 	
				\PP^N \Big[0\leq \Jpn - \JpnN  \le  \eps \Big] \ge 1-\beta \,.
				\end{align}
			\end{subequations}
		\end{Thm}	
		
		Theorem~\ref{thm:semi-fin:rand} extends the result \cite[Theorem 3.6]{ref:MohSut-13} in two respects: 
		\begin{enumerate} [label=$\bullet$, itemsep = 1mm, topsep = -1mm]
			\item The bounds \eqref{eq:thm:semi-fin:rand} are described in terms of a generic norm and the corresponding dual optimizer bound.
			\item Through the optimizer of \ref{primal-n,N}, the bounds involve an a posteriori element (cf. \eqref{eq:thm:semi-fin:rand:1} to \eqref{eq:thm:semi-fin:rand:2}). 
		\end{enumerate}
		Before proceeding with the proof, we first remark on the complexity of the a~priori bound of Theorem~\ref{thm:semi-fin:rand}, its implications for an appropriate choice of $\xnb$, and its dependence on the dual pair norms. 
		
		\begin{Rem}[Curse of dimensionality]
			\label{rem:curse}
			The TB function $h$ of Example~\ref{ex:TB} may grow exponentially in the dimension of the support set $\Kb$ (i.e., $h(\alpha,\eps) \propto \eps^{-\dim(\Kb)}$). Since $\NN(n,\cdot,\beta)$ admits a linear growth rate, the a~priori bound \eqref{eq:thm:semi-fin:rand:2} effectively leads to an exponential number of samples in the precision level $\eps$, an observation related to the curse of dimensionality \cite[Remark~3.9]{ref:MohSut-13}. To mitigate this inherent computational complexity, one may resort to a more elegant sampling approach so that the required number of samples $\NN$ has a sublinear rate in the second argument, see for instance \cite{ref:NemShap-06}. 
		\end{Rem}
		
		\begin{Rem}[Optimal choice of $\xnb$]
			\label{rem:theta}
			In view of the a priori error in Theorem~\ref{thm:semi-fin:rand}, the parameter $\xnb$ may be chosen so as to minimize the required number of samples. To this end, it suffices to maximize $z_n$ defined in \eqref{NN} over all $\xnb> \|b\|\gamma^{-1}$, see Assumption~\ref{a:reg}\ref{a:reg:inf-sup}, where $\ynb$ is defined in \eqref{yb}. One can show that the optimal choice in this respect is analytically available as 
			\begin{align*}
			\xnb\opt \Let {\|b\| \over \gamma} + \sqrt{\Big({\|b\| \over \gamma} + {\|b\| \over \ratio\|\op\|} \Big)\Big({\|b\| \over \gamma}- {\Jlb \over \|\cnew\|_{\Rnorm*}} \Big)} \, ,
			\end{align*}
			where $\Jlb$ is a lower bound on the optimal value of \ref{primal-n} used in \eqref{yb}.
		\end{Rem}
		
		\begin{Rem}[Norm impact on finite approximation]
			\label{rem:norm-fin}
			Besides to what has already been highlighted in Remark~\ref{rem:norm-semi}, the choice of norms in the dual pairs of normed vector spaces also has an impact on the function $g^{-1}(\eps)$. More specifically, the stronger the norm in the space $\B$, the larger the balls in the dual space $\Y$, and thus the smaller the function $g^{-1}$. 
		\end{Rem}
		
		To prove Theorem~\ref{thm:semi-fin:rand} we need a few preparatory results.
		
		\begin{Lem}[Perturbation function]
			\label{lem:Lemma0}
			Given $\delta\in\B$, consider the $\delta$-\emph{perturbed} program of \ref{primal-n} defined as
			\begin{align} 
			\label{primal-n-pert} 
			\tag{$\Prim_n(\delta)$}
			\Jpnd \Let \left\{ \begin{array}{ll}
			\Inf{\alpha\in\R^n} & \alpha \bdot \cnew   \\
			\st & \opn \alpha \geqc{\cone} b - \delta \\
			& \|\alpha\|_\Rnorm \le \xnb.
			\end{array} \right.
			\end{align}
			Under Assumption \ref{a:reg}, we then have $\Jpn - \Jpnd \leq \inner{\delta}{\yn}$, where $\yn$ is an optimizer of \ref{dual-n}.
		\end{Lem}
		\begin{proof}
			For the proof we first introduce the dual program of \ref{primal-n-pert}:
			\begin{align} 
			\label{dual-n-pert} 
			\tag{$\Dual_n(\delta)$}
			\Jdnd \Let \left\{ \begin{array}{ll}
			\Sup{y} & \inner{b-\delta}{y} - \xnb\|\opn^* y-\cnew\|_{\Rnorm^*}  \\
			\st & y\in\cone^{*}.
			\end{array} \right.
			\end{align}
	We then have 
			\begin{align*}
			\Jpn - \Jpnd & = \Jdn - \Jpnd= \inner{b}{\yn} - \xnb\|\opn^* \yn-\cnew\|_{\Rnorm^*}  - \Jpnd \\ 
			&= \inner{\delta}{\yn} + \inner{b-\delta}{\yn}- \xnb\|\opn^* \yn-\cnew\|_{\Rnorm^*}  - \Jpnd \\
			&\leq \inner{\delta}{\yn} + \Jdnd - \Jpnd \leq \inner{\delta}{\yn},
			\end{align*}
			where the first line follows from the strong duality (gap-free) between \ref{primal-n} and \ref{dual-n} by Proposition \ref{prop:SD}. The third line is due to the fact that $\yn$ is a feasible solution of \ref{dual-n-pert}, and the last line follows from weak duality between \ref{primal-n-pert} and \ref{dual-n-pert}.
		\end{proof}
						
		\begin{Lem}[Perturbation error] 
			\label{lem:Lemma1}
			Let $\alpha\opt_N$ be an optimal solution of \ref{primal-n,N} and assume that $\delta\in\B$ satisfies the conic inequality $ \opn \alpha\opt_N \geqc{\cone} b - \delta$. Then, under Assumption \ref{a:reg}, we have $0\leq \Jpn - \JpnN \leq \inner{\delta}{\yn}$.
			\end{Lem}
		\begin{proof}
			The lower bound on $\Jpn -\JpnN$ is trivial since \ref{primal-n,N} is a relaxation of \ref{primal-n}. For the upper bound the requirement on $\delta$ in the program \ref{primal-n-pert} implies that $\alpha\opt_N$ is a feasible solution of \ref{primal-n-pert}. We then have $\JpnN \geq \Jpnd$, and thus $0\leq \Jpn - \JpnN \leq \Jpn - \Jpnd$. Applying Lemma~\ref{lem:Lemma0} completes the proof.
		\end{proof}
		
		The following fact follows readily from Definition \ref{def:tail}. 
		
		\begin{Lem}[TB lower bound] \label{lem:ccp}
			If $\alpha \in \R^n$ satisfies $\PP\left[ y ~ : ~ \inner{\opn \alpha-b}{y} < 0 \right] \leq \eps$, then for any TB function in the sense of Definition~\ref{def:tail} we have $\supp{\Kb}{-\opn \alpha + b}\leq h(\alpha,\eps)$.
		\end{Lem}
		
		\begin{proof}
			By the definition of the support function we can equivalently write 
			\begin{align*}
			p(\alpha,\zeta)= \PP \left[ y ~ : ~ \inner{\opn \alpha - b }{y} < \zeta - \supp{\Kb}{-\opn \alpha + b} \right].
			\end{align*}
			Now setting $\zeta = \supp{\Kb}{-\opn \alpha + b}$ in the above relation together with the assumption of Lemma~\ref{lem:ccp} yields $p(\alpha,\zeta) \le \eps$, which in light of a TB in Definition \ref{def:tail} suggests that $\supp{\Kb}{-\opn \alpha + b}\leq h(\alpha,\eps)$. 
		\end{proof}

		We follow our discussion with a result from randomized optimization in a convex setting.

		\begin{Thm}[Finite-sample probabilistic feasibility {\cite[Theorem 1]{ref:CamGar-08}}] \label{thm:Campi}
			Assume that the program \ref{primal-n,N} admits a unique minimizer $\alpha\opt_N$.\footnote{The uniqueness assumption may be relaxed at the expense of solving an auxiliary convex program, see \cite[Section 3.3]{ref:MohSut-13}.} If $N \ge \NN(n,\eps,\beta)$ as defined in \eqref{N}, then with confidence at least $1 - \beta$ (across multi-scenarios $\{y_j\}_{j\le N}\subset \Kb$) we have $\PP \big[ y ~ : ~ \inner{\opn\alpha_N - b}{y} < 0  \big] \leq \eps$.
		\end{Thm}
		
		We are now in a position to prove Theorem~\ref{thm:semi-fin:rand}.
		
		\begin{proof}[Proof of Theorem~\ref{thm:semi-fin:rand}]
			By definition of the support function we know that $\supp{\Kb}{\delta} = \supp{\conv(\Kb)}{\delta}$ where $\conv(\Kb)$ is the convex hull of $\Kb$. Recall that by definition of the set $\Kb$ in \eqref{conic-const}, we also have $y/\|y\|_* \in \conv(\Kb)$ for any $y \in \cone^*$. Thus, for any $\delta \in \B$ and $y \in \cone^*$ we have $\inner{\delta}{y} \le \|y\|_* \supp{\Kb}{\delta}$. This leads to 
			\begin{align*}
			0\leq \Jpn-\JpnN \leq \inner{-\opn\alpha\opt_N + b}{\yn} \leq \|\yn\|_{*} \supp{\Kb}{-\opn\alpha\opt_N + b} 
			\end{align*}
			where the second inequality is due to Lemma~\ref{lem:Lemma1} as $\delta=-\opn\alpha\opt_N + b$ clearly satisfies the requirements. By  Lemma~\ref{lem:ccp} and Theorem~\ref{thm:Campi}, we know that with probability at least $1-\beta$ we have $\supp{\Kb}{-\opn\alpha_N + b} \le h(\alpha_N,\eps)$, which in conjunction with the dual optimizer bound in Proposition~\ref{prop:SD} results in \eqref{eq:thm:semi-fin:rand:1}. Now using the TB candidate in Example \ref{ex:TB} immediately leads to the first assertion of  \eqref{eq:thm:semi-fin:rand:2}. Recall that the solution \ref{primal-n,N} obeys the norm bound $\|\alpha\opt_N\|_\Rnorm \le \xnb$. Thus, by employing the triangle inequality together with Lemma~\ref{lem:opn} we arrive at the second assertion \eqref{eq:thm:semi-fin:rand:2}. 
		\end{proof}
		
		Theorem~\ref{thm:semi-fin:rand} quantifies the approximation error between programs \ref{primal-n} and \ref{primal-n,N} probabilistically in terms of the TB functions as introduced in Definition \ref{def:tail}. The natural question is under what conditions can the proposed bound be made arbitrarily small. This question is intimately related to the behavior of TB functions. For the TB candidate proposed in Example \ref{ex:TB}, the question translates to when does the measure of a ball $\ball{\kappa}{\gamma} \subset \Kb$ have a lower bound $g(\gamma)$ uniformly away from $0$ with respect to the location of its center: The answer to this question also depends on the properties of the norm on $(\B,\Y,\|\cdot\|)$. A positive answer to this question requires that the set $\Kb$ can be covered by finitely many balls, indicating that $\Kb$ is indeed compact with respect to the (dual) norm topology. In the next subsection we study this requirement in more detail in the MDP setting.

		\subsection{Randomized results in the MDP setting} 
		\label{subsec:rand:MDP}
		We return to the MDP setting and discuss the implication of Theorem~\ref{thm:semi-fin:rand} as the bridge from the semi-infinite program \ref{primal-n} to the finite counterpart \ref{primal-n,N}. Recall the dual pairs of vector spaces setting in \eqref{AC:pairs} with the assigned norms \eqref{norm}. To construct the finite program \ref{primal-n,N}, we need to sample from the set of extreme points of $\Prob(K)$, i.e., the set of point measures
		\begin{align*}
		\Kb \Let \ext\big(\Prob(K)\big) = \big\{ \dir{(s,a)} : (s,a)\in K \big\},
		\end{align*}
		where $\dir{(s,a)}$ denotes a point probability distribution at $(s,a) \in K$. In this view, in order to sample elements from $\Kb$ it suffices to sample from the state-action feasible pairs $(s,a) \in K$. 
		
		\begin{Cor}[MDP finite randomized approximation error]
			\label{cor:adp:semi-finite}
			Let $\{(s_j,a_j) \}_{j\le N}$ be $N$ i.i.d.~samples generated from the uniform distribution on $K$. Consider the program 
			\begin{align}
			\label{AC-LP-n,N} 
			-\JacnN = & \left\{ \begin{array}{ll}
			\inf\limits_{(\rho, \alpha)\in\R^{n+1}} & -\rho   \\
			\st &\rho + \sum\limits_{i=1}^{n} \alpha_i\big(u_i(s_j) - Qu_i(s_j,a_j)\big) \leq \cost(s_j,a_j), \quad \forall j \in \{1,\cdots,N\}\\
			&  \| \alpha \|_{\Rnorm} \le \xnb.
			\end{array} \right. 
			\end{align}
			where the basis functions $\{u_i\}_{i \le n}$ introduced in \eqref{AC-LP-n} are normalized (i.e., $\|u_i\|_\lip = 1$). Let $L_Q$ be the Lipschitz constant from Assumption~\ref{a:CM}\ref{a:CM:Q}, and define the constant 
			\begin{align*}
			z_n \Let \big(\xnb \ratio (\max\{L_Q,1\}+ 1)+ \|\cost\|_\lip \big)^{-1},
			\end{align*}
			where $\ratio$ is the ratio constant introduced in \eqref{opt}. Then, for all $\eps, \beta$ in $(0,1)$ and $N \ge \NN\big(n+1,(z_n\eps)^{\dim(K)},\beta\big)$ defined in \eqref{N}, we have
			\begin{align*}
			\PP^N \Big[0\leq \JacnN - \Jacn \le   \eps \Big] \ge 1-\beta.
			\end{align*}
		\end{Cor}	
		
		\begin{proof}
			Let $(\rho\opt_N,\alpha\opt_N)$ be the optimal solution for \eqref{AC-LP-n,N}. Observe that in the MDP setting, Assumption~\ref{a:CM}\ref{a:CM:Q} implies
			\begin{align}
			\|\opn\alpha\opt_N - b\| & = \Big\|-\rho\opt_N+ \sum_{i=1}^{n} \alpha\opt_{N(i)}(Q-I)u_i + \cost\Big\|_\lip \le (\max\{L_Q,1\}+1)\Big\|\sum_{i=1}^{n}\alpha\opt_{N(i)} u_i \Big\|_\lip + \|-\rho\opt_N+\cost\|_\lip \notag \\
			\label{Lg}  
			& \le (\max\{L_Q,1\}+1)\xnb \ratio \big(\max_{i\le n}\|u_i\|_\lip\big) + \|\cost\|_\lip,
			\end{align}
			where the equality $\|-\rho\opt_N+\cost\|_{\lip} = \|\cost\|_{\lip}$ leading to \eqref{Lg} follows from the fact that $\cost$ and $\rho\opt$ are non-negative (note that $\alpha = 0, \rho = 0$ is a trivial feasible solution for \eqref{AC-LP-n,N}). In the second step, we propose a TB candidate in the sense of Definition~\ref{def:tail}. Note that for any $k, k' \in K$, by the definition of the Wasserstein norm we have $\|\dir{\{k\}} - \dir{\{k'\}}\|_\wass = \min\{1,\|k - k'\|_\infty\}$. Thus, generating samples uniformly from $K$ leads to 
			\begin{align}
			\label{balls:AC}
			\PP\big[ \ball{\kappa}{\gamma} \big] \geq \PP\big[ \ball{k}{\gamma} \big] \ge {\gamma}^{\dim(K)}, \qquad \forall \kappa \in \Kb, \quad \forall k \in K,
			\end{align}
			where, with slight abuse of notation, the first ball $\ball{\kappa}{\gamma}$ is a subset of the infinite dimensional space $\Y$ with respect to the dual norm $\|\cdot\|_\wass$, while the second ball $\ball{k}{\gamma}$ is a subset of the finite dimensional space $K$ whose respective norm is $\|\cdot\|_\infty$. The relation \eqref{balls:AC} readily suggests a function $g:\R_+ \ra [0,1]$ for Example~\ref{ex:TB}, which together with \eqref{Lg} and the fact that the basis functions are normalized, it yields
			\begin{align*}
			h(\alpha,\eps) \Let \|\opn \alpha - b\| g^{-1}(\eps) &  \le \big(\xnb\ratio(\max\{L_Q,1\}+1) + \|\cost\|_\lip \big)\eps^{{1/\dim K}}.
			\end{align*}			
			Recall from Lemma~\ref{lem:MDP:bd-dual} that the dual multiplier bound is $\ynb = 1$, and feasible solutions $\alpha$ is bounded by $\xnb$. Finally, note that the decision variable of the program \eqref{AC-LP-n,N} is the $n+1$ dimensional pair $(\rho,\alpha)$. Given all the information above, the claim then readily follows from the second result of Theorem~\ref{thm:semi-fin:rand} in \eqref{eq:thm:semi-fin:rand:2}.
		\end{proof}
				
To select $\xnb$, one may minimize the complexity of the a priori bound in Corollary~\ref{cor:adp:semi-finite}, which is reflected through the required number of samples. At the same time, the impact of the bound $\xnb$ on the approximation step from infinite to semi-infinite in Corollary~\ref{cor:MDP:inf-semi} should also be taken into account. The first factor is monotonically decreasing with respect to $\xnb$, i.e., the smaller the parameter $\xnb$, the lower the number of the required samples. The second factor is presented through the projection residual (cf. Remark~\ref{rem:projection}). Therefore, an acceptable choice of $\xnb$ is an upper bound for the projection error of the optimal solution onto the ball $\Uball$ uniformly in $n \in \N$, i.e., 
\begin{subequations}
	\label{theta*:MDP}
	\begin{align}
	\label{theta*:MDP1}
	\xnb \ge \sup\bigg\{\|\alpha\opt\|_\Rnorm ~:~ \proj_{\Uball}(x\opt) = \sum_{i=1}^{n} \alpha\opt_iu_i, \quad n \in \N \bigg\}. 
	\end{align}
The above bound may be available in particular cases, e.g., when $\|\cdot\|_\Rnorm = \|\cdot\|_{\ell_2}$ it yields the bound
	\begin{align}
	\label{theta*:MDP2}
	\|\alpha\opt\|_{\ell_2} = \sqrt{\int_{S} {u\opt}^2(s) \diff s} \le \|u\opt\|_\lip \le \max\{L_Q,1\}\|\cost\|_\infty, 
	\end{align} 
\end{subequations}	
where $L_Q$ is the Lipschitz constant in Assumptions~\ref{a:CM}\ref{a:CM:Q}. We note that the first inequality in \eqref{theta*:MDP2} follows since $S$ is a unit hypercube, and the second inequality follows from \cite[Lemma~2.3]{ref:Dufour-15}, see also \cite[Section~5]{ref:Dufour-15} for further detailed analysis.


		\section{Semi-infinite to Finite Program: Structural convex optimization} 
		\label{sec:semi-fin:smoothing}
		
		This section approaches the approximation of the semi-infinite program \ref{primal-n} from an alternative perspective relying on an iterative first order decent method. As opposed to the scenario approach presented in Section~\ref{sec:semi-fin:rand}, that is probabilistic and starts from the program \ref{primal-n}, the method of this section is deterministic and starts with the dual counterpart \ref{dual-n}, in particular a {\em regularized} version of whose solutions can be computed efficiently. It turns out that the regularized solution allows one to reconstruct a nearly feasible solution for both programs \ref{primal-n} and \ref{dual-n}, offering a meaningful performance bound for the approximation step from the semi-infinite program to a finite program.  
		
		\subsection{Structural convex optimization}
		\label{subsec:smooth}
		
		The basis of our approach is the fast gradient method that significantly improves the theoretical and, in many cases, also the practical convergence speed of the gradient method. The main idea is based on a well known technique of smoothing nonsmooth functions \cite{ref:Nest-05}. To simplify the notation, for a given $\xnb$ we define the sets
		\begin{align*}
		\Ab \Let \big\{ \alpha\in\R^{n} :  \| \alpha \|_\Rnorm \leq \xnb \big\}, \qquad
		\Yb \Let \bigg\{ y\in\cone^{*} : \|y\|_* \leq \ynb \bigg\},
		\end{align*}
		where $\ynb$ is the constant defined in \eqref{yb}. Recall that in the wake of Proposition~\ref{prop:SD} we know that the decision variables of the dual program \ref{dual-n} may be restricted to the set $\Yb$ without loss of generality. We modify the program \ref{dual-n} with a regularization term scaled with the non-negative parameter $\eta$ and define the \emph{regularized} program
		\begin{align}  \label{dual-n-eta}
		\tag{$\Dual_{n,\eta}$}
		\Jneta \Let \Sup{y\in\Yb} 	\Big\{  \inner{b}{y} - \xnb\|\opn^* y-\cnew\|_{\Rnorm^*}   - \eta d(y) \Big\}, 		
		\end{align}
		where the regularization function $d: \Yb \ra \R_{+}$, also known as the {\em prox-function}, is strongly convex. The choice of the prox-function depends on the specific problem structure and may have significant impact on the approximation errors. Given the regularization term $\eta$ and the parameter $\alpha \in \R^n$, we introduce the auxiliary quantity
		\begin{align} \label{yeta}
		\yeta(\alpha)\Let \arg\max_{y\in\Yb} \Big\{  \inner{b-\opn \alpha }{y}  - \eta d(y) \Big\}.
		\end{align}
		It is computationally crucial for the solution method proposed in this part that the prox-function allows us to have access to the auxiliary variable $\yeta(\alpha)$ for each $\alpha \in \R^n$. This requirement is formalized as follows.
		
		\begin{As}[Lipschitz gradient]
			\label{a:yeta}
			Consider the adjoint operator $\opn^*$ in \eqref{Ln} and the optimizer $\yeta(\alpha)$ of the auxiliary quantity \eqref{yeta}. We assume that for each $\alpha \in \Ab$ the vector $\opn^*\yeta(\alpha) \in \R^n$ can be approximated to an arbitrary precision, and the mapping $\alpha \mapsto \opn^*\yeta(\alpha)$ is Lipschitz continuous with a constant $\tfrac{L}{\eta}$, i.e.,  
			\begin{align*}
			\|\opn^*\yeta(\alpha) - \opn^*\yeta(\alpha')\|_{\Rnorm^*} \le {L \over \eta} \|\alpha - \alpha'\|_\Rnorm, \qquad \forall  \alpha, \alpha' \in \Ab. 
			\end{align*}
		\end{As}
		
		Let $\vartheta > 0 $ be the strong convexity parameter of the mapping $\alpha \mapsto \tfrac{1}{2}\|\alpha\|_\Rnorm^2$ with respect to the $\Rnorm$-norm. We then define the operator $\T :\R^n\times\R^n \ra \R^n$ as
		\begin{align}
		\label{T}
		\T(q,\alpha) \Let \arg\min_{\beta \in \Ab} \Big\{q \bdot \beta + {1 \over 2\vartheta} \|\beta - \alpha\|^2_\Rnorm \Big \}, 
		\end{align}
		More generally, a different norm can be used in the second term in \eqref{T} when $\vartheta$ is a different strong convexity parameter. However, we forgo this additional generality to keep the exposition simple. The operator $\T$ is defined implicitly through a finite convex optimization program whose computational complexity may depend on the $\Rnorm$-norm through the constraint set $\Ab$. For typical norms in $\R^n$ (e.g., $\|\cdot\|_{\ell_p}$) the pointwise evaluation of the operator $\T$ is computationally tractable. Furthermore, if $\|\cdot\|_\Rnorm = \|\cdot\|_{\ell_2}$, then the definition of \eqref{T} has an explicit analytical description for any pair $(q,\alpha)$ as follows.
		
		\begin{Lem}[Explicit description of {$\T$}]
			\label{lem:T}
			Suppose in the definition of the operator \eqref{T} the $\Rnorm$-norm is the classical $\ell_2$-norm. Then, the operator $\T$ admits the analytical description $\T(q,\alpha) = \xi\, (\alpha - q)$ where $\xi \Let \min\big\{1,\xnb\|q-\alpha\|_{\ell_2}^{-1}\big\}$.
		\end{Lem}
		
		\begin{proof}
			In case of the $\ell_2$-norm the strong convexity parameter is $\vartheta = 1$. Now using the classical duality theory, the objective function of \eqref{T} is equal to
			\begin{align}
			\label{pd}
			\min_{\|\beta\|_{\ell_2} \le \xnb } \Big\{q \bdot \beta + {1 \over 2} \|\beta - \alpha\|^2_{\ell_2}\Big\} & = \max_{\lambda \ge 0} \Big\{-\lambda\xnb^2 + \min_\beta \big\{ q \bdot \beta + {1 \over 2} \|\beta - \alpha\|^2_{\ell_2} + \lambda \|\beta\|^2_{\ell_2}\big\}\Big\} \\
			&\notag = \max_{\lambda \ge 0} \Big\{ - \lambda\xnb^2 + {q \bdot \alpha -\|q\|^2_{\ell_2} \over 1 + 2\lambda} + {\|q+2\lambda\alpha\|^2_{\ell_2} \over 2(1+2\lambda)^2} + {\lambda \|\alpha - q\|^2_{\ell_2} \over (1+2\lambda)^2} \Big\},
			\end{align}
			where the second equality in \eqref{pd} follows by substituting the explicit solution of the unconstrained inner problem described by
			\begin{align}
			\label{beta}
			\beta\opt(\lambda) \Let \arg\min_\beta \big\{ q \bdot \beta + {1 \over 2} \|\beta - \alpha\|^2_{\ell_2} + \lambda \|\beta\|^2_{\ell_2}\big\} = {\alpha - q \over 1 + 2\lambda}.
			\end{align}
			To find the optimal $\lambda$ in the right-hand side of \eqref{pd}, it suffices to set the derivative to zero, which yields $\lambda\opt \Let \tfrac{1}{2\xnb}\max\{\|\alpha - q\| - \xnb, 0\}$. By substituting $\lambda\opt$ in \eqref{beta}, we have an optimal solution $\beta\opt(\lambda\opt) = \xi(\alpha - q)$ that is feasible since $\|\beta\opt(\lambda\opt)\|_{\ell_2} \le \xnb$. By virtue of the equality in \eqref{pd}, this concludes the desired assertion.
		\end{proof}
		
		Algorithm~\ref{alg} exploits the information revealed under Assumption~\ref{a:yeta} as well as the operator $\T$ to approximate the solution of the program \ref{dual-n}. The following proposition provides explicit error bounds for the solution provided by Algorithm \ref{alg} after $k$ iterations. The result is a slight extension of the classical smoothing technique in finite dimensional convex optimization \cite[Theorem 3]{ref:Nest-05} where the prox-function is not necessarily uniformly bounded, a potential difficulty in infinite dimensional spaces. We address this difficulty by considering a growth rate for the prox-function $d$ evaluated at the optimal solution $\yeta$. We later show how this extension will help in the MDP setting. 
		
		\begin{algorithm}[t!]
			\caption{Optimal scheme for smooth convex optimization}
			\label{alg}
			\begin{algorithmic} 
				\State Choose some $w^{(0)} \in \Ab$
				\State {For $k \ge 0$ do}
				\begin{enumerate} [label=$(\roman*)$, itemsep = 1mm, topsep = 1mm, leftmargin = 2cm]
					
					\item [{\bf Step 1:}] Define $r^{(k)} \Let \frac{\eta}{L} \big(\cnew - \opn^*\yeta(w^{(k)})\big)$;
					
					\item [{\bf Step 2:}] Compute  $z^{(k)} \Let \T\big(\sum_{j=0}^{k} \frac{j+1}{2} r^{(j)}, 0\big), \quad \alpha^{(k)} \Let \T\big({1 \over \vartheta}r^{(k)}, w^{(k)}\big)$;
					
					\item [{\bf Step 3:}] Set $w^{(k+1)}=\frac{2}{k+3} z^{(k)} + \frac{k+1}{k+3} \alpha^{(k)}$.
				\end{enumerate}
			\end{algorithmic}
		\end{algorithm}

		\begin{Thm}[Smoothing approximation error] 
			\label{thm:semi-fin:smooth}
			Suppose Assumption \ref{a:yeta} holds with constant $L$ and $\vartheta$ is the strong convexity parameter in the definition of the operator $\T$ in \eqref{T}. Given the regularization term $\eta > 0$ and $k$ iterations of Algorithm \ref{alg}, we define 
			\begin{align*}
			\wh{\alpha}_\eta \Let \alpha^{(k)}, \qquad \wh{y}_\eta \Let \sum_{j=0}^{k}\frac{2(j+1)}{(k+1)(k+2)} \yeta(w^{(j)}).
			\end{align*}
			Under Assumption~\ref{a:reg}, the optimal value of the program \ref{primal-n} is bounded by $\Jnlb \le \Jpn \le \Jnub$ where 
			\begin{align}
			\label{Jn-algo}
			\Jnlb \Let \inner{b}{\wh{y}_\eta} - \xnb \|\opn^*\wh{y}_\eta -\cnew\|_{\Rnorm^*}, \qquad \Jnub \Let \wh{\alpha}_\eta \bdot \cnew + \sup_{y \in \Yb} \inner{b-\opn \wh{\alpha}_\eta}{y}	
			\end{align} 
			Moreover, suppose there exist positive constants $c, C$ such that
			$$C \max\big\{\log\big(c\eta^{-1}\big),1\big\} \ge d\big(\yeta(\alpha)\big), \qquad \forall \eta > 0, \quad \forall \alpha \in \Ab,$$
			and, given an a~priori precision $\eps>0$, the regularization parameter $\eta$ and the number of iterations $k$ satisfy 
			\begin{align}
			\label{eta-k}
			\eta \le {\eps \over 2C\max\{2\log(2cC\eps^{-1}),1\}}, \qquad k \ge 2\xnb \ratio \frac{\sqrt{CL\max\{2\log(2cC\eps^{-1}),1\}}}{\sqrt{\vartheta}~\eps}, 
			\end{align} 
			where $\ratio$ is the constant defined in \eqref{opt}. Then, after $k$ iterations of Algorithm~\ref{alg} we have $\Jnub - \Jnlb \le \eps$. 
		\end{Thm}
		
		\begin{proof} 
			Observe that the bounds $\Jnlb$ and $\Jnub$ in \eqref{Jn-algo} are the values of the programs \ref{dual-n} and \ref{primal-n} evaluated at $\wh{y}_\eta$ and $\wh \alpha_\eta$, respectively. As such, the first assertion follows immediately. Towards the second part, thanks to the compactness of the set $\Ab$, the strong duality argument of Sion's minimax theorem \cite{ref:Sion-58} allows to describe the program \ref{dual-n-eta} through
			\begin{align} 
			\notag
			\Jneta & \Let \sup_{y \in \Yb} \inner{b}{y} - \Big[ \sup_{\alpha \in \Ab} \inner{\opn \alpha}{y} - \alpha \bdot \cnew + \eta d(y)\Big] \\
			\notag 
			& = \inf_{\alpha \in \Ab} \alpha \bdot \cnew + \sup_{y\in\Yb}  \Big[ \inner{b - \opn \alpha}{y}  - \eta d(y) \Big]\\
			\label{eq:dual:sion}	
			& = \inf_{\alpha \in \Ab} \alpha \bdot \cnew + \inner{b - \opn \alpha}{\yeta(\alpha)}  - \eta d\big(\yeta(\alpha)\big), 
			\end{align}
			where the last equality follows from the definition in \eqref{yeta}. Note that the problem \eqref{eq:dual:sion} belongs to the class of smooth and strongly convex optimization problems, and can be solved using a fast gradient method developed by \cite{ref:Nest-05}. For this purpose, we define the function 
			\begin{align} 
			\label{eq:objective:smoothing}
			\phi_{\eta}(\alpha) &\Let \alpha \bdot \cnew + \inner{b - \opn \alpha}{\yeta(\alpha)}  - \eta d\big(\yeta(\alpha)\big).
			\end{align}
			Invoking similar techniques to \cite[Theorem~1]{ref:Nest-05}, it can be shown that the mapping $\alpha \mapsto \phi_\eta(\alpha)$ is smooth with the gradient $\nabla \phi_{\eta}(\alpha) = \cnew - \opn^* \yeta(\alpha)$. The gradient $\nabla \phi_\eta(\alpha)$ is Lipschitz continuous by Assumption \ref{a:yeta} with constant $\tfrac{L}{\eta}$. Thus, following similar arguments as in the proof of \cite[Theorem~3]{ref:Nest-05} we have
			\begin{align}
			\label{eq:bound}
			0\le \Jnub - \Jnlb \le {L\|\alpha\opt\|^2_\Rnorm \over \vartheta (k+1)(k+2)\eta} + \eta d\big(\yeta(\alpha^*)\big) \le {L(\xnb \ratio)^2\over \vartheta k^2\eta} + C\eta\max\big\{\log(c\eta^{-1}),1\big\}.
			\end{align}
			Now, it is enough to bound each of the terms in the right-hand side of the above inequality by $\tfrac{1}{2}\eps$. It should be noted that this may not lead to an optimal choice of the parameter $\eta$, but it is good enough to achieve a reasonable precision order with respect to $\eps$. To ensure $\eta\log(\eta^{-1}) \le \eps$ for an $\eps \in (0,1)$ , it is not difficult to see that it suffices to set $\eta \le \tfrac{\eps}{2\log (\eps^{-1})}$. In this observation if we replace $\eta$ and $\eps$ with $\tfrac{1}{c}\eta$ and $\tfrac{1}{2cC}\eps$, respectively, we deduce that the second term on the right-hand side in \eqref{eq:bound} bounded by $\tfrac{1}{2}\eps$. Thus, the desired assertion follows by equating the first term on the right-hand side in \eqref{eq:bound} to $\tfrac{1}{2}\eps$ while the parameter $\eta$ is set as just suggested. 
		\end{proof}
		
		\begin{Rem}[Computational complexity]
			\label{rem:complex}
			Adding the prox-function to the problem~\ref{dual-n} ensures that the regularized counterpart \ref{dual-n-eta} admits an efficiency estimate (in terms of iteration numbers) of the order $\order\big(\sqrt{\tfrac{L}{\eta}\eps^{-1}}\big)$. To construct a smooth $\eps$-approximation for the original problem \ref{dual-n}, the Lipschitz constant $\tfrac{L}{\eta}$ can be chosen of the order $\order({\eps^{-1}}{\log(\eps^{-1})})$. Thus, the presented gradient scheme has an efficiency estimate of the order $\order\big({\eps^{-1}}{\sqrt{\log(\eps^{-1})}}\big)$, see \cite{ref:Nest-05} for a more detailed discussion along similar objective.	
		\end{Rem}
		
		\begin{Rem}[Inexact gradient]
			\label{ref:inexact}
			The error bounds in Theorem~\ref{thm:semi-fin:smooth} are introduced based on the availability of the exact first-order information, i.e., it is assumed that at each iteration the vector $r^{(k)}$ that due to the bilinear form potentially involves a multi dimensional integration can be computed exactly. In general, the evaluation of those vectors may only be available approximately. This gives rise to the question of how the fast gradient method performs in the case of \emph{inexact} first-order information. We refer the interested reader to \cite{ref:Devolver-13} for further details.
		\end{Rem}
		
		The a priori bound proposed by Theorem~\ref{thm:semi-fin:smooth} involves the positive constants $c, C$, which are used to introduce an upper bound for the proxy-term. These constants potentially depend on $\ynb$, the size of the dual feasible set, hence also on $\xnb$. Therefore, unlike the randomized approach in Section~\ref{sec:semi-fin:rand}, it is not immediately clear how $\xnb$ can be chosen to minimize the complexity of the proposed method, which in this case is the required number of iterations $k$ suggested in \eqref{eta-k} (cf. Remark~\ref{rem:theta}). In the next section, we shall discuss how to address this issue in the MDP setting for particular constants $c,C$.
		
\subsection{Structural convex optimization results in the MDP setting}  
\label{subsec:smooth:MDP}
To link the approximation method presented in Section~\ref{subsec:smooth} to the AC program in \eqref{AC-LP-n}, let us recall the dual pairs \eqref{AC:pairs} equipped with the norms \eqref{norm}. To simplify the analysis, we refine the assertion in Lemma~\ref{lem:MDP:bd-dual} and argue that the dual optimizers are indeed probability measures, i.e.,
	\begin{align}
	\label{Y}
	\Yb \Let \bigg\{ y \in \Meas_+(K) \ : \|y\|_{\wass} = \ynb = 1 \bigg\}.
	\end{align}
To see this, one can consider the norm $\|(\rho,\alpha)\| \Let \|\alpha\|_{\Rnorm}$ and follow similar arguments in the proof of Proposition~\ref{prop:SD}. Strictly speaking, this is not a true norm on $\R^{n+1}$ but it does not affect the technical argument, in particular strong duality between \ref{primal-n} and \ref{dual-n}. The details are omitted here in the interest of space. We consider the prox-function as a relative entropy defined by
		\begin{align} 
		\label{d}
		d(y) \Let \left\{ 
		\begin{array}{cr} 
		\inner{\log\big(\frac{\diff y}{\diff\Leb}\big)}{y} & y \ll \Leb \\
		\infty & \text{o.w.},
		\end{array}\right.
		\end{align}
		where $\Leb$ is the uniform measure supported on the set $K$ and $\tfrac{\diff y}{\diff\Leb} \in \Func_+(K)$ is the Radon-Nikodym derivative between two measures $y$ and $\Leb$. One can inspect that the prox-function \eqref{d} is indeed a non-negative function. The optimizer of the regularized program \ref{dual-n-eta} for the AC program~\eqref{AC-LP-n} is
		\begin{align}
		\label{yeta:AC}
		\yeta(\rho,\alpha) \Let \arg\max_{y \in \Yb} \bigg\{ \inner{-\cost + \rho - \sum_{i=1}^{n}\alpha_i(Q-I)u_i}{y} - \eta \inner{\log\big(\tfrac{\diff y}{\diff\Leb}\big)}{y} \bigg\}.
		\end{align}
		To see \eqref{yeta:AC}, check \eqref{yeta} together with the definitions of the operator $\opn$ in \eqref{Ln} and the AC problem parameters in \eqref{AC-setting}. The main reason for such a choice of the regularization term is the fact that the optimizer of the regularized program \eqref{yeta:AC} admits an analytical expression:
		
		\begin{Lem}[Entropy maximization {\cite{ref:Csiszar-75}}] 
			\label{lem:entropy}
			Given a (measurable)  function $g:K\ra \R$ and the set $\Yb \subset \Meas_+(K)$ as defined in \eqref{Y} we have
			\begin{align*}
			y^\star(\diff k) \Let \arg\max_{y \in \Yb} \Big\{\inner{g}{y} - \eta d(y) \Big\} = \frac{\exp\big(\eta^{-1}g(k) \big) \Leb(\diff k)}{\inner{\exp\big(\eta^{-1}g(k)\big)}{\Leb}}.
			\end{align*}
		\end{Lem}
		
		Thanks to Lemma \ref{lem:entropy}, the analytical description of the dual optimizer in \eqref{yeta:AC} is readily available by setting
		\begin{align}
		\label{g}
		g(k) \Let [b-\opn \alpha](k) = -\cost(k) + \rho - \sum_{i=1}^n \alpha_i (Q-I)u_i(k).
		\end{align}		
		The last requirement to implement Algorithm~\ref{alg} is to verify Assumption~\ref{a:yeta}, i.e., we need to compute the Lipschitz constant of the mapping $(\rho,\alpha) \mapsto \opn^*\yeta(\rho,\alpha)$  in which the respective norm is $\|(\rho,\alpha)\| \Let \|\alpha\|_\Rnorm$. By definition of the adjoint operator $\opn^*$ in \eqref{Ln}, it is not difficult to observe that 
		\begin{align} 
		\label{Lny}
		\opn^* \yeta (\rho,\alpha) = 
		\left[ \begin{array}{c}
		\inner{-\ind}{\yeta (\rho,\alpha)} \\ \inner{(Q-I)u_1}{\yeta (\rho,\alpha)} \\ \vdots \\ \inner{(Q-I)u_n}{\yeta (\rho,\alpha)}
		\end{array} \right] 
		= 
		\left[ \begin{array}{c}
		-1\\ \inner{(Q-I)u_1}{\yeta (\rho,\alpha)} \\ \vdots \\ \inner{(Q-I)u_n}{\yeta (\rho,\alpha)}
		\end{array} \right].
		\end{align}
		The next lemma addresses the requirement of Assumption~\ref{a:yeta} for the mapping \eqref{Lny}. 
		
		\begin{Lem}[Lipschitz constant in MDP]
			\label{lem:Lip:AC}
			Consider the entropy maximizers in Lemma~\ref{lem:entropy} with $g$ as defined in \eqref{g} and the adjoint operator in \eqref{Lny}. An upper bound for the Lipschitz constant in Assumption~\ref{a:yeta} is $L \le 4\ratio^2$ where the constant $\ratio$ is the equivalence ratio between the norms $\|\cdot\|_\Rnorm$ and $\|\cdot\|_{\ell_1}$ introduced in \eqref{opt}. 
		\end{Lem}
		
		\begin{proof}
			It is straightforward to see that \eqref{Lny} is differentiable with respect to the variable $(\rho, \alpha)$. Hence, it suffices to bound the norm of the matrix $\nabla \opn^* \yeta (\rho,\alpha) \in \R^{(n+1)\times(n+1)}$ uniformly on $(\rho,\alpha)$. Further, as the first element of the vector \eqref{Lny} is the constant 1, it only requires ro consider the gradient function with respect to the variable $\alpha \in \R^n$. A direct computation yields 
			\begin{align*}
			|\big(\nabla_\alpha   \opn^* \yeta(\rho,\alpha) \big)_{ij} | &= \bigg| \frac{1}{\eta} \inner{(Q-I)u_i(Q-I)u_j}{\yeta(\rho,\alpha)} - \frac{1}{\eta}\inner{(Q-I)u_i}{\yeta(\rho,\alpha)} \inner{(Q-I)u_j}{\yeta(\rho,\alpha)} \bigg|\\
			& = {4 \over \eta } \bigg | \inner{{(Q-I)u_i \over 2}{(Q-I)u_j \over 2}}{\yeta(\rho,\alpha)} - \inner{{(Q-I)u_i \over 2}}{\yeta(\rho,\alpha)} \inner{{(Q-I)u_j \over 2}}{\yeta(\rho,\alpha)} \bigg| \\
			& \le  {4 \over \eta}, \qquad  \forall i,j \in \{1,\cdots,n\}.
			\end{align*}
			where the last inequality is a consequence of the Cauchy-Schwarz inequality and the fact that $\|(Q-I)u_j\|_\infty \le 2$ (recall that $Q$ is a stochastic kernel and all the basis functions are normalized). The Lipschitz constant of the desired mapping can then be upper bounded by
			\begin{align} 
			\label{grad-norm}
			{L \over \eta} & \leq \sup_{\tiny \begin{array}{cc} \|\alpha\|_\Rnorm \le \xnb \\ \|v\|_\Rnorm \le 1 \end{array}} \big\| \nabla_\alpha  \opn^* \yeta (\rho,\alpha) v \big\|_{\Rnorm*} 
			\le \sup_{\tiny \begin{array}{cc} \|\Phi_i\|_{\ell_\infty} \le 1 \\ \|v\|_\Rnorm \le 1 \end{array}}   {4 \over \eta} \Big \| (\Phi_1 \bdot v, \cdots, \Phi_n \bdot v) \Big \|_{\Rnorm^*}
			\end{align}	
			Recall that by the definition of the dual norm, we have $|\Phi_i \bdot v| \le \|\Phi_i\|_{\Rnorm^*}$ for all $\|v\|_{\Rnorm}\le 1$. Thus, substituting the scalar variable $\mu_i \Let \Phi_i \bdot v$ in the right-hand side of \eqref{grad-norm} and eliminating the factor $\eta$ lead to
			\begin{align*}
			L & \le \sup_{\|\Phi_i\|_{\ell_\infty} \le 1}  ~ \sup_{|\mu_i| \le \|\Phi_i\|_{\Rnorm^*}} {4} \big \| (\mu_1, \cdots, \mu_n) \big \|_{\Rnorm^*} = \sup_{|\mu_i| \le \ratio} {4} \big \| (\mu_1, \cdots, \mu_n)\big \|_{\Rnorm^*},
			\end{align*}
			where the last statement follows from the definition of the dual norm, and in particular the equality $$\sup_{\|\Phi_i\|_{\ell_\infty}\le 1} \|\Phi_i\|_{\Rnorm^*} = \sup_{\|\Phi_i\|_{\Rnorm}\le 1} \|\Phi_i\|_{\ell_1} =: \ratio.$$ Thus, using the same equality yields
			\[ L \le \sup_{\|\mu\|_{\ell_\infty} \le 1} {4 \ratio} \|\mu \|_{\Rnorm^*} = 4\ratio^2, \]
			which concludes the first desired assertion. 
		\end{proof}
		
		The performance of Algorithm~\ref{alg} can now be characterized through the following corollary. 
		
		\begin{Cor}[MDP smoothing approximation error]
			\label{cor:adp:smooth}
			Consider the operator \eqref{Ln} with the parameters described in \eqref{AC-setting} for the semi-infinite AC program \eqref{AC-LP-n}. Given this setting and the Lipschitz constant in Lemma~\ref{lem:Lip:AC}, we run Algorithm~\ref{alg} for $k$ iterations using the entropy function \eqref{d} with analytical solution \eqref{yeta:AC} as the prox-function. We define the constants 
			\begin{align*}
			C_1 \Let 2\e \big(\ratio \xnb (\max\{L_Q,1\}+1) + \|\cost\|_\lip\big), \qquad 
			C_2 \Let 4 \xnb \ratio^2 \sqrt{2\dim(K) \over \vartheta}.  
			\end{align*}
			For every $\eps \le C_1$ we set the smoothing factor $\eta$ and the number of iterations $k$ by
			\begin{align*}
			\eta \le {\eps \over 4 \dim(K)\log(C_1\eps^{-1})}, \qquad k \ge C_2 \frac{\sqrt{\log(C_1\eps^{-1})}}{\eps}.
			\end{align*} 
			Then, the outcome of Algorithm~\ref{alg} as defined in \eqref{Jn-algo} is an $\eps$ approximation of the optimal value $\Jacn$ in the sense of Theorem~\ref{thm:semi-fin:smooth}.
		\end{Cor}
		
		Corollary~\ref{cor:adp:smooth} requires one to compute the constants $c, C$ to quantify the a~priori bounds. The following two technical lemmas provide supplementary materials to address this issue.
		
		\begin{Lem}
			\label{lem:exp}
			Let $K \subseteq [0, 1]^m$ and $g: K \ra \R$ be a Lipschitz continuous function with constant $L_g > 0$ (with respect to the $\ell_\infty$-norm) and the maximum value $\gmax \Let \max_{k \in K} g(k)$. Then, for every $\eta > 0$ we have 
			\begin{align*}
			\int_{K} \exp\Big({\eta^{-1}\big(g(k) - \gmax\big)}\Big)\diff k \ge \min\Big\{\Big({m \eta \over L_g}\Big)^m, 1\Big\} \exp \big( - \min\big\{m, L_g  \eta^{-1}\big\} \big).
			\end{align*}
		\end{Lem}
		
		\begin{proof}
			Let us define the set $Z(\delta) \Let \{ k \in K : \gmax - g(k) < \delta\}$. Thanks to the Lipschitz continuity of the function $g$, we have $\gmax - g(k) \le L_g \|k\opt - k\|_{\ell_\infty}$ where $g(k\opt) = \gmax$. Thus, using this inequality one can bound the size of the set $Z(\delta)$ in the sense of 
			$$\int_{Z(\delta)} \diff k \ge \min\{(\delta L_g^{-1})^{m}, 1\}, \qquad \forall \delta \ge 0.$$
			By virtue of the above result, one can observe that for every $\delta > 0$
			\begin{align*}
			\int_{K} \exp\Big({\eta^{-1}\big(g(k) - \gmax\big)}\Big)\diff k & \ge \int_{Z(\delta)} \exp\Big({\eta^{-1}\big(g(k) - \gmax \big)}\Big)\diff k \\
			& \ge \exp(-\eta^{-1}\delta) \int_{Z(\delta)} \diff k \ge \exp(-\eta^{-1}\delta)\min\{(\delta L_g^{-1})^{m},1\}.
			\end{align*}
			Maximizing the right-hand side of the above inequality over $\delta$ suggests to set $\delta = \min\{m \eta, L_g\}$, which yields the desired assertion.
		\end{proof}
		
		In light of Lemma~\ref{lem:exp}, we can bound the entropy prox-function \eqref{d} evaluated at the optimizer \eqref{yeta:AC}. 
		
		\begin{Lem}[Entropy prox-bound]
			\label{lem:entropy:bound}
			Consider the prox-function \eqref{d} and let $\yeta(\rho,\alpha)$ be the optimizer of \eqref{yeta:AC}. Then, for every $\eta > 0$, $\rho$, and $\|\alpha\|_\Rnorm \le \xnb$, we have $d\big(\yeta(\rho, \alpha)\big) \le C\max\big\{\log(c\eta^{-1}),1\big\}$ where
			\begin{align*}
			\qquad C \Let \dim(K), \qquad c \Let {\e \over \dim(K)} \big(\xnb \ratio (\max\{L_Q,1\}+ 1)+ \|\cost\|_\lip \big),
			\end{align*}
			and $\ratio$ is the equivalence ratio between the norms $\|\cdot\|_{\ell_1}$ and $\|\cdot\|_\Rnorm$ as defined in \eqref{opt}. 
		\end{Lem}
		
		\begin{proof}
			The result is a direct application of Lemma~\ref{lem:exp}. Consider the function $g$ as defined in \eqref{g} with Lipschitz constant $L_g \ge 0$; note that the function $g$, as well as its Lipschitz constant $L_g$, depends also on the pair $(\rho,\alpha)$. Observe that 
			\begin{align}
			d\big(\yeta(\rho,\alpha)\big) &=  \inner{\log\big(\exp(\eta^{-1}g)\big)}{\yeta(\rho,\alpha)} - \log\big(\inner{\exp(\eta^{-1}g)}{\Leb}\big)\notag\\
			& = \inner{\eta^{-1}g}{\yeta(\rho,\alpha)} - \log\big(\inner{\exp(\eta^{-1}g)}{\Leb}\big)\notag\\
			& = \inner{\eta^{-1}g}{\yeta(\rho,\alpha)} - \eta^{-1}\gmax - \log\big(\inner{\exp(\eta^{-1}(g - \gmax)}{\Leb}\big) \notag\\
			&\le -\log\big(\inner{\exp(\eta^{-1}(g-\gmax))}{\Leb}\big) \notag\\
			& \le - \log\bigg( \min\Big\{\Big({\dim(K) \eta \over L_g}\Big)^{\dim(K)}, 1\Big\} \exp \big( - \min\big\{\dim(K), L_g  \eta^{-1}\big\} \big)\bigg) \label{eq:lemma:exp} \\
			& \le\dim(K)\max\Big\{\log\Big(\Big({\e L_g\over \dim(K)}\Big)\eta^{-1}\Big),1\Big\}\notag
			\end{align}
			where the inequality \eqref{eq:lemma:exp} follows from Lemma~\ref{lem:exp}. Note also that the Lipschitz constant $L_g$ for the function $g$ defined in \eqref{g} is upper bounded, uniformly in $(\rho,\alpha)$ where $\|\alpha\|_\Rnorm \le \xnb$, by
			\begin{align*}
			L_g \le \|g - \rho\|_{\lip} &\le \Big\| \sum_{i=1}^{n} \alpha_i(Q-I)u_i + \cost\Big\|_\lip \le (\max\{L_Q,1\}+1)\Big\|\sum_{i=1}^{n}\alpha_i u_i \Big\|_\lip + \|\cost\|_\lip \\
			&\le \xnb \ratio (\max\{L_Q,1\}+ 1)+ \|\cost\|_\lip.
			\end{align*} 
			We refer to the proof of Corollary~\ref{cor:adp:semi-finite}, and in particular the paragraph following \eqref{Lg}, for further discussions regarding $L_g$. The desired assertion follows from the last two inequalities and the definition of the constant $\ynb$ in \eqref{Y}.
		\end{proof} 	
		
		The proof of Corollary~\ref{cor:adp:smooth} follows by replacing the constants in Lemma~\ref{lem:entropy:bound} in Theorem~\ref{thm:semi-fin:smooth}. By contrast to the randomized approach in Corollary~\ref{cor:adp:semi-finite} where the computational complexity scales exponentially in dimensional of state-action space, the complexity of the smoothing technique grows effectively linearly (more precisely $\order\big(\eps^{-1}\sqrt{\log(\eps^{-1})}\big)$, cf. Remark~\ref{rem:complex}). The computational difficulty is, however, transferred to Step 1 of Algorithm~\ref{alg} for computation of $\opn^* \yeta$ as defined in \eqref{Lny}. The following remark elaborates this.
		
		\begin{Rem}[Efficient computation of \eqref{Lny}]
			When the transition kernel $Q$ and the basis functions $u_{i}$ are such that the relation \eqref{Lny} involves integration of exponentials of polynomials over simple sets (e.g., box or a simplex), one may utilize efficient methods that require solving a hierarchy of semidefinite programming problems to generate upper and lower bounds which asymptotically converge to the true value of integral, see \cite[Section~12.2]{ref:Lasserre-11} and \cite{ref:Bertsimas-08}. It is also worth noting that a straightforward computation of \eqref{Lny} for a small parameter $\eta$ may be numerically difficult due to the exponential functions. This issue can, however, be circumvented by a numerically stable technique presented in \cite[p.~148]{ref:Nest-05}. 
		\end{Rem}
		
		Regarding the choice of $\xnb$, in similar spirit to Section~\ref{sec:semi-fin:rand}, one can target minimizing the complexity of the a priori bound, in other words the number of iterations $k$ in \eqref{eta-k}. In the setting of  Corollary~\ref{cor:adp:smooth}, one can observe that the smaller the parameter $\xnb$, the lower the number of the required iterations, leading to the choice described as in \eqref{theta*:MDP}.				
		
		\section{Full Infinite to Finite Programs} 
		\label{sec:inf-fin}
		
		The intention in this short section is to combine the two-step process from infinite to semi-infinite programs in Section~\ref{sec:inf-semi} and from semi-infinite to finite programs in Section~\ref{sec:semi-fin:rand} and \ref{sec:semi-fin:smoothing}, and hence establish a link from the original infinite program to finite counterparts. We only present the final result for the general infinite programs without discussing its implication in the MDP setting, as it is essentially a similar assertion. 
		
		\begin{Thm}[Infinite to finite approximation error]
			\label{thm:inf-fin}
			Consider the infinite program \ref{primal-inf} with a solution $\{x\opt,\Jp\}$, the finite (random) convex program \ref{primal-n,N} with the (random) solution $\big\{\alpha\opt_N,\JpnN\big\}$, and the output of Algorithm~\ref{alg} with values $\big\{\Jnlb, \Jnub\big\}$. Suppose Assumption~\ref{a:reg} holds and assume further that there exists constant $d, D$ so that the projection residual of the optimizer $x\opt$ onto the finite dimensional ball defined in Theorem~\ref{thm:inf-semi} is bounded by $\|r_n\| \le D n^{-1/d}$ for all $n\in\N$. Then, for any number of scenario samples $N$ and prox-term coefficient $\eta$, with probability $1-\beta$ we have 
			\begin{align*}
			\max\big\{\JpnN,\Jnlb\big\} - D\big(\|c\|_* + \ynb\|\op\|\big)n^{-1/d} \le \Jp \le \min\big\{\Jnub ~,~\JpnN+ \ynb h(\alpha\opt_N,\eps) \big\}.
			\end{align*}
			where $\ynb$ is as defined in \eqref{yb} and the function $h$ is a TB in the sense of Definition~\ref{def:tail}. Moreover, given an a priori precision level $\eps$, if 
			$$n \ge \Big(D\big(\|c\|_* + \ynb\|\op\|\big)\eps^{-1}\Big)^{d},$$
			and the number samples $N$ are chosen as in \eqref{NN} {\em or} the parameter $\eta$ together with the number of iterations of Algorithm~\ref{alg} is chosen as in \eqref{eta-k}, then with probability $1-\beta$ we have 
			\begin{align*}
			\min\Big\{|\Jp - \JpnN|, |\Jp - \Jnlb| \Big\} \le \eps \,.
			\end{align*}
		\end{Thm}
		
		The proof follows readily from the link between the infinite program~\ref{primal-inf} to the semi-infinite counterpart~\ref{primal-n} in Theorem~\ref{thm:inf-semi}, in conjunction with the link between \ref{primal-n} to the finite programs~\ref{primal-n,N} and \ref{dual-n-eta} in Theorems~\ref{thm:semi-fin:rand} and \ref{thm:semi-fin:smooth}, respectively.
		
		The assertion of Theorem~\ref{thm:inf-fin} can be readily translated into the MDP problem by replacing the dual optimizer bound $\ynb$ with $1$ thanks to Lemma~\ref{lem:MDP:bd-dual}, and the term $(\|c\|_* + \ynb\|\op\|)$ with $(1+\max\{L_Q,1\})$ thanks to Corollary~\ref{cor:MDP:inf-semi}. In this case, the requirement concerning the projection residual bound $\|r_n\|\le Dn^{-1/d}$ is fulfilled due to the Lipschitz continuity of the value function when $d = \dim(S)$ and the finite dimensional approximation is generated by, among others, polynomials \cite{ref:Farouki-12} or the Fourier basis \cite{ref:Olver-09} (cf. Remark~\ref{rem:projection}).
		
\section{Numerical Examples} 
\label{sec:sim}
We present two numerical examples to illustrate the solution methods and corresponding performance bounds. Throughout this section we consider the norm $\|\cdot\|_\Rnorm = \|\cdot\|_{\ell_2}$, leading to $\ratio = \sqrt{n}$ in \eqref{opt}, and we choose the Fourier basis functions. 

\subsection{Example 1: truncated LQG} 
\label{ex:LQG}
Consider the linear system
	\begin{align*}
	s_{t+1}=\vartheta s_{t} + \rho a_{t} + \xi_{t}, \quad t \in \N,
	\end{align*}
with quadratic stage cost $\psi(s,a)=qs^{2}+ra^{2}$, where $q\geq 0$ and $r>0$ are given constants. We assume that $S=A=[-L,L]$ and the parameters $\vartheta, \rho \in\R$ are known. The disturbances $\{\xi_{t}\}_{t\in\N}$ are i.i.d.\ random variables generated by a truncated normal distribution with known parameters $\mu$ and $\sigma$, independent of the initial state $s_{0}$. Thus, the process $\xi_{t}$ has a distribution density
	\begin{align*}
	f(s,\mu,\sigma,L) = \left\{ 
	\begin{array}{cc} 
	\frac{\frac{1}{\sigma}\phi\left( \frac{s-\mu}{\sigma} \right)}{\Phi\left( \frac{L-\mu}{\sigma} \right)-\Phi\left( \frac{-L-\mu}{\sigma} \right)}, \quad & s\in[-L,L]\\
	0 & \text{o.w.},
	\end{array} \right.
	\end{align*}
where $\phi$ is the probability density function of the standard normal distribution, and $\Phi$ is its cumulative distribution function. The transition kernel $Q$ has a density function $q(y|s,a)$, i.e., $Q(B|s,a)=\int_{B}q(y|s,a)\drv y$ for all $B\in\Borel{S}$, that is given by
	\begin{align*}
	q(y|s,a)=f(y-\vartheta s - \rho a,\mu,\sigma,L).
	\end{align*}
In the special case that $L=+\infty$ the above problem represents the classical LQG problem, whose solution can be obtained via the algebraic Riccati equation \cite[p.~372]{ref:Bertsekas-12}. By a simple change of coordinates it can be seen that the presented system fulfills Assumption~\ref{a:CM}. The following lemma provides the technical parameters required for the proposed error bounds.
		
\begin{Lem}[Truncated LQG properties] \label{lem:LQG}
The error bounds provided by Corollaries~\ref{cor:adp:semi-finite} and \ref{cor:adp:smooth} hold with the norms $\|\psi\|_\infty = L^2(q+r)$, $\|\psi\|_\lip = 4L^2\sqrt{q^2+r^2}$, and the Lipschitz constant of the kernel is
	\begin{align*} 
	L_Q &=\frac{2L\max\{\vartheta,\rho\}}{\sigma^2\sqrt{2\pi}\left( \Phi\left( \frac{L-\mu}{\sigma} \right)-\Phi\left( \frac{-L-\mu}{\sigma} \right)\right)}\,.
	\end{align*}
\end{Lem}

\begin{proof}
	In regard to Assumption~\ref{a:CM}\ref{a:CM:K}, we consider the change of coordinates $\bar{s}_t \Let \frac{s_t}{2L}+\frac{1}{2}$ and $\bar{a}_t \Let \frac{a_t}{2L}+\frac{1}{2}$. In the new coordinates, the constants of Lemma~\ref{lem:LQG} follow from a standard computation.
\end{proof}

\paragraph{\emph{Simulation details:}}
For the simulation results we choose the numerical values $\vartheta = 0.8$, $\rho = 0.5$, $\sigma = 1$, $\mu = 0$, $q = 1$, $r = 0.5$, and $L=10$. In the first approximation step discussed in Section~\ref{subsec:semi:MDP}, we consider the Fourier basis $u_{2k-1}(s) = \frac{L}{k\pi}\cos\left(\frac{k \pi s}{L}\right)$ and $u_{2k}(s) = \frac{L}{k\pi}\sin\left(\frac{k \pi s}{L}\right)$. 
	
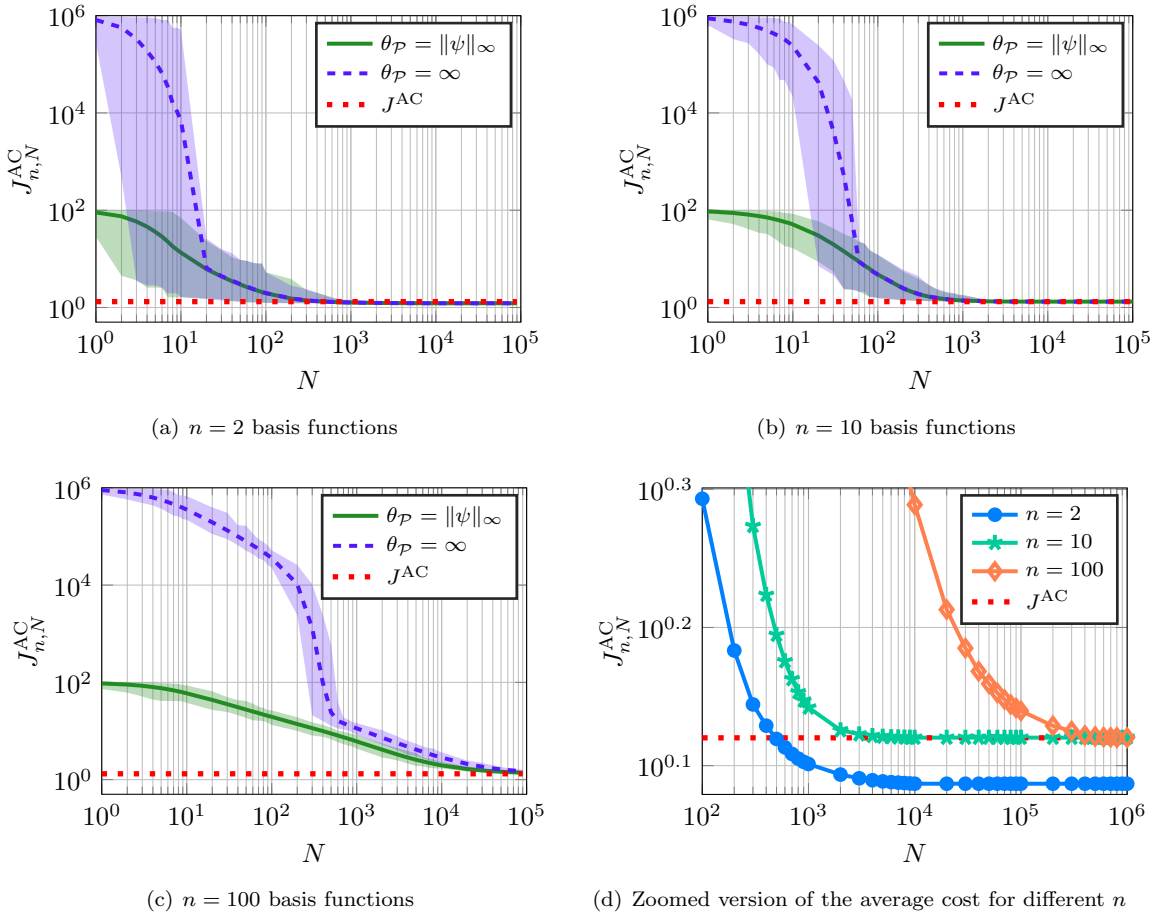
\begin{figure}[t]
	\subfigure[$n=2$ basis functions]{\scalebox{1}{
%
%
\definecolor{mycolor1}{rgb}{0.13, 0.55, 0.13}%
\definecolor{mycolor2}{rgb}{0.13, 0.55, 0.13}%

\definecolor{mycolor3}{rgb}{0.32,0.09,0.98}%
\definecolor{mycolor4}{rgb}{0.32,0.09,0.98}%
\begin{tikzpicture}

\begin{axis}[%
width=2.2in,
height=1.6in,
at={(1.011111in,0.813889in)},
scale only axis,
xmode=log,
xmin=1,
xmax=100000,
xminorticks=true,
xlabel={$N$},
xmajorgrids,
xminorgrids,
ymode=log,
ymin=0.5,
ymax=1000000,
ylabel={$J^{\text{AC}}_{n,N}$},
ymajorgrids,
ylabel style={yshift=-0.3cm},
legend style={legend cell align=left,align=left,draw=white!15!black,line width=1.0pt,font=\footnotesize}
]

\addplot[area legend,solid,fill=mycolor1,opacity=0.3,draw=none,forget plot]
table[row sep=crcr] {%
x	y\\
1	28.5805565420901\\
2	4.47632011102824\\
3	3.63625813510152\\
4	2.57054306764981\\
5	2.5705430451617\\
6	2.51737844086245\\
7	1.62086718607834\\
8	1.62086707873988\\
9	1.62086722353572\\
10	1.62086742950948\\
20	1.4816911948426\\
30	1.48169119341355\\
40	1.33395592104315\\
50	1.33395591974424\\
60	1.26186486900902\\
70	1.25936985320541\\
80	1.24252029228155\\
90	1.24252030145971\\
100	1.24177255933975\\
200	1.22146003433544\\
300	1.22138593743673\\
400	1.22138593696591\\
500	1.21990362266073\\
600	1.21984601184503\\
700	1.21984598782825\\
800	1.21984599252904\\
900	1.21851742859662\\
1000	1.21851742871054\\
2000	1.218517423794\\
3000	1.21851743224633\\
4000	1.21789756009852\\
5000	1.21784630283811\\
6000	1.21784630155865\\
7000	1.21780528641541\\
8000	1.21760200049366\\
9000	1.21744678652966\\
10000	1.21744678649881\\
90000      1.21744678649881\\
90000	1.2367915188808\\
10000	1.2367915188808\\
9000	1.23679151890685\\
8000	1.23874581894327\\
7000	1.24546711731161\\
6000	1.25055594704458\\
5000	1.250555946855\\
4000	1.25815881034301\\
3000	1.26872506989623\\
2000	1.30843172824957\\
1000	1.40877331859085\\
900	1.43408597381899\\
800	1.44803659772777\\
700	1.50904655068873\\
600	1.62281061306987\\
500	1.80402902141993\\
400	2.00288513809667\\
300	2.28684318616537\\
200	3.83320052838975\\
100	5.30781654327367\\
90	8.01504564206975\\
80	8.01504561930379\\
70	8.30009646023389\\
60	9.06756712328384\\
50	9.54571340420395\\
40	10.7947161495485\\
30	14.9844606772684\\
20	23.479991472872\\
10	70.1733036265233\\
9	78.1410743140534\\
8	91.9028975758403\\
7	94.797941357001\\
6	94.7979413611145\\
5	94.797941279872\\
4	98.2305489066226\\
3	99.1672795005246\\
2	99.991713310573\\
1	100.000001529317\\
}--cycle;

\addplot [color=mycolor2, solid,line width=1.5pt]
  table[row sep=crcr]{%
1	90.0988050872139\\
2	74.8559092713465\\
3	58.3740012753949\\
4	45.7955902519346\\
5	35.8635147694922\\
6	28.5860308398038\\
7	22.5827573210655\\
8	18.0333956261914\\
9	15.3877213318488\\
10	13.4436958792473\\
20	6.14068215774133\\
30	4.34954158372674\\
40	3.46122965614533\\
50	2.97739665229225\\
60	2.62103541037543\\
70	2.38307666634709\\
80	2.22704450558327\\
90	2.06498996159715\\
100	1.96438382986823\\
200	1.52532404666544\\
300	1.39413594434313\\
400	1.34570792651888\\
500	1.31627993472894\\
600	1.29786965282214\\
700	1.28371427700216\\
800	1.27386956207888\\
900	1.26695036359398\\
1000	1.26243951330834\\
2000	1.24033826626804\\
3000	1.23266956870538\\
4000	1.22871956187396\\
5000	1.22633726262058\\
6000	1.22479092312255\\
7000	1.22379566139958\\
8000	1.22287742352561\\
9000	1.22215100223844\\
10000	1.22162047102615\\
90000	1.22162047102615\\
};
\addlegendentry{$\xnb = \|\cost\|_\infty$};

\addplot[area legend,solid,fill=mycolor3,opacity=0.25,draw=none,forget plot]
table[row sep=crcr] {%
x	y\\
1	274221.033000622\\
2	469.613556032881\\
3	3.48721473507746\\
4	2.85193524555087\\
5	2.8519364270079\\
6	2.5965971629727\\
7	1.95753391047949\\
8	1.9575541777959\\
9	1.95755429557138\\
10	1.63669708450418\\
20	1.53433558516822\\
30	1.38210203317975\\
40	1.38210216885707\\
50	1.29335500621487\\
60	1.29335470494149\\
70	1.29335432955327\\
80	1.2933542216002\\
90	1.29335455602246\\
100	1.2466505804252\\
200	1.22404341249512\\
300	1.22404290796734\\
400	1.22053012115844\\
500	1.21990358559088\\
600	1.21984488819389\\
700	1.21984599034574\\
800	1.2198459844731\\
900	1.21984548946883\\
1000	1.21943079603014\\
2000	1.21943125742712\\
3000	1.21910571951159\\
4000	1.2177775156776\\
5000	1.21777707029804\\
6000	1.2177771371855\\
7000	1.21777713736774\\
8000	1.21760200096946\\
9000	1.21744678586919\\
10000	1.21744678595159\\
90000	1.21744678595159\\
90000	1.23327839942864\\
10000	1.23327839942864\\
9000	1.23329468709275\\
8000	1.24137528539087\\
7000	1.24137527394481\\
6000	1.24137528421462\\
5000	1.24494923824726\\
4000	1.25882152897169\\
3000	1.29000970091385\\
2000	1.29949424210475\\
1000	1.40877323061208\\
900	1.43408596392102\\
800	1.43542924566361\\
700	1.50904653496347\\
600	1.56708936228215\\
500	1.75134995924504\\
400	2.00288506198568\\
300	2.39057443930394\\
200	2.79984469509506\\
100	5.30961100123332\\
90	8.01503965956585\\
80	8.01503700224842\\
70	8.30009582075411\\
60	9.06756522040735\\
50	9.54571331006024\\
40	13.0100271086148\\
30	15.8465288076035\\
20	26.0095176807156\\
10	516138.231060319\\
9	563124.099372117\\
8	563124.102533274\\
7	817984.800975262\\
6	817984.800673086\\
5	891470.172514893\\
4	919705.358565746\\
3	921882.29880044\\
2	941777.972236723\\
1	945309.87173611\\
}--cycle;

\addplot [color=mycolor4,dashed,line width=1.5pt]
  table[row sep=crcr]{%
1	815296.000743184\\
2	570766.640331088\\
3	330784.106503749\\
4	187151.022981435\\
5	112220.758936073\\
6	65285.3958875767\\
7	34921.2444390044\\
8	15847.8435941127\\
9	10790.164414946\\
10	6999.1311452946\\
20	6.46017233573243\\
30	4.40786153197229\\
40	3.57754254365448\\
50	3.03359733715022\\
60	2.71604155729113\\
70	2.4146911555697\\
80	2.2565160912416\\
90	2.08240316384355\\
100	1.97999087837892\\
200	1.53403928117415\\
300	1.40535725279897\\
400	1.35232089143797\\
500	1.32083068402102\\
600	1.3009730333718\\
700	1.28655650264788\\
800	1.27561693749942\\
900	1.26919993033891\\
1000	1.26373398482813\\
2000	1.24143831516366\\
3000	1.23353693215556\\
4000	1.22911257498496\\
5000	1.22657498445719\\
6000	1.22497392900458\\
7000	1.22395425939084\\
8000	1.22310237614617\\
9000	1.22221048015023\\
10000	1.22176594472622\\
90000	1.22176594472622\\
};
\addlegendentry{$\xnb=\infty$}

\addplot [color=red,loosely dotted,line width=2pt]
  table[row sep=crcr]{%
1	1.3187\\
90000	1.3187\\
};
\addlegendentry{$\Jac$}

\end{axis} 
\end{tikzpicture}%
}\label{fig:LQG:ran:2}}
	\qquad 
	\subfigure[$n=10$ basis functions]{\scalebox{1}{
%
%
\definecolor{mycolor1}{rgb}{0.13, 0.55, 0.13}%
\definecolor{mycolor2}{rgb}{0.13, 0.55, 0.13}%

\definecolor{mycolor3}{rgb}{0.32,0.09,0.98}%
\definecolor{mycolor4}{rgb}{0.32,0.09,0.98}%
\begin{tikzpicture}

\begin{axis}[%
width=2.2in,
height=1.6in,
at={(1.011111in,0.813889in)},
scale only axis,
xmode=log,
xmin=1,
xmax=100000,
xminorticks=true,
xlabel={$N$},
xmajorgrids,
xminorgrids,
ymode=log,
ymin=0.5,
ymax=1000000,
ylabel={$J^{\text{AC}}_{n,N}$},
ymajorgrids,
ylabel style={yshift=-0.3cm},
legend style={legend cell align=left,align=left,draw=white!15!black,line width=1.0pt,font=\footnotesize}
]

\addplot[area legend,solid,fill=mycolor1,opacity=0.3,draw=none,forget plot]
table[row sep=crcr] {%
x	y\\
1	64.8626787674932\\
2	51.8028663321495\\
3	40.8698505042221\\
4	34.2057123029274\\
5	32.327740150497\\
6	26.5906149656799\\
7	21.8167512994986\\
8	18.2942570690963\\
9	17.509433217615\\
10	16.5109237494213\\
20	8.68290108849987\\
30	5.49612925906822\\
40	2.35413086420275\\
50	2.32862407932307\\
60	2.12306400135207\\
70	1.86120000680995\\
80	1.63535781691707\\
90	1.63535781690026\\
100	1.55959083049867\\
200	1.39333461801874\\
300	1.3643441040029\\
400	1.35394428022446\\
500	1.34195332050755\\
600	1.33856909460061\\
700	1.32859904928777\\
800	1.32458818880617\\
900	1.32445301927657\\
1000	1.32445302330009\\
2000	1.31914600579702\\
3000	1.31879268391564\\
4000	1.31862187925411\\
5000	1.31844109127923\\
6000	1.31830688543409\\
7000	1.31826094723498\\
8000	1.31826094843103\\
9000	1.3182005639989\\
100000	1.31813536293608\\
100000	1.32263803288161\\
9000	1.32500128260116\\
8000	1.32539113783232\\
7000	1.32655764138582\\
6000	1.32883437777301\\
5000	1.33550740507875\\
4000	1.34085918309822\\
3000	1.37977303081213\\
2000	1.41450337883357\\
1000	1.82423091203754\\
900	1.84987753113501\\
800	1.94883222433922\\
700	2.04079150965957\\
600	2.06445841923901\\
500	2.34248038369511\\
400	2.96176544947812\\
300	3.88985563333783\\
200	5.47246905081519\\
100	12.1911192973486\\
90	13.3995722506444\\
80	15.4831188210459\\
70	20.3527647402716\\
60	22.5255326834499\\
50	27.9046058790261\\
40	32.1987850816379\\
30	42.5032034823197\\
20	59.1930381508071\\
10	84.0304606027895\\
9	86.9142744662561\\
8	86.9142745304304\\
7	93.2294273884365\\
6	94.253548028958\\
5	97.06945360186\\
4	98.4481209823781\\
3	99.2617095114665\\
2	99.9999994398324\\
1	100.000002879261\\
}--cycle;

\addplot [color=mycolor2, solid,line width=1.5pt]
  table[row sep=crcr]{%
1	94.5768870629217\\
2	87.6422421294588\\
3	81.1060535945275\\
4	75.7166598326944\\
5	70.8550067501774\\
6	65.818793100411\\
7	61.4661914492361\\
8	57.5867852858906\\
9	54.506862634783\\
10	51.0690016796575\\
20	30.1255056462458\\
30	19.8540579079187\\
40	14.2465207906624\\
50	10.8717423880872\\
60	8.85011139713774\\
70	7.25012839633849\\
80	6.23207351231921\\
90	5.3655282834809\\
100	4.77894898497726\\
200	2.44822770289855\\
300	1.87735697901229\\
400	1.67372096917809\\
500	1.56577991405623\\
600	1.49838919411101\\
700	1.45440220544531\\
800	1.42286573705222\\
900	1.40333873367853\\
1000	1.38702187301104\\
2000	1.33566597760556\\
3000	1.32648669905007\\
4000	1.32290226192807\\
5000	1.32128509023382\\
6000	1.32039006435097\\
7000	1.31992148967192\\
8000	1.31948814984273\\
9000	1.31924603433781\\
100000	1.31906764370644\\
};
\addlegendentry{$\xnb = \|\cost\|_\infty$};


\addplot[area legend,solid,fill=mycolor3,opacity=0.25,draw=none,forget plot]
table[row sep=crcr] {%
x	y\\
1	639983.366524608\\
2	329098.537697569\\
3	270518.925077757\\
4	162716.677544189\\
5	99203.2256406953\\
6	86226.4301855034\\
7	47489.5274014969\\
8	47489.5282154978\\
9	30133.0276457532\\
10	21613.4199719176\\
20	7.0906145341405\\
30	4.65928077970097\\
40	2.39043513144145\\
50	2.3286038154872\\
60	2.12306379225585\\
70	2.09667479516671\\
80	1.635356391661\\
90	1.63535708245283\\
100	1.55958962460745\\
200	1.42441044263701\\
300	1.3729378008794\\
400	1.34201033025407\\
500	1.34200855870122\\
600	1.33393016386478\\
700	1.32548079869412\\
800	1.3245880787709\\
900	1.32445276232808\\
1000	1.32067603411791\\
2000	1.31914752559519\\
3000	1.31890352096265\\
4000	1.31862187891315\\
5000	1.3184410887336\\
6000	1.31830689055684\\
7000	1.31826094801364\\
8000	1.31820169759628\\
9000	1.31816451028211\\
100000	1.31813534897557\\
100000	1.32635030823924\\
9000	1.3263503050298\\
8000	1.32709601638159\\
7000	1.32710598610841\\
6000	1.32938918012012\\
5000	1.33621956240585\\
4000	1.3408591123418\\
3000	1.38643803937312\\
2000	1.41323205875657\\
1000	1.7807446249738\\
900	1.78586817791992\\
800	1.94883224270823\\
700	1.95467052048799\\
600	2.19564446831188\\
500	2.376049145298\\
400	2.84118283785502\\
300	4.42306322890742\\
200	5.948249385366\\
100	12.1234649487366\\
90	13.2643994103875\\
80	14.8895276860826\\
70	20.6106203279187\\
60	22.9578071142619\\
50	18731.5689369941\\
40	46154.8527946401\\
30	122416.784984592\\
20	245819.513244859\\
10	657237.808452924\\
9	733798.353290344\\
8	733798.349269297\\
7	879520.407606198\\
6	884146.221461426\\
5	916423.123492617\\
4	921568.783455853\\
3	938169.047545757\\
2	948954.054124091\\
1	949823.811310828\\
}--cycle;

\addplot [color=mycolor4,dashed,line width=1.5pt]
  table[row sep=crcr]{%
1	881068.308142879\\
2	753794.288347249\\
3	651989.021149723\\
4	561892.86158924\\
5	489715.217209112\\
6	427140.172458315\\
7	373606.121950658\\
8	324733.263256534\\
9	277359.534656746\\
10	235953.292052392\\
20	42857.0692688388\\
30	4299.73531162478\\
40	503.880225630512\\
50	57.5866681916299\\
60	8.75517131945854\\
70	7.16879104175466\\
80	6.03343592545909\\
90	5.30458286475803\\
100	4.74690136159816\\
200	2.47432622510224\\
300	1.90523920521026\\
400	1.67809177228141\\
500	1.56928466228527\\
600	1.50170304052627\\
700	1.44790616627579\\
800	1.42030204832343\\
900	1.40220936819888\\
1000	1.38672104646169\\
2000	1.3356980463223\\
3000	1.32613696787708\\
4000	1.32267627846345\\
5000	1.32124562408155\\
6000	1.32037996927841\\
7000	1.31984030493074\\
8000	1.31951848108317\\
9000	1.31925491020946\\
100000	1.31909475493793\\
};
\addlegendentry{$\xnb=\infty$}

\addplot [color=red,loosely dotted,line width=2pt]
  table[row sep=crcr]{%
1	1.3187\\
100000	1.3187\\
};
\addlegendentry{$\Jac$}

\end{axis}
\end{tikzpicture}%
}\label{fig:LQG:ran:10}}
	\\	
	\subfigure[$n=100$ basis functions]{\scalebox{1}{
%
%
\definecolor{mycolor1}{rgb}{0.13, 0.55, 0.13}%
\definecolor{mycolor2}{rgb}{0.13, 0.55, 0.13}%

\definecolor{mycolor3}{rgb}{0.32,0.09,0.98}%
\definecolor{mycolor4}{rgb}{0.32,0.09,0.98}%
\begin{tikzpicture}

\begin{axis}[%
width=2.2in,
height=1.6in,
at={(1.011111in,0.813889in)},
scale only axis,
xmode=log,
xmin=1,
xmax=100000,
xminorticks=true,
xlabel={$N$},
xmajorgrids,
xminorgrids,
ymode=log,
ymin=0.5,
ymax=1000000,
ylabel={$J^{\text{AC}}_{n,N}$},
ymajorgrids,
ylabel style={yshift=-0.3cm},
legend style={legend cell align=left,align=left,draw=white!15!black,line width=1.0pt, font=\footnotesize}
]

\addplot[area legend,solid,fill=mycolor1,opacity=0.3,draw=none,forget plot]
table[row sep=crcr] {%
x	y\\
1	73.0725008122354\\
2	61.4824626630137\\
3	51.1246287160035\\
4	44.1816795345505\\
5	41.624170651969\\
6	41.6241708288271\\
7	41.3325160220246\\
8	41.3026068831755\\
9	40.6233055772378\\
10	37.8298523875866\\
20	27.1803843820131\\
30	24.5439377663766\\
40	21.1587220918693\\
50	18.5528520648749\\
60	16.6437499563526\\
70	16.322223991574\\
80	15.8713158532552\\
90	14.8174019929416\\
100	13.3405730708024\\
200	10.0409771992707\\
300	8.04018316672937\\
400	7.60747751315085\\
500	6.78171835223416\\
600	5.89411527131351\\
700	5.34601605040283\\
800	5.02765091315349\\
900	4.77759951871282\\
1000	4.33511024741078\\
2000	2.85423098452847\\
3000	2.49758098138473\\
4000	2.13430309683285\\
5000	1.90766192491257\\
6000	1.87269917912076\\
7000	1.7562695919021\\
8000	1.7135652779055\\
9000	1.69424841286886\\
10000	1.69378703711643\\
20000	1.47208388988814\\
30000	1.42459098794049\\
40000	1.39714278405354\\
50000	1.38376127859128\\
60000	1.3739062869345\\
70000	1.37355295449137\\
80000	1.36373362951018\\
90000	1.35874593505327\\
90000	1.44662276075467\\
80000	1.45652639971006\\
70000	1.47712914703438\\
60000	1.51349283270477\\
50000	1.53008822831339\\
40000	1.61330074144611\\
30000	1.68707340458998\\
20000	1.86627528768281\\
10000	2.3131049866413\\
9000	2.42296812312227\\
8000	2.61777890277017\\
7000	2.68786005791279\\
6000	2.89127196564123\\
5000	3.510806434344\\
4000	3.88588637605768\\
3000	4.44770418607772\\
2000	5.55554394108026\\
1000	8.3041309640639\\
900	8.61069152874115\\
800	9.22388343498101\\
700	9.73895406512103\\
600	10.0887545803061\\
500	11.7423421318748\\
400	13.3781368228904\\
300	15.0578515420467\\
200	19.1904050392464\\
100	26.5167653669121\\
90	26.8915076038712\\
80	29.9242066761444\\
70	31.7378829721566\\
60	36.5436400216705\\
50	41.2533356354866\\
40	43.1544597275795\\
30	49.6880996420488\\
20	65.6018066727078\\
10	86.0282494739563\\
9	88.9376572262689\\
8	92.1518901496837\\
7	93.9846699180999\\
6	95.0637467475903\\
5	96.2394683425346\\
4	99.196295378522\\
3	99.1962963229621\\
2	99.93608188773\\
1	99.9999999628775\\
}--cycle;

  \addplot [color=mycolor2, solid,line width=1.5pt]
  table[row sep=crcr]{%
1	95.0700417717928\\
2	90.0671167740208\\
3	84.9944490270267\\
4	79.9920905070121\\
5	76.0027072571778\\
6	72.0642442192834\\
7	68.5216005663884\\
8	65.3043977366562\\
9	62.3114318130003\\
10	59.4892715267976\\
20	43.4644954999383\\
30	35.5965024418207\\
40	30.8718050780024\\
50	27.5461133266723\\
60	25.1746912498526\\
70	23.3269440347622\\
80	21.782611418054\\
90	20.5563711968901\\
100	19.5603805628542\\
200	13.9065438625585\\
300	11.4420274150589\\
400	9.97217331313981\\
500	8.8661674800461\\
600	8.08608694410586\\
700	7.44042640087966\\
800	6.93601395110263\\
900	6.49772069280504\\
1000	6.1011506875886\\
2000	4.12697750456289\\
3000	3.30182507653445\\
4000	2.84903992360772\\
5000	2.56617423483291\\
6000	2.37412705435564\\
7000	2.22851604396447\\
8000	2.11901250524476\\
9000	2.03455393456047\\
10000	1.96698930950154\\
20000	1.64938776316658\\
30000	1.54311305739475\\
40000	1.48459019402139\\
50000	1.45122052152376\\
60000	1.42868871164553\\
70000	1.41198856115809\\
80000	1.39930609641481\\
90000	1.39159597512432\\
};
\addlegendentry{$\xnb = \|\cost\|_\infty$};

\addplot[area legend,solid,fill=mycolor3,opacity=0.25,draw=none,forget plot]
table[row sep=crcr] {%
x	y\\
1	726032.932683365\\
2	575497.422786663\\
3	487235.424507427\\
4	433795.962742556\\
5	373175.234062319\\
6	311798.973653394\\
7	268995.513623822\\
8	258287.077277272\\
9	227632.440252786\\
10	220081.695430595\\
20	112157.835636989\\
30	87800.1396742205\\
40	65844.4903492122\\
50	50680.4553501273\\
60	43162.6798842125\\
70	35355.7953123166\\
80	32661.7650283694\\
90	26002.954959635\\
100	22773.0465341619\\
200	2411.72466093664\\
300	21.0573383250428\\
400	16.046915139724\\
500	12.8109912873503\\
600	11.2102195393165\\
700	10.7167490565531\\
800	10.1411754565593\\
900	9.30056376015061\\
1000	8.2707566654869\\
2000	5.24154994749335\\
3000	4.03795564502112\\
4000	3.26947862218561\\
5000	3.03295845707396\\
6000	2.77332330384995\\
7000	2.49505756211084\\
8000	2.36738517154882\\
9000	2.25014532078066\\
10000	2.21109518688523\\
20000	1.69815240236852\\
30000	1.62161878310671\\
40000	1.51748252946769\\
50000	1.47201423243165\\
60000	1.45660169781521\\
70000	1.44077296343534\\
80000	1.42369917319803\\
90000	1.40053639945081\\
90000	1.59392070952854\\
80000	1.62476181187348\\
70000	1.68447114549425\\
60000	1.7304245328959\\
50000	1.82134680815319\\
40000	1.97163710481217\\
30000	2.20338202233787\\
20000	2.47092552819889\\
10000	3.66438170316786\\
9000	3.88629035695664\\
8000	4.37476900292604\\
7000	4.48013062748919\\
6000	4.82971826032916\\
5000	5.46485996971368\\
4000	6.26061109589694\\
3000	7.79724788905368\\
2000	10.3707742851764\\
1000	14.5345960731987\\
900	15.2854850836583\\
800	16.928960911425\\
700	19.0934194379443\\
600	38.9268305307339\\
500	607.995945437792\\
400	2447.8486587807\\
300	10365.3366176473\\
200	25996.4007632041\\
100	52013.7262786537\\
90	65575.2480556314\\
80	76659.0671714717\\
70	87594.2707644928\\
60	114733.56462992\\
50	166787.067405686\\
40	170002.388063472\\
30	306333.320936457\\
20	395872.245625063\\
10	677071.315309723\\
9	699076.981019667\\
8	701474.473762693\\
7	822322.595975613\\
6	911948.680060861\\
5	923400.838429024\\
4	936081.988079509\\
3	945357.985362924\\
2	947776.774626934\\
1	951818.750313226\\
}--cycle;

\addplot [color=mycolor4,dashed,line width=1.5pt]
  table[row sep=crcr]{%
1	898168.581130224\\
2	802140.564420225\\
3	711397.054174735\\
4	631227.507176506\\
5	564596.867924508\\
6	509081.327498531\\
7	461185.52507959\\
8	419045.290316326\\
9	388003.069338418\\
10	354681.204123067\\
20	195414.667488606\\
30	133418.7035383\\
40	99905.9936620545\\
50	79327.3864603108\\
60	65543.7009034799\\
70	55436.5410869081\\
80	47602.4289447553\\
90	41155.4071614626\\
100	36113.5661551971\\
200	10498.8978786933\\
300	1302.02082425537\\
400	97.9984650481927\\
500	24.2219285693239\\
600	16.2337383070317\\
700	14.4287485468089\\
800	13.1608781833076\\
900	12.1746393095505\\
1000	11.361749131535\\
2000	7.27993084063056\\
3000	5.65572319796055\\
4000	4.74967312763189\\
5000	4.1522738928749\\
6000	3.73741043135761\\
7000	3.41604996101106\\
8000	3.18587668073005\\
9000	2.99136381697669\\
10000	2.8428530214657\\
20000	2.10067242226305\\
30000	1.83987392899804\\
40000	1.70657578675137\\
50000	1.62699148118933\\
60000	1.57452717762884\\
70000	1.53602078018279\\
80000	1.50792906440744\\
90000	1.4855385240109\\
};
\addlegendentry{$\xnb=\infty$}

\addplot [color=red,loosely dotted,line width=2pt]
  table[row sep=crcr]{%
1	1.3187\\
90000	1.3187\\
};
\addlegendentry{$\Jac$}

\end{axis}
\end{tikzpicture}%
}\label{fig:LQG:ran:100}}
	~
	\subfigure[Zoomed version of the average cost for different $n$]{\scalebox{1}{
%
%

\definecolor{mycolor1}{rgb}{0.0, 0.5, 1.0}
\definecolor{mycolor2}{rgb}{0.0, 0.8, 0.6}
\definecolor{mycolor3}{rgb}{1.0, 0.5, 0.31}
\definecolor{mycolor4}{rgb}{0.13, 0.55, 0.13}%

\begin{tikzpicture}

\begin{axis}[%
width=2.2in,
height=1.6in,
at={(1.011111in,0.813889in)},
scale only axis,
xmode=log,
xmin=100,
xmax=1000000,
xminorticks=true,
xlabel={$N$},
xmajorgrids,
xminorgrids,
ymode=log,
ymin=1.2,
ymax=2,
ylabel={$J^{\text{AC}}_{n,N}$},
ymajorgrids,
ylabel style={yshift=-0.1cm},
legend style={legend cell align=left,align=left,draw=white!15!black,line width=1.0pt, font=\footnotesize}
]

\addplot [color=mycolor1, solid, mark=otimes*, mark size=2 ,line width=1.5pt]
  table[row sep=crcr]{%
1	90.0988050872139\\
2	74.8559092713465\\
3	58.3740012753949\\
4	45.7955902519346\\
5	35.8635147694922\\
6	28.5860308398038\\
7	22.5827573210655\\
8	18.0333956261914\\
9	15.3877213318488\\
10	13.4436958792473\\
20	6.14068215774133\\
30	4.34954158372674\\
40	3.46122965614533\\
50	2.97739665229225\\
60	2.62103541037543\\
70	2.38307666634709\\
80	2.22704450558327\\
90	2.06498996159715\\
100	1.96438382986823\\
200	1.52532404666544\\
300	1.39413594434313\\
400	1.34570792651888\\
500	1.31627993472894\\
600	1.29786965282214\\
700	1.28371427700216\\
800	1.27386956207888\\
900	1.26695036359398\\
1000	1.26243951330834\\
2000	1.24033826626804\\
3000	1.23266956870538\\
4000	1.22871956187396\\
5000	1.22633726262058\\
6000	1.22479092312255\\
7000	1.22379566139958\\
8000	1.22287742352561\\
9000	1.22215100223844\\
10000	1.22162047102615\\
20000 1.22162047102615\\
30000 1.22162047102615\\
40000 1.22162047102615\\
50000 1.22162047102615\\
60000 1.22162047102615\\
70000 1.22162047102615\\
80000 1.22162047102615\\
90000	1.22162047102615\\
100000 1.22162047102615\\
200000 1.22162047102615\\
300000 1.22162047102615\\
400000 1.22162047102615\\
500000 1.22162047102615\\
600000 1.22162047102615\\
700000 1.22162047102615\\
800000 1.22162047102615\\
900000 1.22162047102615\\
1000000 1.22162047102615\\
};
\addlegendentry{$n=2$};

\addplot [color=mycolor2, solid, mark=star, mark size=3, line width=1.5pt]
  table[row sep=crcr]{%
1	94.5768870629217\\
2	87.6422421294588\\
3	81.1060535945275\\
4	75.7166598326944\\
5	70.8550067501774\\
6	65.818793100411\\
7	61.4661914492361\\
8	57.5867852858906\\
9	54.506862634783\\
10	51.0690016796575\\
20	30.1255056462458\\
30	19.8540579079187\\
40	14.2465207906624\\
50	10.8717423880872\\
60	8.85011139713774\\
70	7.25012839633849\\
80	6.23207351231921\\
90	5.3655282834809\\
100	4.77894898497726\\
200	2.44822770289855\\
300	1.87735697901229\\
400	1.67372096917809\\
500	1.56577991405623\\
600	1.49838919411101\\
700	1.45440220544531\\
800	1.42286573705222\\
900	1.40333873367853\\
1000	1.38702187301104\\
2000	1.33566597760556\\
3000	1.32648669905007\\
4000	1.32290226192807\\
5000	1.32128509023382\\
6000	1.32039006435097\\
7000	1.31992148967192\\
8000	1.31948814984273\\
9000	1.31924603433781\\
10000  1.31906764370644\\
20000  1.31906764370644\\
30000  1.31906764370644\\
40000  1.31906764370644\\
50000  1.31906764370644\\
60000  1.31906764370644\\
70000  1.31906764370644\\
80000  1.31906764370644\\
90000  1.31906764370644\\
100000	1.31906764370644\\
200000 1.31906764370644\\
300000 1.31906764370644\\
400000 1.31906764370644\\
500000 1.31906764370644\\
600000 1.31906764370644\\
700000 1.31906764370644\\
800000 1.31906764370644\\
900000 1.31906764370644\\
1000000   1.31906764370644\\
};
\addlegendentry{$n=10$};

 \addplot [color=mycolor3, solid, mark=diamond, mark size=3, line width=1.5pt]
  table[row sep=crcr]{%
1	95.0727659242603\\
2	89.3293464937204\\
3	83.9600007375867\\
4	78.0428561341892\\
5	73.7229205184179\\
6	70.4072926106444\\
7	66.6424865458847\\
8	63.334550659593\\
9	60.216269546597\\
10	58.0616678649646\\
20	42.5296820114253\\
30	34.3797500254576\\
40	29.669338715003\\
50	26.831507897844\\
60	24.6676317740388\\
70	22.7390049288413\\
80	21.4440070632338\\
90	20.3367569025039\\
100	19.2554694316222\\
200	13.76144920044\\
300	11.3359882537192\\
400	9.84420511083659\\
500	8.82623441506369\\
600	7.9885333739873\\
700	7.38758245792849\\
800	6.81009576743343\\
900	6.42393695130645\\
1000	6.02474149436536\\
2000	4.09354171022974\\
3000	3.23896547092362\\
4000	2.78533674555221\\
5000	2.53689020223726\\
6000	2.35122639488588\\
7000	2.20059192075798\\
8000	2.09602080177611\\
9000	2.00843218001926\\
10000	1.94483860948849\\
20000	1.63302526595443\\
30000	1.53134706861736\\
40000	1.4734347662596\\
50000	1.44301446129694\\
60000	1.42129288813096\\
70000	1.4058035303561\\
80000	1.39424179901048\\
90000	1.38593192917479\\
100000	1.379296445348\\
200000	1.34534237322102\\
300000	1.33368847028753\\
400000	1.32412199975997\\
500000	1.32166204586395\\
600000	1.32034235423453\\
700000  1.31986648762340\\
800000  1.31955634230423\\
800000  1.31920979847234\\
1000000 1.31906764370644\\
};
\addlegendentry{$n=100$};

\addplot [color=red,loosely dotted,line width=2pt]
  table[row sep=crcr]{%
100	1.3187\\
1000000	1.3187\\
};

\addlegendentry{$\Jac$}

\end{axis}
\end{tikzpicture}%
}\label{fig:LQG:ran:tail}}
	\caption{The objective performance $\JacnN$ is computed using \eqref{AC-LP-n,N} for Example~\ref{ex:LQG}. The red dotted line denoted by $\Jac$ is the optimal solution approximated by $n = 10^3$ and $N = 10^6$.}  
	\label{fig:LQG:ran}
\end{figure}

\subsubsection*{Randomized approach:} 			
We implement the methodology presented in Section~\ref{subsec:rand:MDP}, resulting in a finite random convex program as in \eqref{AC-LP-n,N}, where the uniform distribution on $K = S\times A = [-L,L]^2$ is used to draw the random samples. Figures~\ref{fig:LQG:ran:2}, \ref{fig:LQG:ran:10}, and \ref{fig:LQG:ran:100} visualize three cases with different number of basis functions $n \in \{2,10,100\}$, respectively. To show the impact of the additional norm constraint, in each case two approximation settings are examined: the constrained (regularized) one proposed in this article (i.e., $\xnb < \infty$), and the unconstrained one (i.e., $\xnb = \infty$). In the former we choose the bound suggested by \eqref{theta*:MDP2}. In the latter, the resulting optimization programs of \eqref{AC-LP-n,N} may happen to be unbounded, particularly when the number of samples $N$ is low; numerically, we capture the behavior of the unbounded $\xnb$ through a large bound such as $\theta = 10^6$. In each sub-figure, the colored tubes represent the results of $400$ independent experiments (shaded areas) as well as the mean value across different experiments (solid and dashed lines) of the objective performance $\JacnN$ as a function of the sample size $N$. 
		
Figure~\ref{fig:LQG:ran:tail} depicts a zoomed perspective of the means for the three cases of $n$. All the results in Figure~\ref{fig:LQG:ran} are obtained based on 400 independent simulation experiments. It is perhaps not surprising that the optimal value depicted in red dotted line is very close to the classical LQG example whose exact solution is analytically available. It can be seen that the randomized approximations asymptotically converge, as suggested by Theorem~\ref{thm:inf-fin}.  		

The simulation results suggest three interesting features concerning  $n$, the number of basis functions: The higher the number of basis functions,
\begin{enumerate}[label=(\roman*), itemsep = 1mm, topsep = -1mm]
	\item \label{sim:err} 
	the smaller the approximation error (i.e., asymptotic distance for $N\to\infty$ to the red dotted line),
	\item \label{sim:var}
	the lower the variance of approximation with respect to the sampling distribution for each $N$, and 
	\item \label{sim:conv}
	the slower the convergence behavior with respect to the sample size $N$.
\end{enumerate}
The features~\ref{sim:err} and \ref{sim:var} are positive impacts of increasing the number of basis functions. While \ref{sim:err} is predicted by Corollary~\ref{cor:MDP:inf-semi}, since the error due to the projection term becomes smaller, it is not entirely clear how to formally explain \ref{sim:var}. On the contrary, the feature~\ref{sim:conv} is indeed a negative impact, as a high number of basis functions requires a large number of samples $N$ to produce reasonable approximation errors. This phenomena can be justified through the lens of Corollary~\ref{cor:adp:semi-finite} where the approximation errors grows proportionally to $n$. 

	
\subsubsection*{Structural convex optimization:} 
Algorithm \ref{alg} was implemented with the parameters described in Corollary~\ref{cor:adp:smooth} leading to deterministic upper and lower bounds ($\Jnub$ and $\Jnlb$, respectively) for the cost function $\Jacn$, see also Theorem~\ref{thm:semi-fin:smooth}. These bounds are computationally appealing as they provide a posteriori bounds on the approximation error that often is significantly smaller than the a priori bounds given by Theorem~\ref{thm:semi-fin:smooth}. This behavior can be seen in the simulation results summarized in Figure~\ref{fig:LQG:smoothing} where the number of basis functions is $n=10$. Similar to Figure~\ref{fig:LQG:ran}, the red dotted line is the optimal value of the original infinite program \ref{primal-inf}, which we approximated by using $10^3$ basis functions and $10^6$ iterations of Algorithm~\ref{alg}; it coincides with the one from the randomized method.
\begin{figure}[t]
	\subfigure[A priori error $\varepsilon$ and a posteriori error $\Jnub-\Jnlb$]{\scalebox{1}{
%
%
\begin{tikzpicture}

\begin{axis}[%
width=2.2in,
height=1.8in,
at={(1.011111in,0.641667in)},
scale only axis,
xmin=10,
xmax=100000,
xmode=log,
xlabel={number of iterations $k$},
xmajorgrids,
xminorgrids,
ymin=0,
ymax=600,
ymode=log,
ylabel={prior $\&$ posterior error},
ymajorgrids,
ylabel style={yshift=-0.15cm},
legend style={legend cell align=left,align=left,draw=white!15!black,line width=1.0pt,font=\footnotesize}
]
\addplot [color=black,solid,line width=1.5pt]
  table[row sep=crcr]{%
10	536.129757314671\\
20	297.757439244727\\
30	209.453263910916\\
40	162.72697105966\\
50	133.599201535883\\
60	113.618546830154\\
70	99.0202124606833\\
80	87.8655207223592\\
90	79.0515988085919\\
100	71.9035232973631\\
200	38.3663375914169\\
300	26.4817232476722\\
400	20.3309563638865\\
500	16.5512539578754\\
600	13.9852128898149\\
700	12.1253407196401\\
800	10.7133433712918\\
900	9.6036263953016\\
1000	8.70775656335686\\
2000	4.56109548359483\\
3000	3.11906336071083\\
4000	2.38020824257256\\
5000	1.92921540328592\\
6000	1.62455856340165\\
7000	1.40460792895481\\
8000	1.23815997581314\\
9000	1.10769960273813\\
10000	1.00262523720079\\
20000	0.51969893006755\\
30000	0.353458689836964\\
40000	0.268758966720061\\
50000	0.217256676546119\\
60000	0.182565901433752\\
70000	0.157578014270967\\
80000	0.138704324779933\\
90000	0.123935179488576\\
100000	0.112056573667818\\
200000	0.0576958143278744\\
300000	0.0391007132359004\\
400000	0.0296601216738937\\
500000	0.0239338161100979\\
};
\addlegendentry{$\varepsilon$};

\addplot [color=black,dashed,line width=1.5pt]
  table[row sep=crcr]{%
10	24.3565005293573\\
20	17.5364265215274\\
30	13.2451063157308\\
40	9.31312109408406\\
50	9.03581976197761\\
60	7.29124989487846\\
70	6.85165949467133\\
80	6.73215972985745\\
90	7.91906192077684\\
100	6.43867285663217\\
200	3.45711494005305\\
300	1.81584251291752\\
400	1.55869807407919\\
500	1.46754111427825\\
600	1.66757043728\\
700	1.65550462437804\\
800	1.37002395686236\\
900	1.26175419668461\\
1000	1.33156112752499\\
2000	0.510705599116073\\
3000	0.320497504636412\\
4000	0.29995740865895\\
5000	0.285731965916484\\
6000	0.190662075258335\\
7000	0.169055891413946\\
8000	0.19042832425306\\
9000	0.184585055707421\\
10000	0.14371286558865\\
20000	0.126900649247346\\
30000	0.0655305859838577\\
40000	0.0253164317947498\\
50000	0.0334016204747398\\
60000	0.0405570683669749\\
70000	0.0271062817729779\\
80000	0.0168730222321867\\
90000	0.0179113676236098\\
100000	0.02217519505106\\
200000	0.00879183854352084\\
300000	0.00619908849403994\\
400000	0.00513727563791022\\
500000	0.00364347933428188\\
};
\addlegendentry{$\Jnub-\Jnlb$};

\end{axis}
\end{tikzpicture}%
}}
	\qquad 
	\subfigure[Upper bound $\Jnub$ and lower bound $\Jnlb$]{\scalebox{1}{
%
%
\begin{tikzpicture}

\begin{axis}[%
width=2.2in,
height=1.8in,
at={(1.011111in,0.641667in)},
scale only axis,
xmin=10,
xmax=100000,
xmode=log,
xlabel={number of iterations $k$},
xmajorgrids,
xminorgrids,
ymin=0.5,
ymax=3,
ylabel={posterior approximation},
ymajorgrids,
ylabel style={yshift=-0.15cm},
legend style={legend cell align=left,align=left,legend pos=north east,draw=white!15!black,line width=1.0pt,font=\footnotesize}
]
\addplot [color=black,solid,line width=1.5pt]
  table[row sep=crcr]{%
10	25.1527051087009\\
20	18.5453358562866\\
30	14.2530400670406\\
40	10.3615635341208\\
50	10.1650692726669\\
60	8.37690974774427\\
70	7.94622728775838\\
80	7.88020667877186\\
90	9.082242328564\\
100	7.63091987563802\\
200	4.69978256390843\\
300	3.07692493690422\\
400	2.8212316189458\\
500	2.73122526180391\\
600	2.94530887515802\\
700	2.93846732574994\\
800	2.65795277930856\\
900	2.55324120683406\\
1000	2.62579210015638\\
2000	1.8153459536679\\
3000	1.62867228923354\\
4000	1.6090745751905\\
5000	1.59625396327131\\
6000	1.50194678447926\\
7000	1.48000020459195\\
8000	1.50290052371669\\
9000	1.49728138060725\\
10000	1.45712194885304\\
20000	1.44419732361832\\
30000	1.38230623729533\\
40000	1.3432552583054\\
50000	1.35079180338104\\
60000	1.35837949527842\\
70000	1.34481075737979\\
80000	1.33488846388037\\
90000	1.33579752492145\\
100000	1.34013935255392\\
200000	1.32686085101654\\
300000	1.32434929714319\\
400000	1.32325490194924\\
500000	1.32179986534431\\
};
\addlegendentry{$\Jnub$};

\addplot [color=black,dashed,line width=1.5pt]
  table[row sep=crcr]{%
10	0.796204579343647\\
20	1.00890933475926\\
30	1.00793375130979\\
40	1.04844244003675\\
50	1.12924951068927\\
60	1.08565985286581\\
70	1.09456779308705\\
80	1.14804694891441\\
90	1.16318040778716\\
100	1.19224701900585\\
200	1.24266762385537\\
300	1.2610824239867\\
400	1.26253354486661\\
500	1.26368414752567\\
600	1.27773843787802\\
700	1.2829627013719\\
800	1.28792882244619\\
900	1.29148701014945\\
1000	1.29423097263139\\
2000	1.30464035455183\\
3000	1.30817478459713\\
4000	1.30911716653155\\
5000	1.31052199735482\\
6000	1.31128470922093\\
7000	1.310944313178\\
8000	1.31247219946363\\
9000	1.31269632489983\\
10000	1.31340908326439\\
20000	1.31729667437098\\
30000	1.31677565131148\\
40000	1.31793882651065\\
50000	1.3173901829063\\
60000	1.31782242691145\\
70000	1.31770447560681\\
80000	1.31801544164819\\
90000	1.31788615729784\\
100000	1.31796415750286\\
200000	1.31806901247302\\
300000	1.31815020864915\\
400000	1.31811762631133\\
500000	1.31815638601002\\
};
\addlegendentry{$\Jnlb$};

\addplot [color=red,loosely dotted,line width=2pt]
  table[row sep=crcr]{%
10	1.3187\\
100000	1.3187\\
};
\addlegendentry{$\Jac$}

\end{axis}
\end{tikzpicture}%
}}
	\caption{The results and error bounds are obtained by Algorithm~\ref{alg} with $n=10$ for Example~\ref{ex:LQG}. The red dotted line is the optimal solution computed as indicated in Figure~\ref{fig:LQG:ran}.}  
	\label{fig:LQG:smoothing}
\end{figure}
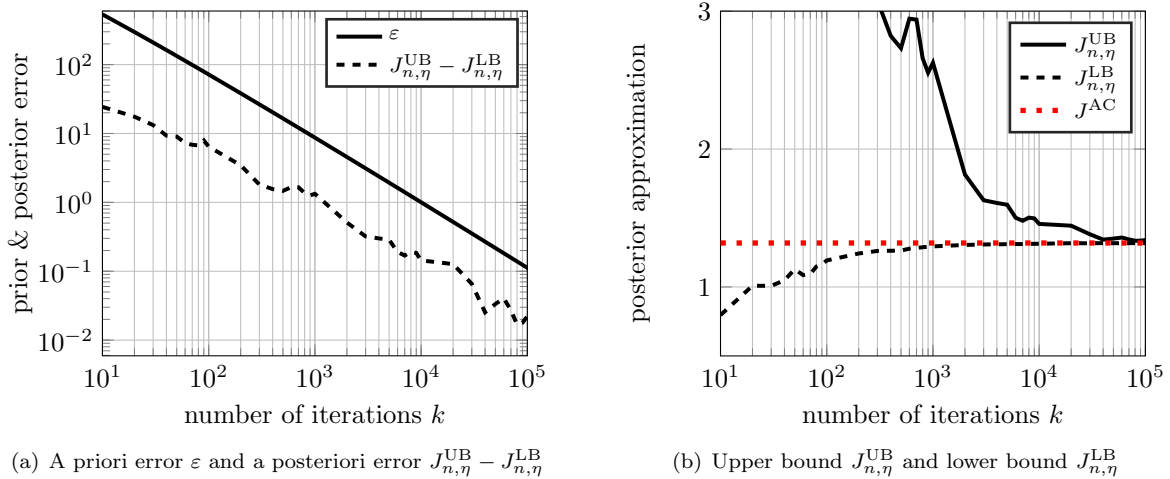

\subsection{Example 2: A fisheries management problem} 
\label{ex:fisheries}
A natural approximation approach toward dynamic programming problems goes through a discretization scheme (e.g., discretization the state and/or action spaces). The main objective of this example is to compare the proposed LP-based approximation of this article with more standard discretization schemes. To this end, we borrow an example from \cite[Section~1.3]{ref:Hernandez-96} and compare our results with the recent discretization method proposed by \cite{ref:Saldi-15}. Consider the population growth model, known as Ricker model, 
	\begin{align*} 
	s_{t+1} = \vartheta_1 a_t  \exp(-\vartheta_2 a_t + \xi_t), \quad t \in \N,
	\end{align*}
where $\vartheta_1, \vartheta_2 \in \Rp$, $s_t$ is the population size in season $t$, and $a_t$ is the population to be left for spawning for the next season, i.e., the difference $s_t -a_t$ is the amount of fish captured in season $t$. The running reward function, to be maximized is $\cost(a,s) = \varphi(s - a)$, where $\varphi$ is the so-called shifted isoelastic utility function $\varphi(z) := 3(z + 0.5)^{1/3} - (0.5)^{1/3}$ \cite[13, Section 4.1]{ref:Dufour-12}. The state space is $S = [\underline{\kappa}, \overline{\kappa}]$, for some $\underline{\kappa}, \overline{\kappa}\in\Rp$.
Since the population left for spawning cannot be greater than the total population, for each $s\in S$, the set of admissible actions is $A(s) = [\underline{\kappa}, s]$. To fulfill Assumption~\ref{a:CM}\ref{a:CM:K}, following the transformation suggested by \cite{ref:Saldi-15}, we equivalently reformulate the above problem using the dynamics 
	\begin{align*}
	s_{t+1} = \vartheta_1 \min(a_t, s_t) \exp( -\vartheta_2 \min( a_t, s_t) + \xi_t), \quad  t\in\N,
	\end{align*}
where the admissible actions set is now the state-independent set $A=[\underline{\kappa}, \overline{\kappa}]$, and the running reward function is $\cost(a,s) = \varphi(s - a)\indic{\{s\geq a\}}$. The noise process $(\xi_t)_{t\in\N}$ is a sequence of i.i.d.\ random variables which have a uniform density function $g$ supported on the interval $[0, \lambda]$. Thus, the corresponding kernel is
	\begin{align*}
	Q(B|s,a) = \int_B g\Big( \log \xi - \log\big( \vartheta_1 \min(a,s) \big) + \vartheta_2 \min(a,s) \Big)\frac{1}{\xi} \drv \xi, \quad \forall B \in \Borel{\R}.
	\end{align*}
Note that to make the model consistent, we must have $ \vartheta_1 a  \exp(-\vartheta_2 a + \xi)\in [\underline{\kappa}, \overline{\kappa}]$ for all $(a,\xi) \in  [\underline{\kappa}, \overline{\kappa}]\times [0,\lambda]$. By defining an appropriate change of coordinate similar to Lemma~\ref{lem:LQG}, Assumption~\ref{a:CM} are fulfilled; we refer the reader to \cite[Section~7.2]{ref:Saldi-15} for further information and detailed analysis. 
		
\paragraph{\emph{Simulation details:}} The chosen numerical values are $\lambda = 0.5$, $\vartheta_1 = 1.1$, $\vartheta_2 = 0.1$, $\overline{\kappa} = 7$, and $\underline{\kappa} = 0.005$.
		
\subsubsection*{Randomized approach:}
\begin{figure}[t]	
	\centering
	\scalebox{1}{
%

\definecolor{mycolor1}{rgb}{0.0, 0.5, 1.0}
\definecolor{mycolor2}{rgb}{0.0, 0.8, 0.6}
\definecolor{mycolor3}{rgb}{1.0, 0.5, 0.31}

\begin{tikzpicture}

\begin{axis}[%
width=3.0in,
height=2.0in,
at={(1.011111in,0.813889in)},
scale only axis,
xmode=log,
xmin=1,
xmax=1000,
xminorticks=true,
xlabel={$N$},
xmajorgrids,
xminorgrids,
ymode=log,
ymin=8,
ymax=100,
ylabel={$J^{\text{AC}}_{n,N}+10$},
ymajorgrids,
ylabel style={yshift=-0.3cm},
legend style={legend cell align=left,align=left,draw=white!15!black,line width=1.0pt,font=\footnotesize}
]

\addplot[area legend,solid,fill=mycolor1,opacity=0.3,draw=none,forget plot]
table[row sep=crcr] {%
x	y\\
1	60.2714090446456\\
2	21.8554017093566\\
3	10.4000775187875\\
4	10.0599599497917\\
5	9.99999999985442\\
6	9.99999999939065\\
7	9.99999999671157\\
8	9.9999999975459\\
9	9.99999999273531\\
10	9.9999999758598\\
20	9.99999991139656\\
30	9.99999991276262\\
40	9.99999993260893\\
50	9.99999989321967\\
60	9.9999999480932\\
70	9.99999995422039\\
80	9.99999994877977\\
90	9.99999991978368\\
100	9.99999990117392\\
200	9.99999984415569\\
300	9.99999979630737\\
400	9.99999994393545\\
500	9.99999999346454\\
600	9.99999999696135\\
700	9.99999999020227\\
800	9.99999829366524\\
900	9.9999986158764\\
1000	9.99999669910639\\
2000	9.99999998568734\\
3000	9.9999892100458\\
4000	9.99999992468276\\
5000	9.99999983334945\\
6000	9.99999942302657\\
7000	9.99999909003206\\
8000	9.99999847199698\\
9000	9.99999735859124\\
9000	9.99999919549118\\
8000	9.99999961746011\\
7000	9.99999977471244\\
6000	9.99999987701005\\
5000	9.9999999959124\\
4000	9.99999999913636\\
3000	9.99999999997259\\
2000	9.99999999998891\\
1000	9.99999999999882\\
900	9.99999999994029\\
800	9.99999999999979\\
700	9.99999999999992\\
600	9.99999999999989\\
500	9.99999999999988\\
400	9.9999999999999\\
300	9.99999999999995\\
200	9.99999999999987\\
100	9.99999999999985\\
90	9.99999999999975\\
80	9.99999999999955\\
70	9.99999999999948\\
60	9.9999999999996\\
50	9.99999999999922\\
40	9.99999999999883\\
30	9.9999999999967\\
20	10.000245544486\\
10	10.4889930872607\\
9	10.5196942809493\\
8	10.6165141763612\\
7	11.0978685740912\\
6	19.8757806859436\\
5	37.5359276660769\\
4	51.7389974282208\\
3	61.3683901801154\\
2	78.3949509625149\\
1	95.1440806603169\\
}--cycle;

\addplot [color=mycolor1, solid, mark=otimes*, mark size=2, line width=1.5pt]
  table[row sep=crcr]{%
1	73.8116476992506\\
2	52.2200402474537\\
3	34.3208400082223\\
4	22.2378459411505\\
5	15.9540088084741\\
6	13.1645557208202\\
7	11.5976590569994\\
8	10.7859782027953\\
9	10.5691168696462\\
10	10.3182842954979\\
20	10.0109439530958\\
30	10.0005830505906\\
40	10.0000330211688\\
50	10.0000020598734\\
60	10.0000020665347\\
70	10.0000020685042\\
80	9.99999997657583\\
90	9.99999996890386\\
100	9.99999996462579\\
200	9.99999993592086\\
300	9.99999992700283\\
400	9.99999992831566\\
500	9.99999992935017\\
600	9.99999993620589\\
700	9.99999988602315\\
800	9.99999966279497\\
900	9.99999953472218\\
1000	9.99999886917443\\
2000	9.99999995315971\\
3000	9.99999594678029\\
4000	9.9999999738941\\
5000	9.9999999284224\\
6000	9.99999968143622\\
7000	9.99999944507632\\
8000	9.99999912828565\\
9000	9.99999833870125\\
};

\addlegendentry{$n=2$};

\addplot[area legend,solid,fill=mycolor2,opacity=0.3,draw=none,forget plot]
table[row sep=crcr] {%
x	y\\
1	71.0635428418405\\
2	54.4907583447821\\
3	43.8661716108083\\
4	36.5780996871585\\
5	32.4371815502727\\
6	28.3410194752331\\
7	25.4192316649149\\
8	23.2143588995335\\
9	21.2853108616049\\
10	19.3548051842559\\
20	10.3237241051524\\
30	10.0667735195079\\
40	10.00258500703\\
50	9.99999999959783\\
60	9.99999999198904\\
70	9.99999996724115\\
80	9.9999999404122\\
90	9.99999984108334\\
100	9.9999997712211\\
200	9.99999933182473\\
300	9.9999993642045\\
400	9.99999920618785\\
500	9.99999884891009\\
600	9.99999860403341\\
700	9.99999805209628\\
800	9.99999728401995\\
900	9.99999738050084\\
1000	9.99999731508774\\
2000	9.99999517823715\\
3000	9.99999632827121\\
4000	9.99999680864103\\
5000	9.99999763401565\\
6000	9.99999543477016\\
7000	9.99999634807419\\
8000	9.9999721106351\\
9000	9.99996259447905\\
9000	9.99999999492866\\
8000	9.99999998691243\\
7000	9.99999979762313\\
6000	9.99999971724793\\
5000	9.99999992532724\\
4000	9.99999991812799\\
3000	9.99999990969219\\
2000	9.99999999267733\\
1000	9.99999999998894\\
900	9.99999999999656\\
800	9.99999999999805\\
700	9.99999999999928\\
600	9.99999999999931\\
500	9.99999999999914\\
400	9.9999999999993\\
300	9.99999999999974\\
200	9.99999999999981\\
100	10.020796222607\\
90	10.0332283363589\\
80	10.0470912295084\\
70	10.0896262225957\\
60	10.1279608944131\\
50	10.2047180037194\\
40	10.3215826390706\\
30	13.3008137624394\\
20	20.1588934476533\\
10	35.9545659296004\\
9	37.8018740545273\\
8	41.5077842536438\\
7	45.4033900498261\\
6	50.9735615277747\\
5	54.7345034400618\\
4	61.2250648726374\\
3	70.6520034208808\\
2	84.9041241759474\\
1	97.2182321130333\\
}--cycle;

  \addplot [color=mycolor2, solid, mark=star, mark size=3, line width=1.5pt]
  table[row sep=crcr]{%
1	82.8302920074995\\
2	67.1542776172087\\
3	56.5591469514311\\
4	48.5706429806797\\
5	42.5678486857135\\
6	38.308406871455\\
7	34.7669609046978\\
8	31.8140914008158\\
9	29.3760188749272\\
10	27.0893824553476\\
20	14.4770025764217\\
30	10.9856218709164\\
40	10.1956350490707\\
50	10.0906260026176\\
60	10.0513319774215\\
70	10.0283322078859\\
80	10.0130450130437\\
90	10.0082409696695\\
100	10.0052594603244\\
200	10.0002824626232\\
300	9.99999983215939\\
400	9.99999977329369\\
500	9.99999969802797\\
600	9.99999962415732\\
700	9.99999953512618\\
800	9.999999379231\\
900	9.99999930992182\\
1000	9.99999924153627\\
2000	9.99999846830891\\
3000	9.99999834891167\\
4000	9.99999861047825\\
5000	9.99999904288912\\
6000	9.99999790475513\\
7000	9.9999983374825\\
8000	9.99998961663984\\
9000	9.99998883275005\\
};
\addlegendentry{$n=10$};

\addplot[area legend,solid,fill=mycolor3,opacity=0.3,draw=none,forget plot]
table[row sep=crcr] {%
x	y\\
1	76.8690534540419\\
2	61.2612795934801\\
3	50.9930519029856\\
4	44.5583041889373\\
5	40.163891548069\\
6	36.4857986940591\\
7	33.6440284341841\\
8	31.6655163940452\\
9	29.7684873627083\\
10	28.1639296310229\\
20	20.3378230160899\\
30	17.2846215664875\\
40	15.7635474192051\\
50	14.8747689952225\\
60	14.1411520510188\\
70	13.5920074124269\\
80	13.1748272025842\\
90	12.8225204061867\\
100	12.5220851922587\\
200	10.8822236062475\\
300	10.2511127501878\\
400	10.1162646347967\\
500	10.0637035797715\\
600	10.0396433237211\\
700	10.0278261468543\\
800	10.0199354477202\\
900	10.014192369165\\
1000	10.010821781857\\
2000	10.0006949963043\\
3000	9.9999999775057\\
4000	9.99999612510167\\
5000	9.99998855645996\\
6000	9.99998089620539\\
7000	9.99997456712099\\
8000	9.9999724240661\\
9000	9.99996296492708\\
9000	10.0001412524095\\
8000	10.0002506308495\\
7000	10.0004807326371\\
6000	10.0008515200666\\
5000	10.0013740348419\\
4000	10.0022627213429\\
3000	10.0047931942489\\
2000	10.0101963405176\\
1000	10.0342813545412\\
900	10.0430817609071\\
800	10.0526797919846\\
700	10.0709304721895\\
600	10.0993999086913\\
500	10.1489977289123\\
400	10.3225245608799\\
300	10.7782785020674\\
200	11.5257227556754\\
100	13.6856221007\\
90	14.0385219185869\\
80	14.6623953784036\\
70	15.3325282914934\\
60	16.2368809933644\\
50	17.4377677770686\\
40	19.2330499972663\\
30	22.536244061436\\
20	27.3823442529203\\
10	41.8611148900239\\
9	43.3392185911427\\
8	46.1128813801639\\
7	50.5634111891501\\
6	54.7656103411951\\
5	58.8625472008749\\
4	65.2751983205023\\
3	73.8591821763853\\
2	86.757401716288\\
1	97.8666275970295\\
}--cycle;

  \addplot [color=mycolor3, solid, mark=diamond ,mark size=3, line width=1.5pt]
  table[row sep=crcr]{%
1	85.9230952118403\\
2	71.4364888956889\\
3	61.7623730955701\\
4	54.2716002676724\\
5	48.4561345240029\\
6	44.4160542125093\\
7	41.1146392918633\\
8	38.3189568150483\\
9	36.1117850737004\\
10	34.1075537436364\\
20	23.6993274998819\\
30	19.7249814939413\\
40	17.4317685952384\\
50	16.1115558609873\\
60	15.1738162498104\\
70	14.438667027732\\
80	13.8509500432951\\
90	13.4102633680365\\
100	13.0741633929513\\
200	11.2208750948707\\
300	10.4842572895281\\
400	10.2012922787993\\
500	10.1081323763323\\
600	10.0679435853212\\
700	10.0485176454434\\
800	10.0358193329015\\
900	10.0281521173758\\
1000	10.0223718339653\\
2000	10.0051981810512\\
3000	10.0018838209429\\
4000	10.0008525584599\\
5000	10.0004302409477\\
6000	10.0002447192102\\
7000	10.0001284627817\\
8000	10.0000748358184\\
9000	10.0000455371626\\
};
\addlegendentry{$n=100$};

\addplot [color=red,loosely dotted,line width=2.5pt]
  table[row sep=crcr]{%
1	10\\
900000	 10\\
};
\addlegendentry{$\Jac$}

\end{axis}
\end{tikzpicture}%
}
	\caption{The objective performance $\JacnN$ is computed using \eqref{AC-LP-n,N} for Example~\ref{ex:fisheries}.  The red dotted line is the optimal value approximated by $n = 10^3$ and $N = 10^6$, which amounts to 0 as also reported in \cite{ref:Saldi-15}.}
	\label{fig:fishery:rand}
\end{figure}
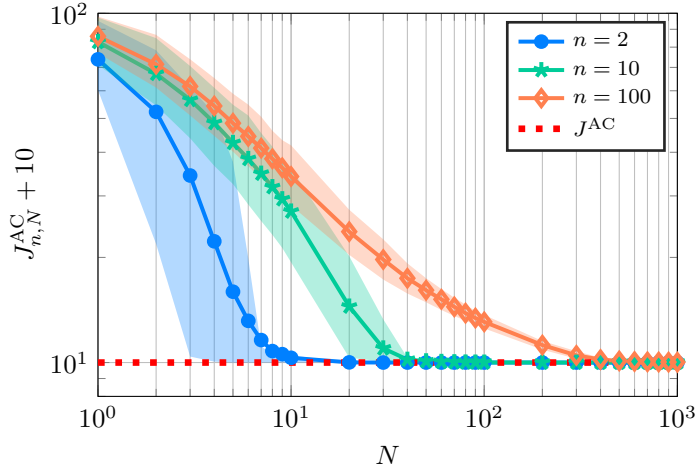		
We implement the methodology presented in Section~\ref{subsec:rand:MDP}, resulting in a finite random convex program \eqref{AC-LP-n,N}, where the uniform distribution on $K = S\times A = [\underline{\kappa},\overline{\kappa}]^2$ is used to draw the random samples. Figure~\ref{fig:fishery:rand} illustrates three cases with the number of basis functions $n \in \{2,10,100\}$ and the bound \eqref{theta*:MDP2}. The colored tubes represent the results between $[10\%,90\%]$ quantiles (shaded areas) as well as the means (solid lines) across $400$ independent experiments of the objective performance $\JacnN$ as a function of the sample size $N$. It is interesting to note that in this example the optimal solution is captured even with $2$ basis functions and only $N = 20$ random samples. This becomes even more attractive when we compare the results with a direct discretization scheme depicted in \cite[Figure~2]{ref:Saldi-15}. 

\subsubsection*{Structural convex optimization:}  
Similar to the LQG example in Section~\ref{ex:LQG}, we also implement the smoothing methodology for the case of $n = 10$. The simulation results are reported in Figure~\ref{fig:fishery:smoothing}. 

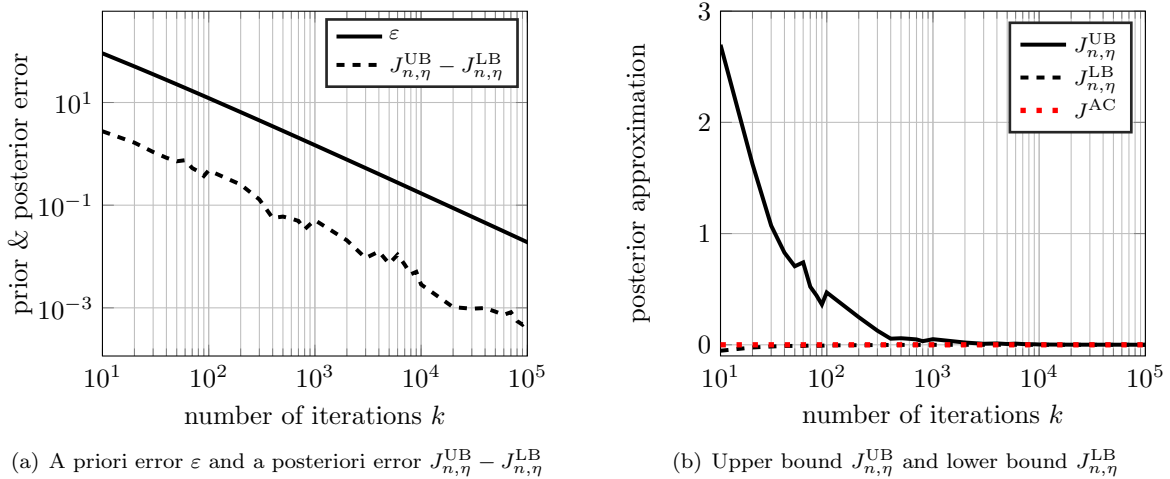
\begin{figure}[t]
	\subfigure[A priori error $\varepsilon$ and a posteriori error $\Jnub-\Jnlb$]{\scalebox{1}{
%
%
\begin{tikzpicture}

\begin{axis}[%
width=2.2in,
height=1.8in,
at={(1.011111in,0.641667in)},
scale only axis,
xmin=10,
xmax=100000,
xmode=log,
xlabel={number of iterations $k$},
xmajorgrids,
xminorgrids,
ymin=0,
ymax=600,
ymode=log,
ylabel={prior $\&$ posterior error},
ymajorgrids,
ylabel style={yshift=-0.15cm},
legend style={legend cell align=left,align=left,draw=white!15!black,line width=1.0pt,font=\footnotesize}
]
\addplot [color=black,solid,line width=1.5pt]
  table[row sep=crcr]{%
10	90.7596988873165\\
20	50.4063697325145\\
30	35.4576514031928\\
40	27.5475115641207\\
50	22.6165676646474\\
60	19.2341086084525\\
70	16.7628047887621\\
80	14.8744638585725\\
90	13.3823841724815\\
100	12.1723100224072\\
200	6.49491059977694\\
300	4.483003482218\\
400	3.44176046714649\\
500	2.80190712795399\\
600	2.36751050897174\\
700	2.05265889084282\\
800	1.81362652235531\\
900	1.62576619993161\\
1000	1.47410735437071\\
2000	0.772132793037242\\
3000	0.528015936747699\\
4000	0.402937593601818\\
5000	0.326590505080941\\
6000	0.275016154676782\\
7000	0.237781438079292\\
8000	0.209604013726553\\
9000	0.187518807967235\\
10000	0.169731115595792\\
20000	0.087978115751977\\
30000	0.0598358544320587\\
40000	0.0454973179960288\\
50000	0.0367786653603076\\
60000	0.0309059785953666\\
70000	0.0266758616911061\\
80000	0.0234807971207534\\
90000	0.0209805772841802\\
100000	0.0189696873296157\\
200000	0.00976713388785029\\
300000	0.00661923062763622\\
400000	0.00502106405626882\\
500000	0.00405167669644238\\
};
\addlegendentry{$\varepsilon$};

\addplot [color=black,dashed,line width=1.5pt]
  table[row sep=crcr]{%
10	2.75114402364085\\
20	1.64989484200388\\
30	1.08662273474772\\
40	0.838413529167986\\
50	0.716303058391313\\
60	0.751629176764794\\
70	0.527843195838203\\
80	0.453976846348352\\
90	0.370513448724604\\
100	0.475541564126903\\
200	0.252205934333526\\
300	0.129153469504749\\
400	0.056317239774328\\
500	0.0599573624198011\\
600	0.0546622918758253\\
700	0.0493018029438565\\
800	0.0338735832992433\\
900	0.0423541661739733\\
1000	0.0510683108240164\\
2000	0.020656372203495\\
3000	0.00933625317231349\\
4000	0.0124975949205784\\
5000	0.00724609981313697\\
6000	0.0108189115952417\\
7000	0.00695926150953026\\
8000	0.00451336505017449\\
9000	0.00501445742333247\\
10000	0.00285021265837832\\
20000	0.00103120305498905\\
30000	0.000962708277496393\\
40000	0.00099187467494551\\
50000	0.000843458013181959\\
60000	0.000731958241294043\\
70000	0.000820653367694514\\
80000	0.000552985426491319\\
90000	0.000470169711212189\\
100000	0.000549664006081856\\
200000	0.000161542871196692\\
300000	0.000157478416186615\\
400000	0.000114286160560574\\
500000	9.27186052237982e-05\\
};
\addlegendentry{$\Jnub-\Jnlb$};

\end{axis}
\end{tikzpicture}%
}}
	\qquad
	\subfigure[Upper bound $\Jnub$ and lower bound $\Jnlb$]{\scalebox{1}{
%
%
\begin{tikzpicture}

\begin{axis}[%
width=2.2in,
height=1.8in,
at={(1.011111in,0.641667in)},
scale only axis,
xmin=10,
xmax=100000,
xmode=log,
xlabel={number of iterations $k$},
xmajorgrids,
xminorgrids,
ymin=-0.1,
ymax=3,
ylabel={posterior approximation},
ymajorgrids,
ylabel style={yshift=-0.15cm},
legend style={legend cell align=left,align=left,legend pos=north east,draw=white!15!black,line width=1.0pt,font=\footnotesize}
]
\addplot [color=black,solid,line width=1.5pt]
  table[row sep=crcr]{%
10	2.69782240051088\\
20	1.62579317933833\\
30	1.07115935272095\\
40	0.828171926336044\\
50	0.70607517282797\\
60	0.740878175245887\\
70	0.522115480764568\\
80	0.444594005149921\\
90	0.362379183009112\\
100	0.470639892582908\\
200	0.248729379903276\\
300	0.127021159055569\\
400	0.0549353014708736\\
500	0.0590281281029852\\
600	0.0536519263944911\\
700	0.0483419461767818\\
800	0.0325118716723125\\
900	0.0416840173242761\\
1000	0.05061899782999\\
2000	0.0204534961184887\\
3000	0.00915265017565384\\
4000	0.0123235985696712\\
5000	0.00716959309992658\\
6000	0.010693155730566\\
7000	0.00685894282273861\\
8000	0.00442100487590579\\
9000	0.004930653756291\\
10000	0.00279382182970044\\
20000	0.00101572126869705\\
30000	0.000955529434186856\\
40000	0.000982917919282597\\
50000	0.000836282492199456\\
60000	0.000721253912624484\\
70000	0.000815445567292819\\
80000	0.000543737656212934\\
90000	0.000465564036385646\\
100000	0.000543300400007567\\
200000	0.000157944041236142\\
300000	0.000154977794771278\\
400000	0.000112678016109268\\
500000	9.1779790071263e-05\\
};

\addlegendentry{$\Jnub$};

\addplot [color=black,dashed,line width=1.5pt]
  table[row sep=crcr]{%
 10	-0.0533216231299718\\
20	-0.0241016626655506\\
30	-0.0154633820267671\\
40	-0.0102416028319417\\
50	-0.0102278855633426\\
60	-0.0107510015189079\\
70	-0.0057277150736354\\
80	-0.00938284119843086\\
90	-0.0081342657154919\\
100	-0.00490167154399528\\
200	-0.00347655443025035\\
300	-0.00213231044917997\\
400	-0.00138193830345437\\
500	-0.000929234316815946\\
600	-0.00101036548133423\\
700	-0.000959856767074679\\
800	-0.00136171162693085\\
900	-0.000670148849697179\\
1000	-0.000449312994026358\\
2000	-0.000202876085006332\\
3000	-0.000183602996659646\\
4000	-0.00017399635090721\\
5000	-7.65067132103906e-05\\
6000	-0.000125755864675721\\
7000	-0.000100318686791647\\
8000	-9.23601742686991e-05\\
9000	-8.38036670414648e-05\\
10000	-5.63908286778793e-05\\
20000	-1.54817862920044e-05\\
30000	-7.17884330953662e-06\\
40000	-8.95675566291295e-06\\
50000	-7.1755209825029e-06\\
60000	-1.07043286695595e-05\\
70000	-5.20780040169506e-06\\
80000	-9.24777027838479e-06\\
90000	-4.60567482654289e-06\\
100000	-6.36360607428934e-06\\
200000	-3.59882996055008e-06\\
300000	-2.50062141533679e-06\\
400000	-1.60814445130632e-06\\
500000	-9.388151525352e-07\\
};
\addlegendentry{$\Jnlb$};

\addplot [color=red,loosely dotted,line width=2pt]
  table[row sep=crcr]{%
10	0\\
100000	0\\
};
\addlegendentry{$\Jac$}

\end{axis}
\end{tikzpicture}%
}}
	\caption{The results and error bounds are obtained by Algorithm~\ref{alg} with $n=10$ for Example~\ref{ex:fisheries}. The red dotted line is the optimal solution computed as indicated in Figure~\ref{fig:fishery:rand}.}  
	\label{fig:fishery:smoothing}
\end{figure}

\appendix
\section{Infinite-Horizon Discounted-Cost Problems} \label{app:discouted:cost}
In the Markov decision process setting, introduced in Section~\ref{subsec:AC-setting}, let us consider \emph{long-run $\tau$-discounted cost} (DC) problems with the discount factor $\tau\in(0,1)$ and initial distribution $\nu\in\mathcal{P}(X)$ described as
	\begin{equation} \label{eq:discounted:problem}
	\Jdc(\nu)\Let \inf_{\pi\in\Pi} \lim_{n\to\infty} \Expecpi{\sum_{t=0}^{n-1}\tau^t \cost(x_{t},a_{t})}.
	\end{equation}
		
As in the average cost setting, in Section~\ref{sec:problem:statement}, we assume that the control model satisfies Assumption~\ref{a:CM}. We refer to \cite[Chapter~4]{ref:Hernandez-96} and \cite[Chapter~8]{ref:Hernandez-99} for a detailed exposition and required technical assumptions in more general settings. As for the AC problems, it is well known that the DC problem~\eqref{eq:discounted:problem} can be alternatively characterized by means of infinite LPs \eqref{primal-inf} and \eqref{dual-inf} introduced in Section~\ref{subsec:dual-pair}, where
	\begin{align}
	\label{AC-setting:discounted}
	\left\{
	\begin{array}{l}
		(\X,\C) \Let (\Cont(S),\Meas(S)) \\
		(\B,\Y) \Let (\Cont(K),  \Meas(K)) \\
		\cone \Let \Cont_{+}(K)\\
		\cone^{*} \Let \Meas_{+}(K) \\
		c(B) = -\nu(B), \quad B\in\borel(S)\\
		b(s,a) = -\cost(s,a) \\
		\op : \X \ra \B, \quad \op x(s,a) \Let - x(s) + \tau Qx(s,a) \\
		\op^{*}: \Y \to \C, \quad \op^{*}y (B) \Let y(B\times A) - \tau y Q(B), \quad B\in \borel(S),
		\end{array}\right.
	\end{align}
	
\begin{Thm}[LP characterization {\cite[Theorem~6.3.8]{ref:Hernandez-96}}] \label{thm:equivalent:LP:discounted}
	Under Assumption~\ref{a:CM}, the optimal value $\Jdc$ of the DC problem in \eqref{eq:discounted:problem} can be characterized by the LP problem \eqref{primal-inf} in the setting \eqref{AC-setting:discounted}, in the sense that $\Jp = -\Jdc$.
\end{Thm}

It is known that under similar conditions as in Assumption~\ref{a:CM} on the control model, the value function $u^{\star}$ in the $\tau$-discounted cost optimality equation is Lipschitz continuous \cite[Section~2.6]{ref:HernandezLerma-89} or \cite[Theorem~3.1]{ref:Dufour-13}. We use the norms similar to the AC-setting \eqref{AC:pairs}. The next step toward studying the approximation error \eqref{inf-semi error} for the DC-setting readily follows by Theorem~\ref{thm:inf-semi} combined with the following lemma.
		
		\begin{Lem}[DC semi-infinite regularity]\label{lem:MDP:DC}
			For the DC-problem \eqref{eq:discounted:problem}, characterized by the dual-pair vector spaces in \eqref{AC-setting:discounted}, under Assumption~\ref{a:CM} we have the operator norm $\|\op\| \leq 1+ \max\{L_Q,1\}\tau$, the inf-sup constant of Assumption~\ref{a:reg}\ref{a:reg:inf-sup} $\gamma = 1-\tau$, and the dual optimizer norm
			\begin{align}
			\label{yb:DC}
			\|y\opt\|_\wass \le \ynb = \frac{\xnb+ (1-\tau)^{-1}\|\cost\|_{\infty}}{(1-\tau) \xnb - \|\cost\|_{\lip}} \,.
			\end{align}
		\end{Lem}
		\begin{proof}
			With the norms considered and following a proof similar to Lemma~\ref{lem:operator}, the operator norm $\|\op\|$ can be upper bounded as
			$\| \op \| \leq 1+\tau.$ The \emph{inf-sup} condition, Assumption~\ref{a:reg}\ref{a:reg:inf-sup}, holds with $\gamma = 1-\tau$, since
			\begin{align*}
			\inf_{y \in \cone^*}\sup_{x \in \X_n} {\inner{\op x}{y} \over \|x\| \|y\|_{\wass}} \geq \inf_{y \in \cone^*} \frac{(1-\tau)\inner{\ind}{y}}{\| y \|_{\wass}} = 1-\tau.
			\end{align*}
			Moreover $\|\nu\|_{\wass} = 1$ since it is a probability measure. Thus, given the lower bound for the optimal value $\Jdcn \ge {-(1-\tau)^{-1}\|\cost\|_{\infty}}$, the assertion of Proposition~\ref{prop:SD} (i.e, the dual optimizers bound in \eqref{yb}) leads to the desired assertion \eqref{yb:DC}. 
		\end{proof}
		
		Note that when the norm constraint is neglected, the dual program enforces that any solution $\yn$ in the program \ref{dual-n} satisfies $\inner{x}{\op^{*}\yn -c}=0$ for all $x\in \X_{n}$ (cf. the program~\ref{dual-inf}). Assume that a constant function belongs to the set $\X_{n}$. Then, the constraint evaluated at the constant function reduces to $(1-\tau)\inner{\ind}{\yn} = (1-\tau)\|\yn\|_\wass = 1$. It is worth noting that this observation can consistently be captured by Lemma~\ref{lem:MDP:DC} when $\xnb$ tends to $\infty$, in which the bound \eqref{yb:DC} reduces to $\|\yn\|_{\wass} \le ({1-\tau})^{-1}$.

\bibliographystyle{siam}
\bibliography{ref,ref_Pey}

\end{document}